\documentclass[dvipsnames]{amsart}
\usepackage[utf8]{inputenc}
\usepackage{mathtools} %added by PL, to use \mathclap; it loads amsmath
\usepackage{amssymb}
\usepackage{amsthm}
\usepackage[style=alphabetic,maxbibnames=8,safeinputenc,backend=biber]{biblatex}

\usepackage{mathrsfs}
\usepackage{enumerate}
\usepackage[labelfont=bf]{caption}
\usepackage{accents}
\usepackage{microtype}

\usepackage[hyperfootnotes=false]{hyperref}
\usepackage{graphicx}

\usepackage{todonotes}
\usepackage{fancyhdr}

\usepackage{comment}
\usepackage{tikz-cd}%not really used, I put in comments the diagrams I drew
\usepackage{xcolor}
\colorlet{P.-L.color}{OliveGreen}

\colorlet{FZcolor}{DarkOrchid}
\newcommand{\real}{\mathbb{R}}
\newcommand{\ints}{\mathbb{Z}}
\newcommand{\rats}{\mathbb{Q}}
\newcommand{\nats}{\mathbb{N}}

\newcommand{\proj}{\mathbf{P}}

\newcommand{\HH}{\mathbb{H}}
\newcommand{\Isom}{\mathrm{Isom}}
\DeclareMathOperator{\PGL}{PGL}

\DeclareMathOperator{\GL}{GL}
\DeclareMathOperator{\SL}{SL}
\DeclareMathOperator{\SO}{SO}
\DeclareMathOperator{\PSO}{PSO}
\DeclareMathOperator{\PO}{PO}
\DeclareMathOperator{\orth}{O}

\newcommand{\eps}{\epsilon}
\newcommand{\del}{\partial}

\DeclareMathOperator{\id}{id}

\DeclareMathOperator{\Stab}{Stab}
\newcommand{\actson}{\curvearrowright}

\DeclareMathOperator{\axis}{axis}
\DeclareMathOperator{\Aut}{Aut}

\DeclareMathOperator{\Geod}{Geod}
\DeclareMathOperator{\Hopf}{Hopf}

\newcommand{\grop}[3]{\langle #2, #3 \rangle_{#1}}
\newcommand{\dirac}{\mathcal{D}}
\newcommand{\horoball}{\mathcal{H}}
\newcommand{\horosphere}{\del\horoball}
\newcommand{\hor}{\mathtt h}
\newcommand{\orb}{\mathrm{orb}}
\newcommand{\prox}{\mathrm{prox}}
\newcommand{\con}{\mathrm{con}}
\newcommand{\sse}{\mathrm{sse}}
\newcommand{\bip}{\mathrm{bip}}
\newcommand{\spl}{\mathrm{spl}}
\newcommand{\cone}{\mathcal{C}}
\newcommand{\shadow}{\mathcal{O}}
\newcommand{\limshade}{\mathcal{L}}
\newcommand{\st}{\:|\:}
\newcommand{\flip}{\iota}
\newcommand{\Epsy}{\mathcal{E}}
\newcommand{\Em}{\mathcal{M}}
\DeclareMathOperator{\Leb}{Leb}
\DeclareMathOperator{\supp}{supp}

\newcommand{\Rnplusone}{\real^{\mathrm n+1}}
\newcommand{\RPn}{\proj(\Rnplusone)}
\newcommand{\dimension}{\mathrm{n}}
\DeclareMathOperator{\Span}{span}
\newcommand{\rkone}{\mathrm{r1}}
\newcommand{\prkone}{\mathrm{pr1}}

\newcommand{\Acal}{\mathcal{A}}
\newcommand{\Bcal}{\mathcal{B}}
\newcommand{\Ccal}{\mathcal{C}}

\DeclareMathOperator{\interior}{int}

\newcommand{\corfix}{F}
\newcommand{\nbconj}{N}
\newcommand{\Gdelta}{$\mathrm{G}_\delta$}

\newcommand{\diff}{d}

\newcommand{\widesim}[1][1.5]{\mathrel{\scalebox{#1}[1]{$\sim$}}}

\newcommand{\ie}{i.e.\ }
\newcommand{\resp}{resp.\ }
\theoremstyle{plain}
\newtheorem{thm}{Theorem}[section]
\newtheorem{lem}[thm]{Lemma}
\newtheorem{prop}[thm]{Proposition}
\newtheorem{corn}[thm]{Corollary}
\newtheorem*{thm*}{Theorem}
\newtheorem{cor}[thm]{Corollary}

\newtheorem{introthm}{Theorem}

\newtheorem{introcorn}[introthm]{Corollary}

\theoremstyle{definition}
\newtheorem{defn}[thm]{Definition}

\newtheorem{question}[thm]{Question}
\newtheorem{notation}[thm]{Notation}
\newtheorem{eg}[thm]{Example}
\newtheorem{rmk}[thm]{Remark}

\newtheorem{obs}[thm]{Observation}

\numberwithin{equation}{section}

\title{Ergodicity and equidistribution in Hilbert geometry}
\author{Pierre-Louis Blayac}
\author{Feng Zhu}

\begin{document}

\begin{abstract}
    We show that dynamical and counting results characteristic of negatively curved Riemannian geometry, or more generally CAT($-1$) or rank-one CAT(0) spaces, also hold for rank-one properly convex projective manifolds or orbifolds, equipped with their Hilbert metrics, admitting finite Sullivan measures built from appropriate conformal densities. In particular, this includes geometrically finite convex projective manifolds or orbifolds whose universal covers are strictly convex with $C^1$ boundary.
    %quotients of strictly convex domains with $C^1$ boundary.
    
    More specifically, with respect to the Sullivan measure, the Hilbert geodesic flow is strongly mixing, and orbits and primitive closed geodesics equidistribute, allowing us to asymptotically enumerate these objects. 
\end{abstract}

\maketitle

In his influential thesis \cite{Margulis}, Margulis established counting results for negatively curved closed manifolds, by means of ergodicity and equidistribution results for the geodesic flows on these manifolds with respect to a suitable measure on the unit tangent bundle, called the Bowen--Margulis measure. 

In \cite{Roblin}, Roblin extended Margulis' results to the setting of quotients, not necessarily compact, of CAT($-1$) spaces $X$ by discrete subgroups of isometries $\Gamma \leq \Isom(X)$. More precisely, Roblin's results include ergodicity of the horospherical foliations, mixing of the geodesic flow, equidistribution of group orbits, equidistribution of primitive closed geodesics, and, in the geometrically finite case, asymptotic counting estimates. 

One key set of tools that allowed Roblin to work in this generality are Patterson--Sullivan densities, first developed in the context of real hyperbolic spaces by Patterson and Sullivan \cite{Patterson,Sullivan}. These densities are families of measures on the boundary at infinity. They are especially well-adapted to the geometry of the $\Gamma$-orbits in $X$, and they can be used to define useful flow-invariant measures on the unit tangent bundle, referred to here as Sullivan measures, which include the Bowen--Margulis measure. Patterson--Sullivan theory also has other applications than to equidistribution and counting problems, e.g.\ it can be used to compute Hausdorff dimensions of limit sets; for more on this topic, see the recent historical notes in \cite[p.\,2]{DeyKapovich}.

The techniques used by Margulis, Patterson, Sullivan, Roblin and others were later adapted to even more general settings. For instance, Link \cite{Link} recently used these to prove equidistribution results in rank-one quotients of Hadamard spaces.

Here we consider manifolds (and orbifolds) endowed with properly convex projective structures and associated Hilbert metrics (below, sometimes referred to as ``Hilbert geometries''). Properly convex Hilbert geometries, even strictly convex ones, are in general not CAT$(-1)$ or even CAT$(0)$ (see e.g.\ \cite[App.\,B]{Egloff}). 
Nevertheless, those of them which satisfy a rank-one condition, analogous to the rank-one condition in Riemannian or CAT(0) geometry, do exhibit substantial similarities to negatively curved Riemannian geometries.
%, but exhibit substantial similarities to negatively curved Riemannian geometries.
In particular, 
there is a good theory of Patterson--Sullivan measures on these geometries, see \cite{crampon_these,zhu_conf_densities,Bray,Blayac_PSdensities}.
They also come with geodesic flows whose dynamics exhibit features of negative curvature; the study of the dynamics of these flows was initiated by Benoist \cite{Benoist_CDI}, with further contributions from Crampon--Marquis \cite{CM14} and Bray \cite{Bray}.

In \cite{Blayac_topmixing} and \cite{Blayac_PSdensities}, the first author obtained good dynamical properties for the Hilbert geodesic flow in the setting of rank-one properly convex Hilbert geometries, including topological mixing and strong mixing of the geodesic flow with respect to a Sullivan measure (constructed via a Patterson--Sullivan density).

In this paper, we show that we can use this mixing to obtain some of Roblin's equidistribution results in the setting of rank-one properly convex Hilbert geometries. These further generalize analogous results obtained by the second author for geometrically finite strictly convex Hilbert geometries in \cite{zhu_conf_densities}.

\subsection{Main results}

Our results will apply to properly convex domains $\Omega$ of the real projective space $\RPn = \real\proj^{\dimension}$ and rank-one subgroups $\Gamma \leq \Aut(\Omega)$, where $\Aut(\Omega) \leq \PGL(\Rnplusone)$ is the subgroup of invertible projective transformations preserving $\Omega$. 

A properly convex domain of $\RPn$ is an open set which is contained in some affine chart of $\RPn$ and which is bounded and convex in that affine chart in the usual Euclidean sense. {\bf Throughout this paper, we will write ``domain'' to denote a properly convex domain.}

A domain $\Omega$ comes with a 
% projectively invariant 
Finsler metric $d_\Omega$, called the Hilbert metric (see \S\ref{sub:hilbgeom} for the full definition). The intersection of any projective line of $\RPn$ with $\Omega$, if non-empty, is a $d_\Omega$-geodesic, called a straight-line geodesic.
%The projective straight lines in $\Omega$ are geodesics for $d_\Omega$, called straight-line geodesics,
This allows us to define a Hilbert geodesic flow $(g^t)_{t\in\real}$ on the unit tangent bundle $S\Omega$ (it parametrises the straight-line geodesics). The subgroup $\Aut(\Omega)\leq\PGL(\Rnplusone)$ acts isometrically with respect to $d_\Omega$, and commutes with the action of $(g^t)$. Thus, given a discrete subgroup $\Gamma\leq\Aut(\Omega)$, the flow $(g^t)$ descends to a geodesic flow $(g^t_\Gamma)$ on the unit tangent bundle $S\Omega/\Gamma$ of the convex projective manifold or orbifold $\Omega/\Gamma$.

Benoist \cite{Benoist_CDI} studied the dynamics of $(g^t_\Gamma)$ on $S\Omega/\Gamma$ in the case where $\Gamma$ divides (\ie acts cocompactly on) $\Omega$ and $\Omega$ is strictly convex (\ie $\del\Omega$ contains no non-trivial line segments). Crampon--Marquis \cite{CM14} studied $(g^t_\Gamma)$ in cases where $\Omega$ is strictly convex and $\Omega/\Gamma$ is not necessarily compact, and Bray \cite{Bray} studied the case where $\Omega/\Gamma$ is compact three-dimensional, and $\Omega$ is not necessarily strictly convex.

In this paper, we study the dynamics of $(g^t_\Gamma)$ when $\Gamma\leq \Aut(\Omega)$ is non-elementary rank-one, \ie $\Gamma$ contains at least two elements $\gamma_1, \gamma_2 \in \Aut(\Omega)$ each preserving a different axis --- a straight-line geodesic $\axis(\gamma_i)$ of $\Omega$ both of whose endpoints are $C^1$ and strongly extremal points in $\del\Omega$. We refer the reader to \S\ref{subsec:bdry pts} for the precise definitions of $C^1$ and strongly extremal, and \S\ref{sec:rk1} for a longer description of the rank-one condition, which is due to M.\ Islam \cite{Mitul_rk1} and A.\ Zimmer \cite{zimmer_higher_rank}.

We can build a Sullivan measure on $S\Omega/\Gamma$, associated to a Patterson--Sullivan density on the boundary $\del\Omega$ (i.e.\ a measure on $\del\Omega$ satisfying good properties; see \S\ref{subsec:pat_sul} for the details) and prove that the Hilbert geodesic flow $(g^t_\Gamma)$ on $S\Omega / \Gamma$ is mixing with respect to this measure when it is finite.

Mixing then gives us, via an argument of Babillot \cite{Babillot}, equidistribution of the unstable horospheres; the precise statement is a bit technical and we refer the interested reader to Theorem \ref{thm:horosphere_equidist} for this result. Moreover, we have equidistribution of group orbits and of primitive closed geodesics:
\begin{introthm}\label{introthm:A}
Let $\Omega$ be a domain and $\Gamma \leq \Aut(\Omega)$ a non-elementary rank-one discrete subgroup such that $S \Omega / \Gamma$ admits a finite Sullivan measure $m_\Gamma$ associated to a $\Gamma$-equivariant conformal density $(\mu_x)_{x\in\Omega}$ of dimension $\delta = \delta(\Gamma)$.
Then 
\begin{enumerate}[(i)]
\item \label{introthm:orbit_equidist} (Equidistribution of group orbits, Theorem \ref{thm:orbit_equidist})
for all $x, y \in \Omega$, 
\[ \delta \|m_\Gamma\| e^{-\delta t} \, \sum_{\mathclap{\substack{\gamma \in \Gamma\\ d_\Omega(x, \gamma y) \leq t}}} \, \dirac_{\gamma y} \otimes \dirac_{\gamma^{-1} x} \xrightarrow[t\to+\infty]{} \mu_x \otimes \mu_y \] 
in $C(\bar\Omega \times \bar\Omega)^*$, the weak*-dual to the space of continuous functions on $\bar\Omega \times \bar\Omega$;
\item \label{introthm:pcgeod_equidist} (Equidistribution of primitive closed geodesics, Theorem \ref{thm:pcgeod_equidist})
writing $\mathcal{G}^{r1}_\Gamma(\ell)$ to denote the set of primitive closed rank-one geodesics (\ie periodic $(g^t_\Gamma)$-orbits lifting to axes of rank-one elements) of length at most $\ell$, we have
\[ \delta \ell e^{-\delta \ell} \sum_{g \in \mathcal{G}^{r1}_\Gamma(\ell)} \dirac_g \xrightarrow[\ell\to+\infty]{} \frac{m_\Gamma}{\|m_\Gamma\|} \]
in $C_c(S\Omega / \Gamma)^*$, the weak*-dual to the space of compactly-supported continuous functions on $S\Omega / \Gamma$.
\end{enumerate}
\end{introthm}
Here $\dirac_x$ denotes the Dirac mass at $x$ with mass $1$, and $\dirac_g$ denotes the Lebesgue measure supported along $g \in \mathcal{G}_\Gamma(\ell)$ with mass $1$.  Moreover, $\delta(\Gamma)$ is the \emph{critical exponent} of $\Gamma$ with respect to the Hilbert metric $d_\Omega$, \ie for any $x\in\Omega$, $$\delta(\Gamma)=\limsup_{r\to\infty}\frac1r \log\# \{\gamma \in \Gamma \:|\: d_\Omega(x,\gamma x)\leq r\}.$$

Integrating a constant function against both sides of Theorem~\ref{introthm:A}.\eqref{introthm:orbit_equidist} yields the following counting result for orbit points: for any $x, y \in \Omega$,
\begin{equation}\label{eq:counting_orbit}
\#\{\gamma\in\Gamma\st d_\Omega(x,\gamma y)\leq t\}\underset{t\to\infty}{\widesim}\frac{\|\mu_x\|\cdot\|\mu_y\| \cdot e^{\delta t}}{\delta\|m_\Gamma\|}.
\end{equation}

In the case of a compact quotient, we may also integrate a constant function against both sides in Theorem~\ref{introthm:A}.\eqref{introthm:pcgeod_equidist} to obtain a counting result for primitive closed geodesics.

% A more general setting where one may obtain a counting result by integrating \eqref{introthm:pcgeod_equidist} along a well chosen function is when $\Gamma$ acts convex cocompactly on $\Omega$. This notion is due to Danciger--Gu\'eritaud--Kassel \cite{DGK}, and it means that the convex core of $\Omega/\Gamma$ is non-empty and compact (see \S~\ref{sec:geomfin} for more details). 
More generally, we may obtain a similar counting result when $\Gamma$ acts convex cocompactly on $\Omega$, i.e.\ the convex core of $\Omega/\Gamma$ is non-empty and compact. This notion is due to Danciger--Gu\'eritaud--Kassel \cite{DGK}; see \S\ref{sec:geomfin} for more details.
Closed geodesics are always contained in the convex core, and if $\Gamma$ is non-elementary rank-one and convex cocompact, then it can be proved \cite[\S7.1]{Blayac_PSdensities} that  $m_\Gamma$ is concentrated on vectors based in the convex core (and hence finite). Thus integrating a compactly-supported function which is constant on the convex core in \eqref{introthm:pcgeod_equidist} yields:

\begin{introcorn}\label{cor:counting_geod}
Let $\Omega$ be a domain and $\Gamma \leq \Aut(\Omega)$ a convex cocompact, non-elementary rank-one and discrete subgroup. Then
\begin{equation*}
 \#\mathcal G^{r1}_\Gamma(\ell) \underset{\ell\to\infty}{\widesim}\frac{e^{\delta\ell}}{\delta\ell}.   
\end{equation*}
\end{introcorn}

If $\Omega$ is strictly convex with $C^1$ boundary (in which case all points of $\partial\Omega$ are $C^1$ and strongly extremal), then Corollary \ref{cor:counting_geod} can be extended to the class of discrete subgroups $\Gamma \leq \Aut(\Omega)$ acting geometrically finitely on $\del\Omega$, in the sense of Crampon--Marquis \cite{CM12}; here the convex core of $\Omega/\Gamma$ is not necessarily compact. 
The notion of geometric finiteness for strictly convex domains $\Omega$ with $C^1$ boundary is analogous to the negatively curved Riemannian notion of geometric finiteness; for definitions and a fuller description we refer the reader to \S\ref{sub:geomfin}. 

For this class of subgroups, we actually prove a stronger equidistribution result for the closed geodesics.

\begin{introthm}\label{introthm:B}
Let $\Omega$ be a strictly convex domain with $C^1$ boundary and $\Gamma \leq \Aut(\Omega)$ a non-elementary discrete subgroup acting geometrically finitely on $\del\Omega$. Then 
\begin{enumerate}[(i)]
\setcounter{enumi}{2}
\item (Theorem \ref{thm:finite_BMmeas}) $S\Omega / \Gamma$ admits a finite Sullivan measure $m_\Gamma$, associated to a $\Gamma$-equivariant conformal density of dimension $\delta = \delta(\Gamma)$;
\item \label{introthm:pcgeod_equidist_geomfin} (Theorem \ref{thm:pcgeod_equidist_geomfin})
writing $\mathcal{G}_\Gamma(\ell)$ to denote the set of primitive closed geodesics of length at most $\ell$,
\[ \delta \ell e^{-\delta \ell} \sum_{g \in \mathcal{G}_\Gamma(\ell)} \dirac_g \xrightarrow[\ell \to +\infty]{} \frac{m_\Gamma}{\|m_\Gamma\|} \]
in $C_b(S\Omega / \Gamma)^*$, the dual to the space of bounded continuous functions on $S\Omega / \Gamma$.
\end{enumerate}
\end{introthm}

Theorem~\ref{introthm:B} was proved by the second author in \cite{zhu_conf_densities} under the assumption that the action of $\Gamma$ is geometrically finite {\bf on $\mathbf\Omega$}. This other notion of geometric finiteness is stronger than geometric finiteness on $\partial\Omega$, and is also due to Crampon--Marquis (see Definition~\ref{defn:limitset}).

Integrating a constant function against both sides of Theorem~\ref{introthm:B}.\eqref{introthm:pcgeod_equidist_geomfin} yields the following counting result (see Corollary~\ref{cor:countgeomfin}):
\[ \#\mathcal G_\Gamma(\ell) \underset{\ell\to\infty}{\widesim}\frac{e^{\delta\ell}}{\delta\ell}. \]

%As corollaries to these equidistribution results, we obtain precise asymptotics on the number of orbit points (see \S\ref{sec:orbit_equidist}) and primitive closed geodesics (see \S\ref{sec:pcgeod_equidist} and \S\ref{sub:geomfin_pcgeod_equidist}).

In the case where $\Omega$ is strictly convex and $\Omega / \Gamma$ compact, Theorem~\ref{introthm:A}.\eqref{introthm:pcgeod_equidist} (and Theorem~\ref{introthm:B}) follows from the work of Benoist \cite{Benoist_CDI} (see the end of \S\ref{sub:hilbgeom} below) together with general equidistribution and counting results for Anosov systems due to Margulis \cite{Margulis}. Weaker results of a similar flavor also follow from Bray's work \cite{Bray} if $\Omega/\Gamma$ is compact three-dimensional, from Islam's work \cite{Mitul_rk1} if $\Omega/\Gamma$ is compact, and from the first author's work \cite{Blayac_PSdensities} when $\Gamma$ is convex cocompact. 

The proofs of our equidistribution results follow the gist of Roblin's proofs,
making heavy use of mixing, and of cones in the space and shadows on the boundary without reference to any notion of angle, which is not well-defined in Roblin's setting nor in ours. 
As noted above,
a key tool is the theory of Patterson--Sullivan densities and associated Sullivan measures.
%Moreover, as observed by Sullivan, this theory can be applied to answer several types of questions. Here it is used to obtain counting and equidistribution results;  \cite{Roblin}

We remark that while Theorem \ref{introthm:B} relies integrally on the geometric finiteness condition, Theorem \ref{introthm:A} applies to a larger class of discrete subgroups $\Gamma \leq \Aut(\Omega)$ admitting finite Sullivan measures. Note that there exist geometrically infinite hyperbolic surfaces with finite Bowen--Margulis measure (see \cite{Peigne_contrex,ST_entropinf}), and that there also exist geometrically finite Riemannian surfaces with non-constant negative curvature whose Bowen--Margulis measure is infinite (see \cite[Th.\,C]{DOP}).

%  we will take a moment to point out further connections, involving Patterson--Sullivan theory and counting results, to the more general class of Anosov representations, defined by Labourie \cite{Labourie} and Guichard--Wienhard \cite{GW}, which may be viewed as a higher-rank analogue of convex cocompactness.

\subsection{Counting and Patterson--Sullivan theory in higher rank} \label{sec:counting}

As alluded to above, Roblin's results continue a long line of equidistribution and counting results in negative curvature. 
Here, before proceeding with our principal contents, we briefly survey related results in the setting of discrete word-hyperbolic or relatively hyperbolic subgroups of higher-rank semi-simple Lie groups. This includes the setting of higher Teichm\"uller theory, which studies certain spaces of discrete representations of surface groups.

%We remark that holonomies of strictly convex projective structures on closed hyperbolizable manifolds 
% (henceforth, ``Benoist subgroups\todo[color=P.-L.color!50!white,linecolor=black]{Harry Bray call ``Benoist manifolds'' the \emph{non-strictly convex} closed convex projective 3-manifolds...}'', in honor of Yves Benoist's work studying such convex projective structures) 
%are one class of examples of such subgroups; in particular, they satisfy the Anosov condition defined in \cite{Labourie} and \cite{GW}.

Let $G$ be a non-compact semi-simple real-algebraic group and $P < G$ a proper parabolic subgroup. One class of discrete word-hyperbolic subgroups of $G$ for which equidistribution and counting problems have been studied consists of those which satisfy the 
% Anosov
$P$-Anosov 
condition defined in \cite{Labourie} and \cite{GW}. 
% If $P$ is minimal parabolic then $P$-Anosov groups are called Borel-Anosov. 
% these include images of Hitchin representations as examples.
If $G=\PGL(\Rnplusone)$ and $P$ is the stabilizer of a point in $\RPn$, then $P$-Anosov groups are called projective Anosov; these include the images of holonomy representations of strictly convex projective structures on closed manifolds by work of Benoist \cite{Benoist_CDI}). 
% For $G=\PGL(\Rnplusone)$, recall that the Borel-Anosov condition implies the projective Anosov one.
Danciger--Gu\'eritaud--Kassel proved in \cite{DGK} that a discrete subgroup $\Gamma<\PGL(\Rnplusone)$ preserving some domain is projective Anosov if and only if it acts convex cocompactly on some strictly convex domain with $C^1$ boundary. A similar result was proved independently by Zimmer \cite{Zimmer} if $\Gamma$ is irreducible.

The equidistribution results of the present paper apply, further, to two disjoint classes of groups which are not Anosov, and indeed not word-hyperbolic. 

First, they apply to non-hyperbolic groups acting convex cocompactly on non-strictly convex domains. There exist many such examples, see e.g. \cite{Benoist_CDIII,LudoThese,BDL_cvxproj_3mfd,choi2016convex,DGK,DGKLM}; many of them are relatively hyperbolic.

Second, our results also apply to images of geometrically finite but not convex cocompact holonomies of strictly convex projective structures. These are relatively hyperbolic, and those which moreover satisfy the strong condition of geometric finiteness (\ie on the domain $\Omega$, see Definition~\ref{defn:limitset}) satisfy a relative version of the Anosov condition (see \cite{KL} or \cite{reldomreps}).
%Second, holonomies of geometrically finite strictly convex projective structures, in the (stronger) sense of \cite{CM12}, are not Anosov unless they are convex cocompact, but do in general satisfy a relative version of the Anosov condition (see \cite{KL} or \cite{reldomreps}).
Conjecturally, groups satisfying the weaker notion of geometric finiteness (\ie on the boundary $\partial\Omega$) also satisfy this relative version of the Anosov condition, and furthermore any relatively Anosov group preserving a properly convex domain may admit such a boundary geometrically finite action. This would be a relative version of the relationship between the projective Anosov condition and convex cocompactness in real projective geometry established in \cite{DGK} and \cite{Zimmer}, and would indicate that the results in this paper are applicable to a large class of relatively Anosov groups. 

In \cite{_cvx,Sambarino_hypcvx} (see also \cite[App.\,A]{LeonCgrowth}), Sambarino obtains equidistribution and counting results similar to the results presented here for irreducible projective Anosov groups, and for Zariski-dense $P_{\min}$-Anosov subgroups of any semi-simple group $G$, where $P_{\min}$ denotes the minimal parabolic (although he uses different terminology). A particular case of his counting results is the following
\begin{thm*}[{\cite[Th.\,C]{Sambarino_cvx}}]
Let $\Gamma < \PGL(\Rnplusone)$ be $P_{\min}$-Anosov, Zariski-dense and torsion-free.
Then there exists $h> 0$ such that
\begin{equation*}\label{eq:SamGeodCounting} \#\left\{ [\gamma] \in [\Gamma] \mbox{ primitive} \:|\: \frac12 \log\frac{\lambda_1}{\lambda_{\mathrm n+1}}(\gamma) \leq \ell\right\} \underset{\ell\to\infty}{\widesim} \frac{e^{h\ell}}{h\ell}, \end{equation*}
where $[\Gamma]$ is the set of conjugacy classes of $\Gamma$ and $\lambda_1(\gamma)$ (\resp $\lambda_{\mathrm n+1}(\gamma)$) is  the modulus of the largest (\resp smallest) eigenvalue of any lift of $\gamma$ in $\SL^\pm(\Rnplusone)$.
\end{thm*}

Note that $\frac12 \log\frac{\lambda_1}{\lambda_{\mathrm n+1}}(\gamma)$ is the translation of length of $\gamma$ in $(\Omega,d_\Omega)$.

If $\Omega\subset\proj(\Rnplusone)$ is a domain and $\Gamma\subset\Aut(\Omega)$ is a $P_{\min}$-Anosov torsion-free subgroup which is Zariski-dense in $\PGL(\Rnplusone)$, then the conclusion of Corollary~\ref{cor:counting_geod} is a consequence of Sambarino's result above.
($P_{\min}$-Anosov is stronger than projective Anosov, and implies here convex cocompactness.)
%This implies Corollary~\ref{cor:counting_geod} for $\Gamma$ Zariski-dense, Borel-Anosov and torsion-free.
Sambarino also obtains counting results for different notions of lengths than $\frac12 \log \frac{\lambda_1}{\lambda_{\mathrm n+1}} (\gamma)$ in any semi-simple group $G$. 

Moreover, Sambarino obtains orbit-counting results like \eqref{eq:counting_orbit} in the homogeneous space $\PGL(\Rnplusone)/M$ where $M$ consists of diagonal matrices with $\pm1$ diagonal entries, and in analogous homogeneous spaces for other semi-simple groups $G$.

Sambarino's results are also proven using Patterson--Sullivan theory, in his case in conjunction with the thermodynamical formalism.
These methods have also been extended to obtain $\Gamma$-orbit-counting results in other homogeneous spaces $G/H$, where $\Gamma\leq G$ is $P_{\min}$-Anosov, and $H\leq G$ ranges over a wider class of subgroups: \cite{LeonC,LeonCgrowth} studied the case when $G/H$ is a pseudo-Riemannian symmetric space, and \cite{EdwardsLeeOh} studied a more general case encompassing Sambarino's work.

\subsection*{Organization}
Section \ref{sec:hg} collects the necessary background material on Hilbert geometry and geometric finiteness in that setting.

Section \ref{sec:rk1} describes the rank-one condition and some of its consequences, including a topological mixing property for the geodesic flow of non-elementary rank-one convex projective manifold which generalises results of \cite{Blayac_topmixing} and is used to prove the mixing property of the Bowen--Margulis measure in \cite{Blayac_PSdensities}.

Section \ref{sec:ps} describes the construction of Patterson--Sullivan densities and Sullivan measures in our setting, and also a convex projective Hopf--Tsuji--Sullivan--Roblin dichotomy, which implies that, when finite, the Sullivan measure gives full measure to recurrent rank-one vectors.
% ; this is used in Sections \ref{sec:orbit_equidist} and \ref{sec:pcgeod_equidist}.

Section \ref{sec:mixing} states the mixing result (Theorem \ref{thm:mixing}) which is a crucial ingredient in the following sections, and also contains the proof of the equidistribution of the unstable horospheres (Theorem \ref{thm:horosphere_equidist}).

Sections \ref{sec:orbit_equidist} and \ref{sec:pcgeod_equidist} contain the proofs of Theorem~\ref{introthm:A}.(i) and (ii) respectively. 

In Section~\ref{sec:perconj} (see also Section~\ref{sec:nb per conj geomfin}) we use results from the previous sections to prove counting results for rank-one conjugacy classes in the fundamental group $\Gamma$ (which are not in bijective correspondence with primitive closed geodesics if e.g.\ $\Gamma$ has torsion). 

Finally, Section \ref{sec:geomfin} contains the proofs of results specific to geometrically finite subgroups acting on strictly convex domains with $C^1$ boundary, as described in Theorem~\ref{introthm:B}.

\subsection*{Acknowledgements} 
The authors thank Harrison Bray, Fanny Kassel, Ludovic Marquis, and Ralf Spatzier for helpful and encouraging discussions, Yves Benoist for supplying a citation, and Dick Canary for helpful comments on early drafts.

FZ was partially supported by U.S. National Science Foundation (NSF) grant FRG 1564362 and ISF grant 18/171, and acknowledges support from NSF grants DMS 1107452, 1107263, 1107367 ``RNMS: Geometric Structures and Representation Varieties'' (the GEAR Network).

This project has also received funding from the European Research Council (ERC) under the European Union’s Horizon 2020 research and innovation programme (ERC starting grant DiGGeS, grant agreement No. 715982).

\section{Hilbert geometry}\label{sec:hg}
\subsection{Properly convex Hilbert geometries} \label{sub:hilbgeom}

As in the introduction, a properly convex domain (hereafter, ``domain'') $\Omega \subset \RPn = \real\proj^{\mathrm n}$ is an open set contained in some affine chart and bounded and convex in that affine chart, in the usual Euclidean sense.
The Hilbert metric $d_\Omega$ on a domain $\Omega$ is defined as follows: given $x, y \in \Omega$, extend the straight line between them so that it meets $\del \Omega$ in $a$ and $b$ (with $x$ between $a,y$). 
% see Figure~\ref{fig:hilbmetr}. 
Then $d_\Omega(x,y) = \frac 12 \log \frac{|ay||bx|}{|ax||by|}$, where $|\cdot|$ denotes Euclidean distance in the affine chart.

\begin{figure}[ht!]
    \centering
    \includegraphics[width=.32\textwidth]{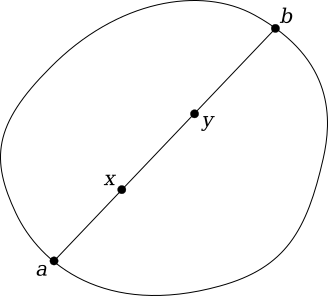}
%     \caption{} 
\label{fig:hilbmetr}
\end{figure}

This can be shown to be a  projectively invariant metric: $d_{g\Omega}(gx,gy)=d_\Omega(x,y)$ for all $g\in\PGL(\Rnplusone)$ and $x,y\in\Omega$. In particular, $d_\Omega$ is well-defined independent of the choice of affine chart. Moreover, $d_\Omega$ is a Finsler metric, i.e. it is induced by an infinitesimal norm on the tangent bundle to $\Omega$.

% with Finsler norm at $x \in \Omega$ given by
% \begin{equation}
% \frac12 \left( \frac1{|xv^+|} + \frac1{|xv^-|} \right) |dx| 
% \notag % \label{eqn:finsler_norm} 
% \end{equation} 
% where $v^\pm$ denote the forward / backward endpoints (respectively) of the geodesic through $x$ tangent to the vector $v$ whose norm we are measuring.

Straight (projective) lines are always geodesics for $d_\Omega$. When $\Omega$ is {\bf strictly convex}, i.e.\ $\del\Omega$ contains no non-trivial line segments, these are the unique geodesics for this metric.

For $\Omega$ a strictly convex domain equipped with its Hilbert metric $d_\Omega$, the Hilbert geodesic flow $(g^t)_{t\in\real}$ on the unit tangent bundle $S \Omega$ is the unit (Hilbert) speed flow along the geodesics. We may similarly define the Hilbert geodesic flow $(g^t)_{t\in\real}$ on the unit tangent bundle $S \Omega$ when $\Omega$ is properly convex, not necessarily strictly convex, as the unit speed flow along the straight-line Hilbert geodesics.

We remark that there is a natural involution $\flip$ on $S\Omega$ such that $g^t \flip v=\flip g^{-t}v$ for all $v\in S\Omega$ and $t\in\real$ and $\flip$ does not change the foot-point of $v$.

When $\Omega$ is an ellipsoid, equipping $\Omega$ with its Hilbert metric gives us the projective model of real hyperbolic space. 
More generally, the geometry of strictly convex domains $\Omega \subset \RPn$ equipped with the Hilbert metric shares many features with negatively curved Riemannian geometries, even beyond what may be expected given $\delta$-hyperbolicity, although they are not in general Riemannian or even CAT($-1$). 
For instance, nearest-point projection from a point $z$ to a Hilbert geodesic $l \subset \Omega$ is well-defined on the nose, not just coarsely as in the case of general $\delta$-hyperbolic space. This follows from the strict convexity of metric balls \cite[(18.6)]{BusemannGeomGeod} and hence of the distance functions $d(z,\cdot): l \to \real$ for $z\in\Omega$.

On the other hand, we can start with $\Omega$ a simplex, and the resulting Hilbert geometry is isometric to $\real^{\mathrm n}$ endowed with a norm equivalent (though not equal) to the Euclidean norm. More generally, domains $\Omega$ which are not strictly convex, equipped with their Hilbert metrics, have geometric and dynamical features qualitatively similar to those of non-positively curved Riemannian geometries. The metric balls in this case remain convex \cite[(18.6)]{BusemannGeomGeod}, as do the distance functions $d(z,\cdot): l \to \real$ for any fixed $z\in\Omega$. While it is no longer in general true that the functions $t\mapsto d_\Omega(l_1(t),l_2(t))$ are convex for any two straight-line Hilbert geodesics $l_1$ and $l_2$, Crampon observed the following useful property; a complete proof can be found in  \cite[App.\,A]{Blayac_topmixing}.

\begin{lem}[{\cite[Lem.\,8.3]{Crampon09}}]\label{lem:crampon}
Let $\Omega \subset \RPn$ be a domain, and let $l_1$ and $l_2$ be two straight-line Hilbert geodesics of $\Omega$ parametrized with constant speed (not necessarily the same speed). Then for all $0\leq t\leq T$,
\[d_{\Omega}(l_1(t),l_2(t))\leq d_\Omega(l_1(0),l_2(0))+d_\Omega(l_1(T),l_2(T)).\]
\end{lem}

Given a properly convex domain $\Omega$, we write $\Aut(\Omega)$ to denote the group of invertible projective automorphisms preserving $\Omega$, \ie 
\[ \Aut(\Omega) := \{\gamma \in \PGL(\Rnplusone) = \Aut(\RPn) \st \gamma(\Omega) \subset \Omega \} .\]
Projective automorphisms $\gamma \in \Aut(\Omega)$ are isometries of $\Omega$ equipped with the Hilbert metric $d_\Omega$; in fact, 
% in the strictly convex case, 
$\Aut(\Omega)$ coincides with the isometry group of $(\Omega, d_\Omega)$, except in the case of certain properly convex domains coming from symmetric spaces where $\Aut(\Omega)$ has index two inside the isometry group: see e.g. \cite[Prop.\,10.2]{Marquis} for the strictly convex case, or \cite{Walsh_Gauge} for the more general case.

%The isometries\todo{define $\ell(\gamma)$ and state CLT result} of $(\Omega, d_\Omega)$ may be classified as hyperbolic, parabolic, or elliptic: the classification can be done in terms of translation length, or by looking at properties of matrices in $\SL(n+1,\real)$ considered as an isomorphic image / lift of $\PGL(n+1,\real)$. (For further details, see e.g. \cite[\S3]{CM12} for the $C^1$ case, or \cite[\S2]{CLT} in the more general properly convex case.) As in the hyperbolic case, closed geodesics on a quotient $\Omega / \Gamma$ lift to axes of hyperbolic isometries on $\Omega$.

The Hilbert metric $d_\Omega$ is proper (closed balls are compact), and this implies that any discrete subgroup $\Gamma \leq \Aut(\Omega)$ acts properly discontinuously on $\Omega$, with quotient $\Omega / \Gamma$ an orbifold (for $\Gamma$ torsion-free, a manifold) equipped with a {\it convex projective structure}, i.e. an atlas of charts to $\RPn$ which locally give the orbifold the geometry of projective space. 
The Hilbert metric $d_\Omega$ descends to a metric on the quotient $\Omega / \Gamma$, and the Hilbert geodesic flow $(g^t)$, since it commutes with the action of $\Gamma$, descends to a flow $(g_\Gamma^t)$ on the quotient $S\Omega / \Gamma $. 

As for real hyperbolic manifolds, there is a partial correspondence between periodic orbits of $(g^t_\Gamma)$ on $S\Omega/\Gamma$ and conjugacy classes of $\Gamma$, and the length of the orbit corresponding to (the conjugacy class of) $\gamma \in \PGL(\Rnplusone)$ can be obtained as the algebraic quantity 
% on the conjugacy class. For any $g \in\PGL(V)$, we set 
\begin{equation} \label{eq:translength}
 \ell(\gamma)=\frac12 \log\frac{\lambda_1}{\lambda_{\mathrm{n}+1}} (\tilde \gamma),
\end{equation}
where $\tilde \gamma\in\GL(\Rnplusone)$ is any lift of $\gamma$, and $\lambda_1(\tilde \gamma),\dots,\lambda_{n+1}(\tilde \gamma)$ are the moduli of the (complex) eigenvalues of $\tilde \gamma$ (with multiplicity), ordered so that $\lambda_1(\tilde \gamma) \geq \dots\geq \lambda_{\mathrm{n}+1}(\tilde \gamma)$.

We say $\Omega$ is {\bf divisible} if there is some discrete subgroup $\Gamma \leq \Aut(\Omega)$ whose action on $\Omega$ is cocompact. The quotients $\Omega / \Gamma$ in this case are most closely analogous to closed hyperbolic manifolds, and share many of their good geometric and dynamical properties. In particular, in \cite{Benoist_CDI} Benoist shows that if 
$\Omega \subset \RPn$ is a properly convex domain which is divisible by $\Gamma$, then the following are equivalent:
\begin{enumerate}[i)]
\item $\Omega$ is strictly convex,
\item $\del\Omega$ is $C^1$,
\item $\Gamma$ is $\delta$-hyperbolic,
\item $(\Omega, d_\Omega)$ is $\delta$-hyperbolic,
\item the geodesic flow on the quotient $S\Omega / \Gamma$ is Anosov.
\end{enumerate}
Recall the Anosov condition for a flow roughly means that the tangent bundle splits into the flow direction, and the stable and unstable distributions, such that vectors in the stable (\resp\,unstable) distribution are uniformly exponentially contracted by the flow in forwards (\resp\,backwards) time. 
% ; in particular one has a splitting $TS\Omega = \real X \oplus E^s \oplus E^u$ where $X$ is the flow direction and $E^s$ and $E^u$ the stable and unstable distributions, which are respectively tangent to stable and unstable submanifolds of $S\Omega$ (which project to horospheres in $\Omega$.)

% \begin{figure}[h!]
%     \centering
%     \includegraphics[width=.35\textwidth]{figures/hilbhoros}
%     \caption{The projection of the stable and unstable manifolds to $T^1\Omega$, at a vector $v \in T^1M$ with foot-point $x$ pointing towards $x^+$: the set of tangent arrows along the horosphere based at $x^-$ is the projection of the unstable manifold; the set of arrows along the horosphere based at $x^+$ is the projection of the stable manifold.} \label{fig:hilbhoros}
% \end{figure}

\subsection{Boundary points of properly convex domains}\label{subsec:bdry pts}
Let $\Omega \subset \RPn$ be a domain. Here we introduce notions and terminology which are useful for dealing more carefully with the boundary $\del\Omega$ when $\Omega$ is not strictly convex and so $\del\Omega$ may contain line segments. Most of this discussion is taken from \cite[\S2]{Blayac_PSdensities}.

The relative interior 
% $\mathrm{int}_{rel}(\Omega)$ 
% (\resp the relative boundary $\del_{rel} \Omega$) 
of a subset $K\subset\RPn$ is its topological interior 
% (resp. boundary) 
with respect to the projective subspace it spans.
For $x \in \Omega$, the open face $F_\Omega(x)$ of $x$ in $\overline\Omega$ consists of the points $y \in \overline\Omega$ such that the segment $[x\,y]$ is contained in the relative interior of a segment contained in $\overline\Omega$.
% The closed face of $x$ is $ \overline{F_\Omega(x)}$.

A point $x \in \del \Omega$ is said to be {\bf extremal} (\resp {\bf strongly extremal}) if $F_\Omega(x) = {x}$ 
(resp.\ $[x,y]\cap\Omega\neq\emptyset$ for any $y\in\overline\Omega\smallsetminus\{x\}$);
%(resp.\ $x \notin \overline{F_\Omega(y)}$ for any $y \in \del \Omega \smallsetminus \{x\}$, the closure being taken in $\overline\Omega$);
$\Omega$ is strictly convex if all the points in the boundary are extremal (and hence strongly extremal).

A supporting hyperplane of $\Omega$ at $\xi \in \del \Omega$ is a hyperplane which contains $\xi$ but does not intersect $\Omega$. Note that there always exists such a hyperplane, by proper convexity. The point $\xi$ is said to be a {\bf $C^1$} point of $\del \Omega$ if there is only one supporting hyperplane of $\Omega$ at $\xi$.

We denote by $\del_{\sse} \Omega$ the set of $C^1$ and strongly extremal points in $\del\Omega$. Observe that when $\Omega$ is strictly convex with $C^1$ boundary, $\del_{\sse} \Omega = \del\Omega$.

Given $\xi, \eta \in \del\Omega$, we define
$d_{\overline\Omega}(\xi,\eta) := d_F(\xi,\eta)$ if $\xi$ and $\eta$ are in a common open face $F$, and $d_{\overline\Omega}(\xi,\eta) = \infty$ otherwise. This extends the distance $d_{\overline\Omega}(x,y) = d_\Omega(x,y)$ for $x,y\in\Omega$, in the sense that
$$\liminf_{x\to\xi,y\to\eta} d_{\overline\Omega}(x,y) \geq d_{\overline\Omega}(\xi,\eta). $$
If $\Omega$ is strictly convex with $C^1$ boundary, then $d_{\overline\Omega}(\xi,\eta)=\infty$ for all $\xi\in\del\Omega$ and $\eta\in\overline\Omega\smallsetminus\{\xi\}$. Below we will occasionally write $B_{\overline\Omega}(x,R)$, for $x\in\overline\Omega$ and $R>0$, to denote a closed ball of radius $R$ and center in the $d_{\overline\Omega}$-metric, \ie the closed ball centered at $x$ in the face of $x$ equipped with its own Hilbert metric. 

We also define the {\bf simplicial distance} $d_{\spl}(\xi,\eta)$ between two points of $\del\Omega$ as the minimal number of points $\xi_1,\dots,\xi_n \in \del\Omega$ such that $\xi_n=\eta$ and $$[\xi\, \xi_1] \cup [\xi_1 \xi_2] \cup\dots\cup [\xi_{n-1} \xi_n] \subset \del\Omega,$$
and we set $d_{\spl}(\xi,\xi)=0$.

Finally, $\Geod^\infty\Omega$ denotes the space of pairs of points $(\xi,\eta) \in (\del \Omega)^2$ for which there exists a bi-infinite straight-line geodesic in $\Omega$ with endpoints $\xi$ and $\eta$. When $\Omega$ is strictly convex with $C^1$ boundary, $\Geod^\infty\Omega = \del^2\Omega$, the set of pairs of distinct points of $\del\Omega$.

\subsection{Horofunctions, horoboundaries and horospheres}
\label{sub:horothings}

Next, we introduce some notions associated to the Hilbert geometry which will naturally arise in the construction of the Patterson--Sullivan densities which yield our ergodicity and equidistribution results: the horofunction boundary and associated objects in this subsection, and the Gromov product in the next.

Let $\Omega \subset \RPn$ be a domain. Given $z \in \Omega$, the {\bf horofunction} $\beta_z: \Omega \times \Omega \to \real$ is defined by 
\[ \beta_z(x,y) := d_\Omega(x,z) - d_\Omega(y,z) .\]
We remark that this uses the sign convention adopted in \cite{Roblin}, which is a little more intuitive geometrically and helpful for working with shadows (see \S\ref{subsec:shadows}); this is opposite to the general sign convention which appears e.g. in \cite{BridsonHaefliger}.

Let $C(\Omega)$ be the space of continuous ($\real$-valued) functions on $\Omega$, equipped with the topology of uniform convergence. If we fix a basepoint $o \in \Omega$, we may check that the map $\beta: \Omega \to C(\Omega)$
% \todo{FZ: changed this because $\iota$ is the involution on $S\Omega$ (p. 5)} 
given by $z \mapsto \beta_z(\cdot,o)$ is an embedding (see \cite[Prop.\,2.2]{WalshEmbedding}); moreover, it is easy to check that $\beta_z(\cdot,o): \Omega \to \real$ is a 1-Lipschitz map for any $z\in\Omega$, using the triangle inequality.  By the Arzel\`a--Ascoli theorem, the image $\beta(\Omega)$ is relatively compact, and the horofunction compactification $\overline{\Omega}{}^{\hor}$ of $\Omega$ is defined as $\overline{\beta(\Omega)}$, and the {\bf horoboundary} $\del_{\hor}\Omega$ is given by $\overline{\beta(\Omega)} \smallsetminus \beta(\Omega)$, where the closure is taken in $C(\Omega)$. Points in the horoboundary will also be called horofunctions. This construction was first introduced by Gromov \cite[\S1.2]{GromovVrac}.

Horofunctions $\beta_\xi \in \overline\Omega{}^{\hor}$ satisfy a cocycle condition
\[  \beta_\xi(x, y) + \beta_\xi(y, z) =\beta_\xi(x, z) \]
for all $x,y,z \in \Omega$, and are also $\Aut(\Omega)$-invariant, in the sense that
\[  \beta_{\gamma\xi}(\gamma x, \gamma y) = \beta_\xi(x,y) \]
for all $\gamma \in \Aut(\Omega)$ and $x, y \in \Omega$. Below, we will sometimes also call a horoboundary point $\beta_\xi$ the ``horofunction based at $\xi \in \del_{\hor}\Omega$''.

The {\bf horoball} (\resp {\bf horosphere}) based (or centered) at $\beta_\xi \in \del_{\hor}\Omega$ (or, we will also write to lighten the notation, based or centered at $\xi$) and passing through $x \in \Omega$ is the set
\begin{gather*}
\horoball_\xi(x) = \{y \in \Omega \:|\: \beta_\xi(x,y) > 0\} \\
\left(\text{\resp }  \horosphere_\xi(x) = \{y \in \Omega \:|\: \beta_\xi(x,y) = 0\}  \right).
\end{gather*}
Note that horospheres (\resp horoballs) are limits of spheres (\resp balls) for the Hausdorff topology; this implies that horoballs are convex, since Hilbert balls are convex \cite[(18.12)]{BusemannGeomGeod}, and their boundaries are the horospheres, which are homeomorphic to $\real^{\dimension-1}$.

By a theorem of Walsh \cite[Th.\,1.3]{Walsh}, the identity map on $\Omega$ extends to a continuous surjective map $\pi_{\hor}: \overline\Omega{}^{\hor} \to \overline\Omega$.
% \ie the horofunction compactification  dominates the usual projective compactification. 

If we are given $\xi\in\del_{\hor}\Omega$ and one horosphere centered at $\xi$, then we can geometrically describe all horospheres centered at $\xi$ in terms of the given horosphere and the projection $\pi_{\hor}(\xi)\in\del\Omega$. More precisely, for all $x,y\in\Omega$ with $y\in[x\,\pi_{\hor}(\xi)]$, it is an immediate consequence of Proposition~\ref{fact:horofoliation}
below that the map sending $x'\in\horosphere_\xi(x)$ to the unique point $y'\in[x \,\pi_{\hor}(\xi))$ at distance $d_\Omega(x,y)$ from $x'$ is a homeomorphism from $\horosphere_\xi(x)$ onto $\horosphere_\xi(y)$; thus the horospheres centered at $\xi$ foliate the domain $\Omega$.
\begin{prop}[{\cite[Lem.\,4.3]{Blayac_PSdensities}}]
\label{fact:horofoliation}
 Let $\Omega$ be a domain and $\xi\in\overline\Omega{}^{\hor}$, then $\beta_\xi(x,y)=d_\Omega(x,y)$ for all $x,y\in\Omega$ such that $y\in[x\,\pi_{\hor}(\xi)]$. 
\end{prop}
\begin{proof}
 If $\beta_\xi(x,y)<d_\Omega(x,y)$, then we can find $(z,y')\in\Omega^2$ close enough to $(\xi,y)$ such that $\beta_{z}(x,y')<d_\Omega(x,y')$ and $y'\in[x\,z]$, which contradicts the fact that $[x\,z]$ is a geodesic segment for the Hilbert metric. Hence $\beta_\xi(x,y)$ is bounded below by $d_\Omega(x,y)$; the other bound follows from the triangle inequality.
\end{proof}

We also make the following elementary observation:
\begin{lem} \label{lem:disjhoroballs}
 Given a domain $\Omega\subset \RPn$ and two distinct points $\xi,\eta\in\overline\Omega{}^{\hor}$, if $x\in[\pi_{\hor}(\xi) \, \pi_{\hor}(\eta)]\cap \Omega$ then $\horoball_\xi(x)$ and $\horoball_\eta(x)$ are disjoint.
 
 \begin{proof}
  Suppose there is $y\in \horoball_\xi(x)\cap \horoball_\eta(x)$. Then we can find $(\xi',x',\eta')\in\Omega^3$ close enough to $(\xi,x,\eta)$ such that $\beta_{\xi'}(x',y)>0$ and $\beta_{\eta'}(x',y)>0$ and $x'\in [\xi'\eta']$. This leads to the following contradiction:
  \begin{align*}
   d_\Omega(\xi',\eta') & = d_\Omega(\xi',x')+d_\Omega(x',\eta') \\
   & = \beta_{\xi'}(x',y) +d_\Omega(\xi',y)+\beta_{\eta'}(x',y)+d_\Omega(y,\eta') \\
   & > d_\Omega(\xi',\eta').\qedhere
  \end{align*}
 \end{proof}
\end{lem}

It follows from this that
\begin{prop}[{cf. \cite[Fig 7.4]{CM12}}] \label{prop:horoball_geomdesc}
Given $\Omega$ a domain, $\xi, \eta \in \del\Omega$ two distinct boundary points, and two horoballs $\horoball_\xi$ and $\horoball_\eta$, centered at $\xi$ and $\eta$ respectively, $\horoball_\xi \cap \horoball_\eta \neq \varnothing$ if and only if $(\xi\,\eta) \cap \horoball_\xi \cap \horoball_\eta \neq \varnothing$.
\begin{proof}
It is clear that $\horoball_\xi \cap \horoball_\eta \neq \varnothing$ if $(\xi\,\eta) \cap \horoball_\xi \cap \horoball_\eta \neq \varnothing$.

Conversely, suppose $\horoball_\xi \cap \horoball_\eta \cap (\xi\,\eta) = \varnothing$. Since $(\xi\eta)$ is connected, it contains a point $x$ outside of $\horoball_\xi\cup\horoball_\eta$. Now $\horoball_\xi\subset\horoball_\xi(x)$ and $\horoball_\eta\subset\horoball_\eta(x)$ by definition of the horoballs, and $\horoball_\xi(x)\cap\horoball_\eta(x)$ is empty by the previous lemma.
\end{proof}
\end{prop}

As noted by Crampon--Marquis \cite{CM12} and Bray \cite{Bray}, regularity properties of the projective boundary $\del\Omega$ have repercussions for the geometry and regularity of the horospheres. In particular any $\xi\in\del\Omega$ has exactly one preimage in $\del_{\hor}\Omega$ if and only if it is a $C^1$ point of $\del\Omega$ \cite[Lem.\,3.2]{Bray}. We will from now on abuse notation by identifying any $C^1$ point of the projective boundary with its preimage in the horoboundary. 

If $\xi$ is $C^1$, the horospheres centered at $\xi$ coincide with the algebraic horospheres defined by Cooper--Long--Tillmann \cite[\S3]{CLT}. More precisely, for $\eta\in\del\Omega\smallsetminus T_\xi\del\Omega$ and $x\in (\xi\,\eta)$, the horosphere $\horosphere_\xi(x)$ is the image of the map that sends  $\zeta\in\del\Omega\smallsetminus T_\xi\del\Omega$ to the point of intersection 
$[\xi \, \zeta] \cap \Span\left(x, T_\xi\del\Omega \cap \Span(\eta,\zeta) \right)$, 
and moreover this map is the restriction of a projective transformation that fixes every point of $T_\xi\del\Omega$.

As a consequence, if $\Omega$ is strictly convex with $C^1$ boundary, then the horoboundary identifies with $\del\Omega$, all horofunctions are $C^1$, and all horoballs are strictly convex with $C^1$ boundary.

\subsection{Gromov products}

Given $\Omega \subset \RPn$ a domain and $x, y, z \in \Omega$, the {\bf Gromov product} $\grop{x}yz$ is defined as 
\[ \grop{x}yz = \frac 12[ d_\Omega(x,y)+d_\Omega(x,z)-d_\Omega(y,z) ] .\]
Roughly speaking, it measures how much the sides of a geodesic triangle overlap. For instance, given a tree $T$, we have $\grop{x}yz = 0$ for any $x, y, z \in T$; more generally, the smaller the Gromov product, the thinner the geodesic triangle is. Indeed, there is a characterization of $\delta$-hyperbolicity in terms of the Gromov product.

We note the following transformation properties that are useful for proving the invariance of our Sullivan measures below: for all $\phi \in \Aut(\Omega)$,
\begin{align*}
\grop{\phi x}{\phi\xi}{\phi\eta} = \grop{x}\xi\eta & = \grop{x'}\xi\eta + \frac 12\left(\beta_\xi(x,x') + \beta_\eta(x,x') \right)
\end{align*}
We remark also that $\beta_\xi(x,u) + \beta_\eta(x,u) = 2 \grop{x}\xi\eta$, for any $u \in (\xi\,\eta)$. 

As in CAT$(-1)$ geometry (though not always for general Gromov hyperbolic spaces, see \cite[\S III.H.3.15]{BridsonHaefliger}), the Gromov product may be extended continuously to an appropriate set of pairs of points of the boundary, here $\pi_{\hor}^{-1}(\Geod^\infty \Omega)$, where $\Geod^\infty \Omega = \{(x,y)\in\del\Omega{}^2 \:|\: [x\,y]\cap\Omega\neq\varnothing\}$ is as at the end of \S\ref{subsec:bdry pts};
we also set $\Geod\Omega=\{(x,y)\in\overline\Omega{}^2 \:|\: [x\,y] \cap \Omega \neq \varnothing\}$. 
% Recall that for general Gromov hyperbolic spaces, the extension to the boundary cannot always be made continuously, and requires taking a $\sup \liminf$ (see \cite[\S III.H.3.15]{BridsonHaefliger}). 
% The continuous extension of the Gromov product to the projective boundary of $C^1$ domains was also proven using a geometric argument by Benoist \cite[Lem.\,5.2]{Benoist_CHQI}.

\begin{prop} [{\cite[Prop.\,3.1]{Blayac_PSdensities}}]\label{prop:Gromovprod}
The map $(\xi,\eta,x) \mapsto \grop x\xi\eta$ defined on $\Omega^3$ extends continuously to $\pi_{\hor}^{-1}(\Geod\Omega)\times\Omega$, and the extension continues to satisfy
\begin{gather*}
x \in (\xi\eta) \implies \grop x\xi\eta = 0 \\
2 \grop x\xi\eta = 2 \grop y\xi\eta + \beta_\xi(x,y) + \beta_\eta(x,y) \\
| \grop x\xi\eta - \grop y\xi\eta | \leq d_\Omega(x,y)
\end{gather*}
\begin{proof}
Let $A:=\{(\xi,\eta,x,y)\in\pi_{\hor}^{-1}(\Geod\Omega)\times\Omega^2 \:|\: y\in[\xi\,\eta]\}$. The projection $(\xi,\eta,x,y) \mapsto (\xi,\eta,x)$ from $A$ to $\pi_{\hor}^{-1}(\Geod\Omega)\times\Omega$ is continuous, surjective and open. Therefore, it is enough to prove that the function $(\xi,\eta,x,y)\mapsto \langle \xi,\eta \rangle_x$ defined on $A\cap\Omega^4$ extends continuously to $A$. This is an immediate consequence of the fact that $\langle \xi,\eta \rangle_x = \beta_\xi(x,y) + \beta_\eta(x,y)$ for any $(\xi,\eta,x,y)\in A\cap\Omega^4$.
\end{proof}
\end{prop}

It is immediate from the first and third properties in the above proposition that
\begin{corn} \label{lem:grop_bound} 
If $(\xi,\eta,x) \in \pi_{\hor}^{-1}(\Geod\Omega) \times \Omega$, then 
$\grop x\xi\eta \leq d_\Omega(x,(\xi\,\eta))$.
\end{corn}

If $\Omega$ is strictly convex with $C^1$ boundary, then $\del_{\hor}\Omega$ is identified with $\del\Omega$, and we see that the Gromov product is well defined on the set $\del^2\Omega$ of pairs of distinct points of $\del\Omega$.

% If $\Omega$ is strictly convex with $C^1$ boundary, the Gromov product can be extended to the boundary $\del\Omega$: this follows from the above since $\pi_{\hor}$ is in this case a homeomorphism, and $\Geod^\infty \Omega = \del^2\Omega$\todo[color=P.-L.color!50!white,linecolor=black]{By the way, according to Fanny, $X^{(2)}$ is a classical notation for the set of pairs of distinct elements of a set $X$}. In this case we may also more directly define the extension as follows: given $\xi, \eta \in \del\Omega$ and $x \in \Omega$, set
% \[ \grop{x}\xi\eta= \lim_{\substack{a_n\to\xi\\b_n\to\eta}} \grop{x}{a_n}{b_n} = \lim_{\substack{a_n\to\xi\\b_n\to\eta}} \frac 12[ d_\Omega(x,a_n)+d_\Omega(x,b_n)-d_\Omega(a_n,b_n) ] .\]

% The following inequality will be useful below for providing estimates for our Sullivan measure:
% \begin{lem} \label{lem:grop_bound}
% If a geodesic segment $(uv)$ intersects $B(x,r)$, then $\grop{x}uv \leq r$.
% \begin{proof}
% Suppose first that $(uv)$ is a finite geodesic segment. Pick $w \in (uv) \cap B(x,r)$. Then
% \begin{align*}
% d_\Omega(u, x) & \leq d_\Omega(u,w) + d_\Omega(w,x) \leq d_\Omega(u,w) + r \\
% d_\Omega(v, x) & \leq d_\Omega(v,w) + r
% \end{align*}
% and adding the two together we get
% \[ \grop{w}{x}{y} = \frac12 (d_\Omega(w,x) + d_\Omega(w,y) - d_\Omega(x, y) ) \leq r .\]
% Note that these inequalities continue to hold under limits, and so the lemma continues to hold if the geodesic in question is infinite or bi-infinite.
% \end{proof}
% \end{lem}

\subsection{Convex cocompactness and geometric finiteness} \label{sub:geomfin}

Let $\Omega$ be a domain and $\Gamma\leq\Aut(\Omega)$ be a discrete subgroup. The {\bf full orbital limit set}, introduced in \cite{DGK}, is
\begin{equation}\label{eq:orb}
\Lambda_\Gamma^{\orb} := \bigcup_{x \in \Omega} \overline{\Gamma\cdot x}\smallsetminus\Gamma\cdot x.
\end{equation}
The {\bf convex core} of the quotient $M=\Omega/\Gamma$ is the projection of the convex hull $CH(\Lambda_\Gamma^{\orb})$ in $\Omega$ of the full orbital limit set. The idea is that the convex core contains all the dynamics of the geodesic flow: for instance, one can check that any non-wandering geodesic lies in it, see \cite[Obs.\,2.12]{Blayac_topmixing}. We will obtain our strongest results when the convex core is geometrically well-understood, more precisely when it is compact (\ie the group is convex cocompact). If $\Omega$ is strictly convex with $C^1$ boundary, these results extend to the case where the convex core can be decomposed into a compact part and finitely many cusp-like parts (\ie the group is geometrically finite).

\begin{defn}[{\cite[Def.\,1.11]{DGK}}]
$\Gamma \leq \Aut(\Omega)$ is {\bf convex cocompact} if the quotient of the convex core 
$CH(\Lambda_\Gamma^{\orb}) / \Gamma$
is compact and non-empty.
\end{defn}

Equivalently, one can characterize the subgroups $\Gamma\leq\Aut(\Omega)$ acting convex cocompactly on $\Omega$ in terms of its action on the boundary:
\begin{defn}
$\xi \in \Lambda_\Gamma^{\orb}$ is a {\bf conical limit point} if there exist a sequence of elements $(\gamma_n) \subset \Gamma$ and $x \in \Omega$ such that $(\gamma_n x)_n$ tends to $\xi$ and $\sup_n d_\Omega(\gamma_n x, [x\,\xi)) < \infty$. We denote by $\Lambda_\Gamma^{\con}$ the set of all conical limit points in $\Lambda_\Gamma^{\orb}$.
\end{defn}
\begin{prop}[{\cite[Cor.\,4.9 \& Lem.\,4.19]{DGK}}]
$\Gamma \leq \Aut(\Omega)$ is convex cocompact if and only if $\Lambda_\Gamma^{\orb} = \Lambda_\Gamma^{\con}$ and this set is closed.
\end{prop}
% (Also \cite[4.11]{DGK} and \cite[Cor.\,8.6]{CM12} when $\Omega$ is strictly convex with $C^1$ boundary.

% \subsubsection{Geometric finiteness}

Now further suppose, for the rest of this section, that $\Omega$ is strictly convex with $C^1$ boundary. In this case, the orbital limit set $\Lambda_\Gamma^{\orb}$ of any discrete subgroup $\Gamma\leq \Aut(\Omega)$ is equal to $\overline{\Gamma\cdot x}\smallsetminus\Gamma\cdot x$ for any $x\in\Omega$, and is closed; therefore $\Gamma$ acts convex cocompactly on $\Omega$ if and only if every point of $\Lambda_\Gamma^{\orb}$ is conical. 

Moreover, the above definition of conical limit points is equivalent, in this case, to a characterization purely in terms of the action $\Gamma \actson \del\Omega$ \cite[Lem.\,5.10]{CM12}.

Geometric finiteness can be seen as a weakening of convex cocompactness, where we allow non-conical limit points, which must however satisfy good properties that allow for some control on the non-compact parts of the convex core.
This notion arose first in the setting of Kleinian groups, and has subsequently been extended to higher dimensions and the more general setting of pinched negative curvature in \cite{Bowditch_GF}; the group-theoretic notion of relative hyperbolicity, see e.g. \cite{Bowditch_relhyp}, may be seen as an extension of geometric finiteness to a more general $\delta$-hyperbolic setting. %This is essentially a notion associated with negatively curved geometry, and generalizes the geometric and dynamical behaviour of finite-volume quotients of hyperbolic space.

Before we state the definition of geometrically finite subgroups of $\Aut(\Omega)$, we recall the classification of elements and of elementary discrete subgroups of $\Aut(\Omega)$ due to Crampon--Marquis \cite[Th.\,3.3 \& \S3.5]{CM12} (keep in mind that we have assumed $\Omega$ is strictly convex with $C^1$ boundary). Any element $\gamma\in\Aut(\Omega)$ is 
\begin{itemize}
    \item either {\bf elliptic}: it fixes a point of $\Omega$;
    \item or {\bf parabolic}: it fixes a unique point of $\overline\Omega$, which is in the boundary $\partial\Omega$;
    \item or {\bf hyperbolic}: it fixes exactly two points of $\overline\Omega$, which are in the boundary.
\end{itemize}
Any discrete subgroup $\Gamma\leq \Aut(\Omega)$ is 
\begin{itemize}
    \item either {\bf elliptic}: $\#\Lambda_\Gamma^{\orb}=0$, and $\Gamma$ is finite, fixes a point of $\Omega$, and consists of elliptic elements;
    \item or {\bf parabolic}: $\#\Lambda_\Gamma^{\orb}=1$, and $\Gamma$ fixes a unique point of $\del\Omega$, consists of elliptic and (at least one) parabolic elements, and acts properly discontinuously on $\del\Omega \smallsetminus \Lambda_\Gamma^{\orb}$ (see Section~\ref{sec:parabgroups} for more properties of these groups);
    \item or {\bf elementary hyperbolic}: $\#\Lambda_\Gamma^{\orb}=2$, and $\Gamma$ consists of elliptic and (at least one) hyperbolic elements, and any hyperbolic element generates a finite-index subgroup;
    \item or {\bf non-elementary}: $\#\Lambda_\Gamma^{\orb}=\infty$, and $\Gamma$ acts minimally on $\Lambda_\Gamma^{\orb}$ (which is perfect), contains a non-abelian free subgroup made of hyperbolic elements, and fixes no point of $\overline\Omega$.
\end{itemize}

\begin{defn} \label{defn:limitset}
Let $\Omega$ be a strictly convex domain with $C^1$ boundary and $\Gamma$ be a discrete subgroup of $\Aut(\Omega)$.
%$\xi \in \Lambda_\Gamma$ is a {\bf conical limit point} if there exists a sequence of elements $(\gamma_n) \subset \Gamma$ such that for some $x \in \Omega$, $\gamma_n x \to \xi$ and $\sup_n d_\Omega(\gamma_n x, [x\xi)) < \infty$.\todo[color=P.-L.color!50!white,linecolor=black]{I suggest we put the definition of bounded parabolic point here}

$\xi \in \Lambda_\Gamma^{\orb}$ is a {\bf bounded parabolic point} if the stabilizer $\Stab_\Gamma(\xi)$ is parabolic and acts cocompactly on $\Lambda_\Gamma^{\orb} \smallsetminus \{\xi\}$. We say $\Gamma$ {\bf acts geometrically finitely on $\del\Omega$} if every point in $\Lambda_\Gamma^{\orb}$ is either conical or bounded parabolic.

$\xi \in \Lambda_\Gamma^{\orb}$ is a {\bf uniformly bounded parabolic point} if $\Stab_\Gamma(\xi)$ is parabolic and acts cocompactly on 
% $\mathcal{D}_\xi (\overline{CH(\Lambda_\Gamma^{\orb} \setminus \{\xi\})})$, 
the set of lines through $\xi$ which meet $\overline{CH(\Lambda_\Gamma^{\orb} \setminus \{\xi\})}$, the closure of the convex hull of $\Lambda_\Gamma^{\orb} \setminus \{\xi\}$. We say $\Gamma$  {\bf acts geometrically finitely on $\Omega$ } if every point in $\Lambda_\Gamma^{\orb}$ is either conical or uniformly bounded parabolic.
\end{defn}
Observe that if $\Gamma$ acts geometrically finitely on $\Omega$, then it also acts geometrically finitely on $\del\Omega$; the converse is not true in general (see \cite[\S10.3]{CM12}). We work in the present article with groups $\Gamma$ that act geometrically finitely on $\del\Omega$. 

Crampon--Marquis gave a more concrete description \cite[Th.\,1.2]{CM12} of groups that acts geometrically finitely on $\Omega$, in terms of geometrical and topological properties of the quotient $\Omega/\Gamma$. For example, they proved that $\Gamma$ acts geometrically finitely on $\Omega$ if and only if $CH(\Lambda_\Gamma^{\orb})/\Gamma$ admits a decomposition into a compact part and finitely many disjoint non-compact but well-understood parts (see the standard parabolic regions in \cite[Def.\,7.22]{CM12}), and this property implies that $\Omega/\Gamma$ is tame (\ie the interior of a compact orbifold with boundary), that $\Gamma$ is hyperbolic relative to the maximal parabolic subgroups, and that $\Gamma$ is finitely presented. 

We will establish and use partial results of a similar nature 
% similar --- but less precise\todo[color=P.-L.color!50!white,linecolor=black]{I cannot find the good words. Your formulation, saying that we have ``partial results'' was good but I do not succeed in putting it here. You can modify as you wish} --- results 
for groups acting geometrically finitely on the boundary $\del\Omega$.
% although we establish only partial results in this direction here. 
More specifically,  
we will decompose into compact and non-compact parts a subset of $CH(\Lambda_\Gamma^{\orb})/\Gamma$, which is the quotient by $\Gamma$ of the union of lines between points in $\Lambda_\Gamma^{\orb}$ (see Propositions~\ref{prop:finitely many cusps} and \ref{prop:noncuspidal is cpct}).

If $\Gamma$ acts geometrically finitely on $\Omega$, and in addition the parabolic ends satisfy an additional technical condition (``asymptotic hyperbolicity''; this condition is always satisfied, for instance, when the quotient $\Omega / \Gamma$ has finite volume), then the Hilbert geodesic flow on $S \Omega / \Gamma$ is uniformly hyperbolic on the non-wandering set $\mathtt{NW}$ \cite[Th.\ 5.2]{CM14}, generalizing Benoist's result that the Hilbert geodesic flow on a compact convex projective manifold is Anosov. 
% \todo{Actually there's also an extra ``asymptotically hyperbolic'' assumption on the cusps (which I had thought was automatic once one has the inscribed and circumscribed ellipsoids at the parabolic fixed point, but apparently isn't quite.)}
% In particular, we have, over $\mathtt{NW} \subset S\Omega$, a splitting of the tangent bundle into the flow direction and the stable and unstable distributions with the latter tangent to stable and unstable submanifolds of $S\Omega$ projecting to horospheres in $\Omega$.  

There is also an analogue of Benoist's characterization of strict convexity in the finite-volume setting, due to Cooper--Long--Tillmann \cite[Th.\ 0.15]{CLT},
which states that for 
$M = \Omega / \Gamma$ a properly convex manifold of finite volume which is the interior of a compact manifold $N$ with boundary and where the holonomy of each component of $\del N$ is parabolic, the following are equivalent:
\begin{enumerate}[i)]
\item $\Omega$ is strictly convex,
\item $\del\Omega$ is $C^1$,
\item $\pi_1N$ is hyperbolic relative to the subgroups of the boundary components.
\end{enumerate}
We remark that there is a fair amount of overlap between the geometric results in \cite{CM12} and \cite{CLT}, although the presentation of their results differs somewhat; the reader may, in large part, consult either or both of these sources to taste.

\section{The rank-one condition} \label{sec:rk1}

Our ergodicity and equidistribution results will require at least a little negative curvature, which may be provided by $\Omega$ being strictly convex with $C^1$ boundary, or more generally by the rank-one condition.

In Riemannian geometry, or more generally CAT$(0)$ geometry, the rank-one condition describes, in the first instance, geodesics in ``negatively curved directions'', or isometries which act like hyperbolic isometries in negatively curved geometry. We then say that any discrete group of isometries containing a rank-one isometry is rank-one, because --- very loosely speaking --- these isometries tend to display north-south dynamics that propagate their rank-one behavior throughout the subgroup.

In the setting of non-positively curved Riemannian manifolds, the higher-rank rigidity theorem of Ballmann--Spatzier tells us that any irreducible compact Riemannian manifold of non-positive curvature either is rank-one (\ie its isometry group is rank-one), or is a higher-rank locally symmetric space. 

Recently, A.\ Zimmer established  a similar result \cite[Th.\,1.4]{zimmer_higher_rank} for closed convex projective manifolds, using a definition of the rank-one condition in Hilbert geometry first defined and studied by M.\ Islam:
\begin{defn}[\cite{Mitul_rk1}]
Given a domain $\Omega \subset \RPn$, an infinite-order element $\gamma \in \Aut(\Omega)$ is {\bf rank-one} if it preserves a straight-line geodesic of $\Omega$ both of whose endpoints are $C^1$ and strongly extremal points of $\del\Omega$.

A discrete subgroup $\Gamma \leq \Aut(\Omega)$, and the quotient $\Omega/\Gamma$ are said to be rank-one if $\Gamma$ contains a rank-one element.
\end{defn}
Note that rank-one elements are biproximal \cite[Prop.\,6.3]{Mitul_rk1}, and that any conjugate of a rank-one element remains rank-one.

A rank-one group $\Gamma$, and the quotient $\Omega/\Gamma$, are said to be {\bf elementary} if and only if $\Gamma$ is virtually isomorphic to $\ints$, and {\bf non-elementary} otherwise. This definition is compatible with the one from the previous section.

\begin{defn}
An element $\gamma \in \PGL(\Rnplusone)$ is {\bf proximal} if it has an attracting fixed point in $\RPn$, which is denoted by $\gamma^+$. If $\gamma$ and its inverse are both proximal, then we say that $\gamma$ is {\bf biproximal}, and write $\gamma^- := (\gamma^{-1})^+$.

Given a subgroup $\Gamma \leq\PGL(\Rnplusone)$, the {\bf proximal limit set} $\Lambda_\Gamma^{\prox}$ is the closure of the set of all attracting fixed points of proximal elements in $\Gamma$.

Below, we will write $\Lambda_\Gamma$ to refer to $\Lambda_\Gamma^{\prox}$, and label other limit sets, where we refer to them, with their respective superscripts ($\Lambda_\Gamma^{\orb}$, $\Lambda_\Gamma^{\con}$, etc.)
\end{defn}

One can show that $\Lambda_\Gamma\subset\Lambda_\Gamma^{\orb}$, and moreover that if $\Omega$ is strictly convex with $C^1$ boundary and $\Gamma$ is discrete, then the proximal limit set as defined here agrees with the full orbital limit set as defined in (\ref{eq:orb}) unless $\Gamma$ is a parabolic group. 

The proximal limit set can be used, as in \cite[Def.\,1.1]{Blayac_topmixing} to define a natural subset of the unit tangent bundle $SM$ of a convex projective manifold $M$ given by the set of vectors $v\in SM$ such that $\tilde v^+$ and $\tilde v^-$ belong to $\Lambda_\Gamma$ for any lift $\tilde v\in S\Omega$: this is called the {\bf biproximal unit tangent bundle} and denoted by $SM_{\bip}$. 

When $\Omega$ is strictly convex with $C^1$ boundary, $SM_{\bip} $ is the non-wandering set. Using Zimmer's higher-rank rigidity result \cite[Th.\,1.4]{zimmer_higher_rank}, the first author proved \cite{Blayac_proxestbord} that if $M$ is rank-one and compact, then $SM_{\bip}=SM$.
We will see (Corollary \ref{cor:rk1_top_mixing}) that the geodesic flow is topologically mixing on the biproximal unit tangent bundle of non-elementary rank-one convex projective manifolds.

\subsection{Abundance of rank-one geodesics}

Fixed points of rank-one elements behave well for our purposes; the next two results are precise instances of this general statement:

\begin{lem} \label{lem:diff att => diff rep}
Let $\Omega \subset \RPn$ be a domain and $\Gamma$ a discrete subgroup of $\Aut(\Omega)$. Let $\gamma_1$, $\gamma_2\in\Gamma$ be rank-one elements. 
If $\gamma_1^+=\gamma_2^+$, then $\gamma_1^-=\gamma_2^-$, and in fact $\gamma_1$ and $\gamma_2$ must have a common power.

\begin{proof}
Let $l_i := (\gamma_i^- \gamma_i^+)$ be the axes for $\gamma_i$ ($i=1,2$). These are well-defined and unique (\cite[Def 6.2 and Prop 6.3]{Mitul_rk1}), and $\gamma_i^{\pm1}$ translates along $l_i$ towards $\gamma_i^\pm$.

Fix $x_i \in l_i$ such that $\beta_{\gamma_1^+}(x_1,x_2)=0$. Given $n \geq 1$, write $m_n := \left \lfloor n \frac{d_\Omega(x_1, \gamma_1 x_1)}{d_\Omega(x_2, \gamma_2 x_2)} \right\rfloor$. Fix $\eps > 0$. Since $\gamma_i^k x_i = \gamma_i^+$ as $k\to\infty$, and $\gamma_1^+ = \gamma_2^+$, thus by \cite[Prop.\,5.1]{Blayac_topmixing}, we have 
\begin{equation*}
d_\Omega(x_1,\gamma_1^{-n}\gamma_2^{m_n}x_2) = d_\Omega(\gamma_1^n x_1, \gamma_2^{m_n} x_2) < \eps 
\end{equation*}
for all large enough $n$.

Since the action of $\Gamma$ on $\Omega$ is properly discontinuous, this implies for all sufficiently large $n$ we have $\gamma_1^{-n} \gamma_2^{m_n} = \id$. Fix any one such $N$, and define $h := \gamma_1^N=\gamma_2^{m_N}$. 

Then $d_\Omega(h^{-p}x_1,h^{-p}x_2) = d_\Omega(x_1,x_2)<\infty$ for all $p \in \ints_{>0}$. Since $h^{-p} x_i \to \gamma_i^-$ as $p \to \infty$, and the $\gamma_i^-$ are strongly extremal points, this implies that $\gamma_1^- = \gamma_2^-$.
\end{proof}
\end{lem}

% \subsection{non-elementary implies there are many rank-one geodesics}

The following proposition is a generalisation of \cite[Cor.\,6.3]{CM14}:

\begin{prop}\label{prop:density of att-rep pairs}
Let $\Omega \subset \RPn$ be a domain and $\Gamma$ a non-elementary rank-one discrete subgroup of $\Aut(\Omega)$. Then $\Lambda_\Gamma$ is infinite and is the smallest $\Gamma$-invariant closed subset of $\overline\Omega$, 
% \todo[color=P.-L.color!50!white,linecolor=black]{added this for use in my paper on PS densities in convex projective geometry}
and $\{(\gamma^-,\gamma^+) \st \gamma\in\Gamma \text{ rank-one}\}$ is dense in $\Lambda_\Gamma \times \Lambda_\Gamma$.
\begin{proof}
We start by showing $\#\Lambda_\Gamma \geq 3$. Fix a rank-one element $\gamma_0\in\Gamma$. Suppose, on the contrary, that $\#\Lambda_\Gamma \leq 2$, so that $\Lambda_\Gamma =\{\gamma_0^-,\gamma_0^+\}$. Any infinite-order element $\gamma \in \Gamma$ preserves the line $(\gamma_0^-\gamma_0^+)$ and is hence rank-one, and by Lemma~\ref{lem:diff att => diff rep} it has a common power with $\gamma_0$. This implies, by discreteness of $\Gamma$, that $\gamma_0$ generates a finite-index subgroup of $\Gamma$, contradicting our assumption that $\Gamma$ is non-elementary.

Therefore $\Gamma$ contains a proximal element $\gamma$ with $\gamma^+ \notin \{\gamma_0^-,\gamma_0^+\}$. We denote by $\gamma_*^-$ the repelling hyperplane of $\gamma$, which does not intersect $\Omega$. If $\gamma_0^+ \in \gamma_*^-$, then $\gamma_*^- \cap \del\Omega =\{\gamma_0^+\}$ since $\gamma_0^+$ is strongly extremal, so $g$ is rank-one \cite[Lem.\,3.2]{Blayac_topmixing} and $\gamma^+=\gamma_0^+$ by Lemma~\ref{lem:diff att => diff rep}, which is a contradiction. Hence $\gamma_0^+\not\in \gamma_*^-$, and similarly $\gamma_0^-\not\in \gamma_*^-$. The points $\{\gamma^n\gamma_0^+\}_n$ are pairwise distinct, and hence $\Lambda_\Gamma$ is infinite.

Next we show that $\Lambda_\Gamma$ is the smallest $\Gamma$-invariant closed subset of $\overline\Omega$, by showing that $\Lambda_\Gamma \subset\overline{\Gamma\cdot \xi}$ for any $\xi \in \overline\Omega$. Given any such $\xi$, the sequence $(\gamma_0^n\xi)_{n\geq 0}$ converges to some $\eta\in\{\gamma_0^-,\gamma_0^+\}$ (with $\eta=\gamma_0^-$ if and only if $\xi=\gamma_0^-$). The point $\gamma^+$, being the limit of $(\gamma^n\eta)_n$, belongs to $\overline{\Gamma\cdot\xi}$. Hence, $\gamma_0^\pm$, being the limit of $(\gamma_0^{\pm n} \gamma^+)_{n\geq 0}$, belongs to $\overline{\Gamma\cdot\xi}$.
For any proximal element $h\in\Gamma$, there exists $\zeta\in\{\gamma_0^+,\gamma_0^-\} \smallsetminus h_*^-$, and $h^+ = \lim_{n\to\infty} h^n\zeta$ is in $\overline{\Gamma\cdot \xi}$. Thus $\Lambda_\Gamma\subset \overline{\Gamma \cdot \xi}$, as desired.

Finally, let us show that $\{(\gamma^-,\gamma^+) : \gamma\in\Gamma \text{ rank-one}\}$ is dense in $\Lambda_\Gamma \times \Lambda_\Gamma$. Fix two non-empty open subsets $U,V\subset\Lambda_\Gamma$. Note that $\Lambda_\Gamma$ is 
% a perfect topological space, since it is 
infinite and minimal under the action of $\Gamma$,
% . Therefore, 
and using minimality we can find two (rank-one) conjugates $\gamma_1$ and $\gamma_2$ of $\gamma_0$ such that $\gamma_1^+\in U$ and $\gamma_2^+\in V\smallsetminus\{\gamma_1^+,\gamma_1^-\}$. Consider $h_n=\gamma_1^n\gamma_2^{-n}$ for $n\geq 0$; we will show that, for $n$ large enough, $h_n$ is rank-one with $h_n^+\in U$ and $h_n^-\in V$. 

By Lemma~\ref{lem:diff att => diff rep}, the four points $\gamma_1^-$, $\gamma_1^+$, $\gamma_2^-$ and $\gamma_2^+$ are distinct. Consider lifts $\tilde{\gamma}_1$ and $\tilde{\gamma}_2\in\GL(\Rnplusone)$, and numbers $\lambda_1,\lambda_2>0$ such that for $i=1,2$:
\begin{equation*}
    \tilde{\gamma}_i^n \xrightarrow[n\to\infty]{} \pi_i^+ \text{ and  } \lambda_i^n\tilde{\gamma}_i^{-n}\xrightarrow[n\to\infty]{} \pi_i^-,
\end{equation*}
where $\pi_i^\pm$ is the projector on $\gamma_i^\pm$ along $\gamma_{i*}^\mp$. For each $n\geq 0$ set $\tilde{h}_n:=\tilde{\gamma}_1^n\tilde{\gamma}_2^{-n}$. Then
\begin{align*}
    \lambda_2^n\tilde{h}_n  \xrightarrow[n\to\infty]{} \pi_1^+\pi_2^- \quad \text{and} \quad \lambda_1^n\tilde{h}_n^{-1} \xrightarrow[n\to\infty]{} \pi_2^+\pi_1^-.
\end{align*}
Since $\gamma_2^-\neq\gamma_1^-= \ker\pi_1^+ \cap \overline\Omega$, the endomorphism $\pi_1^+\pi_2^-$ is non-zero, and so is its square, since $\gamma_1^+\neq\gamma_2^+ = \ker\pi_2^- \cap \del\Omega$. Thus, $\pi_1^+\pi_2^-$ is, up to a multiplicative scalar, the rank-one projector on $\gamma_1^+$ along $\gamma_{2*}^+$. Similarly, $\pi_2^+\pi_1^-$ is the rank-one projector on $\gamma_2^+$ along $\gamma_{1*}^+$. As a consequence, $g_n$ is biproximal for $n$ large enough, and $(h_n^+)_n$ (\resp $(h_n^-)_n$)  converges to $\gamma_1^+$ (\resp $\gamma_2^+$). It remains to show that $h_n$ is rank-one for $n$ large enough, but this is a direct consequence of \cite[Prop.\,3.4.1]{Blayac_topmixing}.
\end{proof}
\end{prop}

\subsection{Irreducibility for rank-one groups}

Let $V = \Rnplusone$. Recall that a subgroup $\Gamma \leq \PGL(V)$ is said to be {\bf irreducible} if it does not preserve any proper subspace of $\proj(V)$, and {\bf strongly irreducible} if any finite-index subgroup of $\Gamma$ is irreducible. 
Not all non-elementary rank-one discrete groups of automorphisms of properly convex domains are strongly irreducible: consider for example a cocompact lattice $\Gamma$ of $\SO(1,2) < \SO(1,3) < \SL(4,\real)$. In this case, $\Gamma$ preserves the span of $\Lambda_\Gamma$, which is a proper subspace of $\proj(\real^4)$. Crampon and Marquis thought that this was basically the only obstruction (see \cite[Lem.\,2.4]{CM14}). This is not exactly true, though not far from the truth: the other possible obstruction is that the \emph{dual} proximal limit set $\Lambda^*_\Gamma$ may not span $\proj(V^*)$, as explained in the next proposition.

We recall that, given a domain $\Omega \subset \proj(V)$, the dual domain $\Omega^*$ consists of all points in $\proj(V^*)$ corresponding to hyperplanes of $\proj(V)$ which do not intersect $\bar\Omega$. Given a subgroup $\Gamma \leq \Aut(\Omega)$, we write $\Gamma^*$ to denote the image of $\Gamma$ under the natural identification between $\PGL(V)$ and $\PGL(V^*)$ sending $\gamma\in\PGL(V)$ to $(f\mapsto f\circ \gamma^{-1})\in\PGL(V^*)$, and $\Lambda_\Gamma^*$ to denote the proximal limit set of $\Gamma^*$ in $\del\Omega^*$.

\begin{prop} \label{prop:spanspan_irr}
Suppose we have $\Omega \subset \proj(V)$ a domain, 
% with $\dim V = n+1 \geq 2$, 
and $\Gamma$ a non-elementary rank-one discrete subgroup of $\Aut(\Omega)$. 

The following are equivalent 
% (and in particular imply that $\Gamma$ is non-elementary):
\begin{enumerate}[(i)]
    \item \label{item:span} $\Lambda_\Gamma$ spans $\proj(V)$ and $\Lambda^*_{\Gamma}$ spans $\proj(V^*)$.
    \item \label{item:irred} $\Gamma$ is irreducible.%, i.e. the only $\Gamma$-invariant subspaces of $V$ are trivial.
    \item \label{item:strong irred} $\Gamma$ is strongly irreducible.%, i.e. all finite-index subgroups of $\Gamma$ are irreducible.
\end{enumerate}
\begin{proof}
Strong irreducibility implies {\it a fortiori} irreducibility, \ie \eqref{item:strong irred} implies \eqref{item:irred}.

\eqref{item:irred} implies \eqref{item:span} because $\Lambda_\Gamma$ (\resp $\Lambda^*_\Gamma$) is $\Gamma$-invariant (\resp $\Gamma^*$-invariant), and $\Gamma$ is irreducible if and only if $\Gamma^*$ is irreducible.

The limit set of any finite-index subgroup of $\Gamma$ is $\Lambda_\Gamma$, and its dual limit set is $\Lambda^*_\Gamma$. Thus, to establish that \eqref{item:span} implies \eqref{item:strong irred}, it is enough to prove that \eqref{item:span} implies \eqref{item:irred}.

Suppose we have a $\Gamma$-invariant subspace $W\subset V$. If $\proj(W)$ contains a point of $\Lambda_\Gamma$, then it contains them all by Proposition~\ref{prop:density of att-rep pairs}, and $W=V$ since $\Lambda_\Gamma$ spans $V$. Let us assume the contrary. Then  for any proximal element $\gamma\in\Gamma$, we have $\gamma^+\not\subset W$, and one can check that this implies that $W\subset\gamma_*^-$ (because $W$ is $\gamma$-invariant). In other words, $W \subset \bigcap_{H\in\Lambda^*_\Gamma} H$, and this right-hand side is trivial since $\Lambda^*_\Gamma$ spans $V^*$.
\end{proof} \end{prop} 

\subsection{Restricting and projecting properly convex domains} \label{sub:restrict_project}

As observed by Crampon--Marquis, Proposition~\ref{prop:spanspan_irr} has interesting consequences: in many cases, given a non-elementary rank-one group $\Gamma$, one can, by ``restricting to the spans of $\Lambda_\Gamma$ and $\Lambda^*_\Gamma$'', project $\Gamma$ onto a strongly irreducible rank-one group $\Gamma'$, and then try to pull back nice properties of $\Gamma'$ to nice properties of $\Gamma$. Let us formalize this idea.

Let $V=\Rnplusone$, and consider two subspaces $V_1,V_2\subset V$. Given a subgroup $G\leq\GL(V)$ (or $\PGL(V)$), we denote by $G^{V_1}$ (\resp $G^{V_1,V_2}
$) the set of elements of $G$ preserving $\proj(V_1)$ (\resp $\proj(V_1)$ and $\proj(V_2)$); we naturally identify $V_1/(V_1\cap V_2)$ (\resp $(V/V_2)^*$) as a subspace of $V/V_2$ (\resp $V^*$).
To produce the natural map from $\GL(V)^{V_1,V_2}$ to $\GL(V_1/(V_1\cap V_2))$ (\resp from $\GL(V)^{V_2}$ to $\GL((V/V_2)^*)$) one can equivalently go through $\GL(V_1)^{V_1\cap V_2}$ or $\GL(V/V_2)^{V_1/(V_1\cap V_2)}$ (\resp $\GL(V/V_2)$ or $\GL(V^*)^{(V/V_2)^*}$). 

We set, for any $\gamma\in\PGL(V)$ and any lift $\tilde\gamma \in\GL(V)$,
\begin{equation}\label{eq:kappa}
 \kappa(\gamma)=\|\tilde \gamma\|\cdot\|\tilde \gamma^{-1}\|,
\end{equation}
%\begin{equation}\label{eq:kappa}
% \kappa(g)=\|\tilde g\|\cdot\|\tilde g^{-1}\| \quad \text{and} \quad \ell(g)=\frac12 \log \frac{\lambda_1}{\lambda_{n+1}}(\tilde g),
%\end{equation}
where $\|\cdot\|$ is a fixed norm on $\mathrm{End(V)}$.%, and $\lambda_1(\tilde g)\geq \dots\geq \lambda_{n+1}(\tilde g)$ are the moduli of (complex) eigenvalues of $\tilde g$.

Consider a properly convex domain $\Omega\subset\proj(V)$. By an abuse of notation, we write $\Omega\cap V_1$ to denote $\Omega\cap\proj(V_1)$ and $\Omega/V_2$ to denote the projection of $\Omega$ in $\proj(V/V_2)$. Assume that $\Omega\cap V_1\neq\varnothing$, that $\overline\Omega\cap V_2=\varnothing$ (\ie $\Omega^*\cap (V/V_2)^*\neq\varnothing$), and that $(\Omega\cap V_1)/(V_1\cap V_2) = (\Omega/V_2)\cap (V_1/(V_1\cap V_2))$. Observe that $(\Omega/V_2)^*$ is naturally identified with $\Omega^*\cap (V/V_2)^*$. Denote by $\rho$ the natural map from $\Aut(\Omega)^{V_1,V_2}$ to $\Aut((\Omega\cap V_1)/(V_1\cap V_2))$. We will find a constant $C>0$ such that 
\begin{equation}\label{eq:rho proper}
 C^{-1}\kappa(\gamma) \leq \kappa(\rho(\gamma)) \leq C\kappa(\gamma)
\end{equation}
for any $\gamma\in\Aut(\Omega)^{V_1,V_2}$; in particular, this will imply that the map $\rho$ is proper. Using a duality argument, one checks without difficulty that it suffices to find $C>0$ such that $C^{-1}\kappa(\gamma) \leq \kappa(\rho_1(\gamma)) \leq C\kappa(\gamma)$ for any $\gamma\in\Aut(\Omega)^{V_1}$, where $\rho_1$ is the natural map from $\Aut(\Omega)^{V_1}$ to $\Aut(\Omega\cap V_1)$. This is an immediate consequence of the following 
% (which will be used one more time, to extend the definition of critical exponent).
\begin{prop}[{\cite[Prop.\,10.1]{DGK}}]\label{prop:kappa and Hilbert}
 For any pointed domain $x\in\Omega\subset\proj(V)$, there exists $C>0$ such that $C^{-1}\kappa(\gamma)\leq e^{2d_\Omega(x,\gamma x)}\leq C\kappa(\gamma)$ for any $\gamma\in\Aut(\Omega)$.
\end{prop}

Another useful formula will be the following: for any $\gamma\in\Aut(\Omega)^{V_1,V_2}$,
\begin{equation}\label{eq:ellrho is ell}
 \ell(\rho(\gamma))=\ell(\gamma),
\end{equation}
where $\ell(\gamma)$ is defined as in (\ref{eq:translength}).
To see this, it suffices, using again a duality argument, to fix $\gamma \in \Aut(\Omega)^{V_1}$ with $\ell(\gamma)>0$ and check that $\ell(\rho_1(\gamma))=\ell(\gamma)$. It is obvious that $\ell(\rho_1(\gamma))\leq \ell(\gamma)$. To establish the converse inequality, we only need to show that $\gamma^\pm \cap V_1$ are non-empty, where $\gamma^+$ (\resp $\gamma^-$) is the sum of all generalized eigenspaces associated to eigenvalues with maximal (\resp minimal) moduli. Observe that any accumulation point of $(\gamma^{\pm n} x)_{n}$, where $x\in\Omega\cap V_1$, belongs to $\gamma^\pm \cap \proj(V_1)$.

Let us apply the previous observations to a non-elementary rank-one discrete subgroup  $\Gamma\leq\Aut(\Omega)$. Then $V_1 := \Span(\Lambda_\Gamma)$ intersects $\Omega$, and $V_2':=\Span(\Lambda^*_\Gamma)$ intersects $\Omega^*\neq\varnothing$. Furthermore, $(\Omega\cap V_1)/(V_1\cap V_2)$ is equal to $(\Omega/V_2)\cap (V_1/(V_1\cap V_2))$, where $V_2=(V^*/V_2')^*$ is the intersection of all hyperplanes in $\Lambda^*_\Gamma$. Thus, we obtain by Proposition~\ref{prop:spanspan_irr}, and \eqref{eq:rho proper} and \eqref{eq:ellrho is ell}, the following

\begin{lem}\label{lem:proj rk1 gp}
If $\Gamma$ is a non-elementary rank-one discrete group preserving a domain $\Omega$ (i.e. $\Omega \subset \proj(V)$ is a domain and $\Gamma \leq \Aut(\Omega)$), then $\Gamma$ projects via a morphism $\rho$ with finite kernel onto a strongly irreducible rank-one discrete group preserving a domain $\Omega'$, with $\ell(\rho(\gamma))=\ell(\gamma)$ for any $\gamma\in\Gamma$.
\end{lem}

\subsection{The length spectrum of a non-elementary rank-one group}
\label{sub:lenspec}

Given a rank-one convex projective manifold $M=\Omega/\Gamma$, the {\bf length spectrum} of $\Gamma$ is the collection 
$ \left\{ \ell(\gamma) \:|\: \gamma \in \Gamma \right\} $. It contains the lengths of all closed rank-one Hilbert geodesics in $SM=S\Omega/\Gamma$, which are given by the set of $\ell(\gamma)$ as $\gamma$ varies over the rank-one elements of $\Gamma$. The following result says that the length spectrum of a non-elementary rank-one group is {\bf non-arithmetic}.

\begin{lem}\label{lem:nonarithm}
Let $\Omega\subset \RPn$ be a domain and $S\subset\Aut(\Omega)$ a sub-semigroup that generates a non-elementary rank-one discrete subgroup $\Gamma$. Then $\{\ell(s) \:|\: s\in S\}$ generates a dense subgroup of $\real$.
\end{lem}
\begin{proof}
 By Lemma~\ref{lem:proj rk1 gp}, the length spectrum of $S$ is the length spectrum of a semi-group $S'$ that generates a strongly irreducible group and preserves a domain. 
 Therefore the lemma follows immediately from \cite[Rem.\,p.\,165]{Benoist2000} and \cite[Cor.\,4.4]{Blayac_topmixing}.
\end{proof}

By the work of Dal'bo \cite{Dal99,Dal00} (on negatively curved Riemannian manifolds), we know that non-arithmeticity of the length spectrum is closely related to mixing of the geodesic flow and will be useful for establishing this property.

We can further deduce, from the next lemma, that the the \emph{local} length spectrum of a non-elementary rank-one group is also non-arithmetic.
\begin{lem}\label{lem:locnonarithm}
 Let $\Omega\subset \RPn$ be a domain and $\Gamma \leq \Aut(\Omega)$ a non-elementary rank-one discrete subgroup. Then the set of lengths of rank-one periodic orbits through any non-empty open subset $U\subset SM_{\bip}$ generates a dense subgroup of $\real$.
\end{lem}
\begin{proof}
Let $\widetilde{U}\subset S\Omega$ be the preimage of $U$. By Proposition~\ref{prop:density of att-rep pairs}, there exists a rank-one element $\gamma_1\in\Gamma$ whose axis encounters $\widetilde{U}$. Up to reducing $U$, we may assume that $d_{\spl}(\phi_{-\infty}v,\phi_\infty v)\geq 3$ for any $v\in\widetilde U$ (see \S\ref{subsec:bdry pts} for the definition of $d_{\spl}$). Since $\Gamma$ is non-elementary, we can find a rank-one element $\gamma_2\in\Gamma$ such that $\gamma_1^+$, $\gamma_1^-$, $\gamma_2^+$ and $\gamma_2^-$ are four distinct points. 

By \cite[Lem.\,4.7]{Blayac_topmixing}, we can find $N\geq 1$ such that $\gamma_1^N$ and $\gamma_2^N$ generate a non-elementary, rank-one free subgroup $\Gamma'\subset\Gamma$ whose elements are biproximal, and such that the axis of every element of the form $\gamma_1^N\gamma_{i_1}^N\cdots\gamma_{i_n}^N\gamma_1^N$ encounters $\widetilde U$, where $n\geq 0$ and $1\leq i_1,\dots,i_n\leq 2$.

Let $\Gamma^+\subset\Gamma'$ be the sub-semigroup generated by $\gamma_1^N$ and $\gamma_1^N\gamma_2^N\gamma_1^N$; it generates $\Gamma'$ as a group. By Lemma~\ref{lem:nonarithm}, the additive group generated by $\{\ell(\gamma) \:|\: \gamma\in\Gamma^+\}$, is dense in $\real$. But by construction, each non-trivial element $\gamma$ of $\Gamma^+$ is biproximal with axis \emph{through $\widetilde{U}$}; furthermore $d_{\spl}(\gamma^-,\gamma^+)\geq 3$, and this implies that $\gamma$ is rank-one by \cite[Fact\,2.14]{Blayac_PSdensities}. By definition the axis of each non-trivial element of $\Gamma^+$ projects on a rank-one periodic geodesic \emph{through $U$}, whose length is $\ell(\gamma)$.
\end{proof}

\begin{corn} \label{cor:rk1_top_mixing}
Let $\Omega\subset \RPn$ be a domain and $\Gamma\leq\Aut(\Omega)$ a non-elementary rank-one discrete subgroup. Then the geodesic flow on $SM_{\bip}$ is topologically mixing.
\end{corn}

\begin{proof}
It is easy to establish, when $\Gamma$ is rank-one and non-elementary, that $\Lambda_\Gamma$ is not contained in $\Span(\gamma^+, \gamma^\natural) \cup \Span(\gamma^-, \gamma^\natural)$ for any biproximal $\gamma\in\Gamma$, where $\gamma^\natural$ denotes the unique $\gamma$-invariant complement of $\Span(\gamma^+, \gamma^-)$. The desired corollary is then a consequence of \cite[Prop.\,6.1]{Blayac_topmixing}, Lemma~\ref{lem:locnonarithm}, and this aforementioned fact.
\end{proof}

\section{Patterson--Sullivan and (Bowen--Margulis--)Sullivan measures}\label{sec:ps}
\subsection{Conformal densities and Patterson--Sullivan measures} \label{subsec:pat_sul}
Fix a domain $\Omega$ and let $\Gamma \leq \Aut(\Omega)$ be a discrete subgroup.

A {\bf conformal density of dimension $\delta \geq 0$} (or a {\bf $\delta$-conformal dimension}, for short) on $\Omega$ is a function $\mu$ which associates to each $x \in \Omega$ a positive finite measure $\mu_x$ on $\del_{\hor} \Omega$, satisfying the property that $\mu_{x'}$ is absolutely continuous with respect to $\mu_x$ for all $x, x' \in X$, with Radon-Nikodym derivative given by
\[ \frac{d\mu_{x'}}{d\mu_x}(\xi) = e^{-\delta \beta_\xi(x',x)} \]
where $\beta_\xi$ denotes the horofunction based at $\xi \in \del_{\hor}\Omega$.
A density $\mu$ is said to be {\bf $\Gamma$-equivariant} if $\gamma_*\mu_x = \mu_{\gamma x}$ for all $\gamma \in \Gamma$ and all $x \in \Omega$.

The (convex projective) {\bf critical exponent} $\delta_\Gamma$ (also written $\delta(\Gamma)$ or just $\delta$ if the context is clear) of $\Gamma$ is the critical exponent of the Poincar\'e series 
% $g_{\Gamma,\Omega}(s,x) = g_\Gamma(s,x) :=$
$\sum_{\gamma \in \Gamma} e^{-s\cdot d_\Omega(x,\gamma x)}$, i.e. the infimum of all $s$ for which the series converges.
It is straightforward to check, using the triangle inequality, that the convergence of the Poincar\'e series, and hence the critical exponent, is well-defined independent of the choice of basepoint $x$; in fact, the convergence of the Poincar\'e series and $\delta_\Gamma$ do not depend on $\Omega$ (see Proposition~\ref{prop:kappa and Hilbert}).
The series may or may not converge at $s = \delta(\Gamma)$: if it does not, we say that $\Gamma$ is {\bf divergent}.

One can construct a $\Gamma$-equivariant conformal density of dimension $\delta(\Gamma)$ as follows: fix a basepoint $o \in \Gamma$. Given $s < \delta(\Gamma)$, define 
\[ \mu_{x,s} := \frac{1}{\sum_{\gamma \in \Gamma} e^{-s \cdot d_\Omega(\gamma \cdot o, x)}} \sum_{\gamma \in \Gamma} e^{-s \cdot d_\Omega(\gamma \cdot o, x)} \dirac_{\gamma \cdot o} \]
if $\Gamma$ is divergent, where we use $\dirac_x$ to denote the Dirac delta measure supported at $x$; if $\Gamma$ is not divergent, then set
\[ \mu_{x,s} := \frac{1}{\sum_{\gamma \in \Gamma} j(d_\Omega(\gamma \cdot o, x)) e^{-s \cdot d_\Omega(\gamma \cdot o, x)}} \sum_{\gamma \in \Gamma}  j(d_\Omega(\gamma \cdot o, x)) e^{-s \cdot d_\Omega(\gamma \cdot o, x)} \dirac_{\gamma \cdot o} \]
where $j: \real_+ \to \real_+$ is a suitable auxiliary function of subexponential growth (\ie $t\mapsto j(t)e^{-\eta t}$ is bounded for any $\eta > 0$), more precisely such that the modified Poincar\'e series $\sum_{\gamma \in \Gamma} j(d_\Omega(\gamma \cdot o, x)) e^{-s \cdot d_\Omega(\gamma \cdot o, x)}$ diverges at $s = \delta_\Gamma$ (for full details, see \cite{Patterson} or \cite{Sullivan}.)

Consider a weak* accumulation point $\mu_x$ of $\mu_{x,s}$ when $s$ tends to $\delta(\Gamma)$. Following the arguments in \cite{Patterson}, one may check that this yields a $\Gamma$-equivariant conformal density of dimension $\delta(\Gamma)$.

This construction was originally due to Patterson \cite{Patterson} and Sullivan \cite{Sullivan}, and these conformal densities are commonly known as Patterson--Sullivan densities, and the individual measures $\mu_x$ (for each fixed $x \in \Omega$) as Patterson--Sullivan measures.

We can consider our conformal densities to be families of measures on $\del\Omega$ by taking the push-forward by $\pi_{\hor}$ of a conformal density on $\del_{\hor}\Omega$; such conformal densities, measuring subsets of $\del\Omega$, remain $\Gamma$-quasi-invariant. As noted above, when $\Omega$ has $C^1$ boundary, $\del\Omega \cong \del_{\hor} \Omega$ and the push-forward conformal densities are identified with the conformal densities on $\del_{\hor}\Omega$.

\subsection{Shadows} \label{subsec:shadows}
One geometric way of understanding Patterson--Sullivan measures is provided by the Sullivan shadow lemma, which estimates the measure of a particular family of sets in terms of Hilbert distances.

Given $x, y \in \Omega $ and $r > 0$, define the shadow 
\[ \shadow_r(x,y) := \{\xi \in \del\Omega \st\, (x\,\xi) \cap \overline B_\Omega(y,r) \neq \varnothing \} \]
where $(x\,\xi)$ denotes the geodesic ray starting from $x$ in the direction of $\xi$. We may also take $x \in \del\Omega$, in which case $(x\,\xi)$ should be interpreted as the bi-infinite geodesic with endpoints $x$ and $\xi$.
The terminology comes from viewing $\shadow_r(x,y)$ as the shadow cast by the ball $\overline B_\Omega(y,r)$ on the boundary $\del \Omega$, when we have a light source located at the point $x$. 

We work here with \emph{closed} shadows, contrary to the common convention (adopted for instance in Roblin \cite[\S1B]{Roblin}) of working with \emph{open} shadows (cast by open balls): this makes almost no difference in practice since open shadows are contained in closed shadows, and contain closed shadows of any smaller radius.

The shadow lemma states, informally, that the sizes of suitably chosen shadows may be approximated in terms of distances between orbit points. Before presenting the shadow lemma, we note a lemma used in its proof which will also be useful later on; it is an immediate consequence of the definition of shadows and Proposition~\ref{fact:horofoliation}.
\begin{lem} \label{lem:roblin_12}
For all $\xi \in \pi_{\hor}^{-1} \left( \shadow_r(x,y) \right)$, we have \[ d_\Omega(x,y) - 2r < \beta_\xi(x,y) \leq d_\Omega(x,y) .\]
%\begin{proof}
%See \cite[Lem. 1.2]{Roblin}, or \cite[Lem. 4.7]{Bray} for a version in the context of Hilbert geometry.
%\end{proof}
\end{lem}

\begin{lem}[Sullivan shadow lemma, {\cite[Lem.\,8]{zhu_conf_densities}} or {\cite[Lem.\,4.2]{Blayac_PSdensities}}] \label{lem:sullivan_shadow}
Let $\Omega$ be a domain, $\Gamma\leq\Aut(\Omega)$ a non-elementary rank-one discrete subgroup, $\mu$ be a $\Gamma$-equivariant conformal density of dimension $\delta$ on $\Omega$, and $x \in \Omega$. Then, for all large enough $r > 0$, there exists $C > 0$ such that for all $\gamma \in \Gamma$,
\[ \frac1C e^{-\delta \cdot d_\Omega (x,\gamma x)} \leq \mu_x\left( \shadow_r(x,\gamma x) \right) \leq Ce^{-\delta \cdot d_\Omega(x,\gamma x)} .\]
\end{lem}

% We make one other observation about shadows which will be useful later\todo[color=P.-L.color!50!white,linecolor=black]{I think that we do not need it anymore}:
% \begin{lem} %\label{lem:roblin_1B}
% Suppose we have a finitely generated subgroup $\Gamma < \Aut(\Omega)$ and fixed $x \in \Omega$ and $r>0$ such that the shadow lemma applies. 

% There exists $M$ depending only on $\Gamma$, $x$ and $r$ such that for all $t \geq 1$, the family $S_t := \{\shadow_r(x, \gamma x) : \gamma\in\Gamma, t-1 < d_\Omega(x,\gamma x) \leq t \}$ covers some open subset of $\del\Omega$ with multiplicity bounded above by $M$.
% \begin{proof}
% Given any $\xi \in \del\Omega$, consider the geodesic ray $[x \xi)$. Since $\Gamma$ is finitely generated, there is a uniform bound $M$ on the number of orbit points $\gamma \cdot x$ within distance $r$ of any unit-length interval of this geodesic ray.
% \end{proof} \end{lem}

\subsection{The Hopf parametrization}\label{sec:hopf}

Let $\Omega\subset \RPn$ be a domain with basepoint $o\in\Omega$. The {\bf Hopf parametrization} based at $o$ is the continuous surjective map
 \[\Hopf_o\colon \pi_{\hor}^{-1}(\Geod^{\infty}\Omega)\times\real\longrightarrow S\Omega,\]
 that sends $(\xi,\eta,t)\in\pi_{\hor}^{-1}(\Geod^{\infty} \Omega)\times\real$ to the vector $\Hopf_o(\xi,\eta,t)$ which is tangent to the geodesic $(\pi_{\hor}(\xi)\;\pi_{\hor}(\eta))$ and such that $\beta_{\eta}(o,\pi\Hopf_o(\xi,\eta,t))=t$.

When the context is clear, we will simply write $\Hopf$ instead of $\Hopf_o$. Changing the base-point for $x\in\Omega$ yields the following formula: for any $(\xi,\eta,t)\in\pi_{\hor}^{-1}(\Geod^{\infty} \Omega)\times\real$,
\begin{equation*}
 \Hopf_x(\xi,\eta,t)=\Hopf_o(\xi,\eta,t+\beta_\eta(o,x)).
\end{equation*}

We can lift to $\pi_{\hor}^{-1}(\Geod^{\infty}\Omega) \times\real$ the three actions on the unit tangent bundle $S\Omega$ given by the geodesic flow, $\Aut(\Omega)$ and the flip involution $\flip$, so that the Hopf parametrization is equivariant; observe that, apart from the action of the geodesic flow, these actions depend on the base-point $o$: given $(\xi,\eta,t)\in \pi_{\hor}^{-1}(\Geod^{\infty}\Omega)\times\real$ and $s\in\real$ and $\gamma\in\Aut(\Omega)$ these lifts may be written as
\begin{align*}
    g^s(\xi,\eta,t)&=(\xi,\eta,t+s), \\
    \gamma \cdot(\xi,\eta,t)&=(\gamma\xi,\gamma\eta,t+\beta_\eta(\gamma^{-1}o,o)), \\
    \flip(\xi,\eta,t)&=(\eta,\xi, \langle \xi,\eta \rangle_o -t).
\end{align*}

If $\Omega$ is strictly convex with $C^1$ boundary, then $\pi_{\hor}^{-1}(\Geod^\infty\Omega)$ and $\Geod^\infty\Omega$ are both identified with $\partial^2\Omega$ (the set of pairs of distinct points of $\partial\Omega$). The Hopf parametrization is then a homeomorphism from $\partial^2\Omega\times\real$ to $S\Omega$.

\subsection{(Bowen--Margulis--)Sullivan measures}
Let $\Omega$ be a domain, $\Gamma\leq\Aut(\Omega)$ a discrete subgroup and $\mu$ a $\Gamma$-equivariant conformal density of dimension $\delta\geq 0$ on $\del_{\hor}\Omega$. We first define a measure $m^{\hor}$ on $\pi_{\hor}^{-1}(\Geod^{\infty}\Omega)\times\real$, by using the following formula of Sullivan \cite[Prop.\,11]{Sullivan}:
\[ dm^{\hor}(\xi,\eta,t) = e^{2\delta \langle \xi,\eta \rangle_x} d\mu_x(\xi) \, d\mu_x(\eta) \, dt ,\]
where $dt$ denotes (the infinitesimal form of) the Lebesgue measure on the $\real$ factor; $x \in S\Omega$ is an arbitrary base-point, and one can check that $m^{\hor}$ is independent of the choice of $x$. Under the assumption that $\Gamma$ is rank-one and non-elementary, one can verify that $m^{\hor}$ defines a non-zero Radon measure (since the function that sends $(\xi,\eta,x)\in\pi_{\hor}^{-1}(\Geod^\infty\Omega) \times\Omega$ to $\grop x\xi\eta$ is continuous).

One may also check that $m^{\hor}$ is invariant under the actions of the geodesic flow, $\Gamma$ and the flip involution. As a consequence, it induces a measure $m^{\hor}_\Gamma$ on the quotient $\pi_{\hor}^{-1}(\Geod^{\infty}\Omega) \times\real / \Gamma$. It also induces a $(g^t,\Gamma,\flip)$-invariant measure $m$ on the unit tangent bundle $S\Omega$ by simply pushing forward via the Hopf parametrization. Finally, this induces a $(g^t,\flip)$-invariant measure $m_\Gamma$ on $S\Omega/\Gamma$. We will call these measures ($m$ and $m_\Gamma$) {\bf Sullivan measures} associated to our conformal density $\mu$.

%The measure on the quotient space is really associated to the group action of $\Gamma$ on $\Omega$ and also depends on $\mu$; below, we will abuse terminology slightly and refer to this as a measure associated to $\Gamma$. Similarly we will write $m$ to denote the measure on $S\Omega$, although this also depends on the action of $\Gamma$ on $\Omega$ and on $\mu$.
% \todo[color=P.-L.color!50!white,linecolor=black]{These measures actually also depend on $\mu$ so if we keep this sentence, we have to mention it} \todo{But the dependence on $\mu$ is not exactly separate (at least if the conformal density comes from the Patterson--Sullivan construction) since $\mu$ also depends on $\Gamma \actson \Omega$?}

For cocompact irreducible $\Gamma$, this measure coincides with the Bowen--Margulis measure, which is the unique measure of maximal entropy of a topologically mixing Anosov flow \cite[Th.\,1.3]{Blayac_PSdensities}. We remark that Roblin \cite{Roblin} refers to $m_\Gamma$ as a Bowen--Margulis--Sullivan measure, and also that these $m_\Gamma$ are examples of what Paulin--Pollicott--Schapira \cite{PPS15} refer to $m_\Gamma$ as Gibbs measures (with zero Gibbs potential, and the Busemann cocycle).
% Notes from the literature: Roblin uses ``Bowen--Margulis--Sullivan (BMS)''; Ricks uses ``Bowen--Margulis''; Gabriele Link uses ``Ricks' Bowen--Margulis''

%strictly convex and smooth case
If $\Omega$ is strictly convex with $C^1$ boundary, then  we may identify $m^\hor$ with $m$ and $m^\hor_\Gamma$ with $m_\Gamma$ since we have identified $\pi_{\hor}^{-1}(\Geod^\infty\Omega)$ with $\partial^2\Omega$, and since the Hopf parametrization is a homeomorphism.

\subsection{The Hopf--Tsuji--Sullivan--Roblin dichotomy}

In this section we state a convex projective version of the celebrated Hopf--Tsuji--Sullivan--Roblin (HTSR) dichotomy, which establishes equivalences between several different ways of characterizing a subgroup of automorphisms as ``small'' or ``large'' in terms of conformal measures, associated Sullivan measures, and ergodicity of the geodesic flow.
It will be useful for us as a criterion for ergodicity of the Sullivan measure $m_\Gamma$ on $S\Omega/\Gamma$ under the action of the geodesic flow $(g^t_\Gamma)$, and as a tool for showing that our Patterson--Sullivan measures concentrate on the $C^1$ and strongly extremal points of $\del\Omega$.

% Let us note that the first version \cite{zhu_conf_densities} of the present paper, prior to \cite{Blayac_PSdensities} and only written by the second author, already contained the proofs of the divergent case Theorem~\ref{thm:HTSR}.\ref{item:divergent} below, under the additionnal assumptions that the domain is strictly convex with smooth boundary and that the action of $\Gamma$ on the domain $\Omega$ is geometrically finite REF: see \cite[Prop.\,12,14,16 and Cor.\,13]{zhu_conf_densities}.

%Other possible formulation of the previous paragraph
A proof of the divergent case of the HTSR dichotomy (Theorem~\ref{thm:HTSR}(\ref{item:divergent}) below), in the case where $\Omega$ is strictly convex with $C^1$ boundary and $\Gamma$ acts geometrically finitely on $\Omega$, can also be found in \cite[Prop.\,12,14,16 and Cor.\,13]{zhu_conf_densities}.

\begin{thm}[{\cite[Th.\,1.6]{Blayac_PSdensities}}]\label{thm:HTSR}
 Let $\Omega$ be a domain, and $\Gamma\leq\Aut(\Omega)$ be a non-elementary rank-one discrete subgroup. Let $o\in\Omega$ and $\delta\geq 0$, let $(\mu_x)_{x\in\Omega}$ be a $\Gamma$-equivariant conformal density of dimension $\delta$ on $\Omega$, and $m$ and $m_\Gamma$ be the induced Sullivan measures on $S\Omega$ and $SM = S\Omega/\Gamma$. Then there are two possibilities:
\begin{enumerate}
\item \label{item:convergent} either $\sum_{\gamma\in\Gamma} e^{-\delta \cdot d_\Omega(o,\gamma o)}<\infty$, and then $\mu_o(\Lambda_\Gamma^{\con})=0$ and the geodesic flow on $(SM,m_\Gamma)$ is dissipative and non-ergodic;
\item \label{item:divergent} or $\sum_{\gamma\in\Gamma} e^{-\delta \cdot d_\Omega(o,\gamma o)}=\infty$, in which case $\delta=\delta_\Gamma$, $\Gamma$ is divergent, and 
\begin{itemize}
 \item $(\mu_x)$ is the only $\delta_\Gamma$-conformal density on $\Omega$ (up to scaling); 
 \item $\mu_o$ is non-atomic, $\mu_o(\del_{\sse}\Omega\cap\Lambda_\Gamma \cap \Lambda_\Gamma^{\con})=1$ and $\supp(m)=SM_{\bip}$;
 \item the geodesic flow on $(SM,m_\Gamma)$ is conservative and ergodic;
\end{itemize}
\end{enumerate}
\end{thm}

Recall that a measurable flow $(\phi^t)_t$ is {\bf conservative} if $\{t \:|\: \phi^t(A)\cap A\neq\varnothing\}$ is unbounded for any set $A$ with positive measure, and {\bf dissipative} if one can cover the ambient space by countably many wandering measurable sets (\ie sets $A$ such that $\{t \:|\: \phi^t(A)\cap A\neq\varnothing\}$ is bounded). The only facts about conservativity and dissipativity that we need in this paper are that the two notions are mutually exclusive, and the Poincar\'e recurrence theorem, which states that any flow that preserves a finite measure is conservative. Below, our assumptions will always contain or imply that the Sullivan measure $m_\Gamma$ is finite, and hence that we are in the divergent case of the HTSR dichotomy. 

Within the divergent case of the dichotomy, the most important part for us is the fact that the Patterson--Sullivan measures are concentrated on $\del_{\sse}\Omega$. Since, by \cite[Lem.\,3.2]{Bray}, $\del_{\sse}\Omega$ can be identified with its preimage in the horoboundary $\del_{\hor}\Omega$, an immediate consequence is that from the point of view of the Patterson--Sullivan (\resp Sullivan) measures we deal with here, $\pi_{\hor}$ (\resp the Hopf parametrization) is a bijection, and the dynamics of the geodesic flow on $\left( (\del_{\hor}\Omega^2\times \real)/\Gamma,m^{\hor}_\Gamma \right)$ and on $(S\Omega/\Gamma,m_\Gamma)$ are the same.

\section{Mixing of the geodesic flow} \label{sec:mixing}

In this section, we state the measure-theoretic mixing result which will be a crucial ingredient in the equidistribution results which follow. 
\begin{defn} 
Given a finite-measure Borel space $(X,\nu)$, a flow $(g^t)_{t \in \real}$ on $X$ is {\bf (strongly) mixing} with respect to $\nu$ if for all $\varphi, \psi \in L^2(X,\nu)$,
\[ \int_X (\varphi \circ g^t) \cdot \psi \,d\nu 
\xrightarrow[t\to\pm\infty]{}
\frac{1}{\|\nu\|} \int_X \varphi \,d\nu \cdot \int_X \psi \,d\nu ,\]
or equivalently if for all Borel subsets $A, B \subset X$, we have
$$\nu(A \cap g^t B) \xrightarrow[t\to\pm\infty]{} \frac{\nu(A) \nu(B)}{\|\nu\|} .$$
\end{defn}

Mixing is a characteristic property of geodesic flows in negative curvature; one may view measure-theoretic mixing as a quantitative version of topological mixing.
% which measures
% how chaotic the dynamical systems given by the geodesic flows are. 

Measure-theoretic mixing results
have been proven in a wide range of settings where Sullivan measures may be defined, for instance for geometrically finite subgroups in constant negative curvature (see e.g. \cite{Rudolph}), or in great generality for all discrete groups of CAT$(-1)$ isometries with quotient admitting a finite Sullivan measure by Roblin \cite[Th.\ 3.1]{Roblin}. (See also \cite[\S3]{Sambarino_hypcvx} for related results about the mixing of Weyl chamber flows.)

\begin{thm}[{\cite[Th.\,18]{zhu_conf_densities}} or {\cite[Th.\,1.6]{Blayac_PSdensities}}] \label{thm:mixing}
% [cf. \cite{Roblin}, Th\'eor\`eme 3.1] 
Let $\Omega \subset \RPn$ be a domain and $\Gamma \leq\Aut(\Omega)$ be a non-elementary rank-one
discrete group such that $S \Omega / \Gamma$ admits a finite Sullivan measure $m_\Gamma$ 
associated to a $\Gamma$-equivariant $\delta_\Gamma$-conformal density.

Then the Hilbert geodesic flow $(g_\Gamma^t)_{t \in \real}$ on $S\Omega / \Gamma$ is mixing with respect to $m_\Gamma$.
\end{thm}

\begin{rmk}\label{rmk:mixing in the universal cover}
For the sake of clarity moving forward, we formulate the previous result here in terms of functions on the universal cover. Let $\hat{X}$ be either $S\Omega$ or $\pi_{\hor}^{-1}(\Geod^{\infty}\Omega)\times\real$, and let $m$ and $m_\Gamma$ denote the Sullivan measures on $\hat{X}$ and $\hat{X} / \Gamma$ respectively. For any measurable $m$-integrable (\resp non-negative) function $\phi$ on $\hat{X}$, we denote by $\int_\Gamma\phi$ the measurable $m_\Gamma$-integrable (\resp non-negative) function on $X = \hat{X}/\Gamma$ defined by
\begin{equation*}
    \int_\Gamma\phi(v):=\sum_{\gamma\in\Gamma}\phi(\gamma\tilde{v})
\end{equation*}
 for any $v\in X$, where $\tilde{v}\in \hat{X}$ is any lift of $v$. By the definition of $m_\Gamma$, we have
 \begin{equation*}
     \int_{\hat{X}}\phi\,dm=\int_{X}\left(\int_\Gamma\phi\right)\,dm_\Gamma.
 \end{equation*}
 If $\phi,\psi$ are integrable functions on $\hat{X}$ such that $|\phi|\int_\Gamma|\phi|$ and $|\psi|\int_\Gamma|\psi|$ are integrable on $\hat{X}$ (e.g.\ if $\phi$ and $\psi$ are bounded with compact support), then it is easy to see that 
 \begin{equation*}
     \left(\int_\Gamma\phi\right)\left(\int_\Gamma\psi\right) = \sum_{\gamma\in\Gamma}\int_\Gamma(\phi\cdot(\psi\circ\gamma)).
 \end{equation*}
 Therefore Theorem~\ref{thm:mixing} can be reformulated as
 \begin{equation*}
     \sum_{\gamma\in\Gamma}\int_{\hat{X}}\phi\cdot (\psi\circ\gamma\circ g^{t})\,dm \xrightarrow[t\to\infty]{} \frac{1}{\|m_\Gamma\|} \int_{\hat{X}}\phi\,dm \cdot \int_{\hat{X}}\psi\,dm.
 \end{equation*}
 In particular, if $A,B\subset \hat{X}$ are relatively compact Borel subsets, then
 \begin{equation*}
     \sum_{\gamma\in\Gamma}m(A\cap g^t\gamma B) \xrightarrow[t\to\infty]{} \frac{m(A)m(B)}{\|m_\Gamma\|}.
 \end{equation*}
 Also, if $\tilde\psi$ is a $\Gamma$-invariant function on $\hat{X}$ that lifts a square-integrable function $\psi$ on $X = \hat{X}/\Gamma$ and $\phi$ is an integrable function on $\hat{X}$ such that $|\phi|\int_\Gamma|\phi|$ is integrable, then
 \begin{equation*}
     \int_{\hat{X}}\phi\cdot (\tilde\psi\circ g^t)\,dm \xrightarrow[t\to\infty]{} \frac{1}{\|m_\Gamma\|}\int_{\hat{X}}\phi\,dm \cdot \int_{X}\psi\,dm_\Gamma.
 \end{equation*}
\end{rmk}

Theorem \ref{thm:mixing} was proven in the case of $\Omega$ strictly convex with $C^1$ boundary by the second author \cite[Th.\,18]{zhu_conf_densities}, and then in the more general properly convex rank-one case by the first author \cite[Th.\,1.6]{Blayac_PSdensities}. 

Very briefly, the proofs adapt arguments from \cite{Babillot} and \cite{Ricks}, using cross-ratios, length spectrums, and topological mixing. More specifically, it proceeds by showing that if $m_\Gamma$ is not mixing, then the length spectrum is contained in a discrete subgroup of $\real$; since this is not the case, $m$ must in fact be mixing. 

We refer the interested reader to \cite{Blayac_PSdensities} or \cite{zhu_conf_densities} for the details.
We remark that the proof in the case, when $\Gamma$ is not strongly irreducible reduces to the strongly irreducible case using the results in \S\ref{sub:restrict_project} and \S\ref{sub:lenspec} here.

% \subsection{Non-arithmeticity of length spectrum} \label{sub:lenspec}

% An important ingredient in the proof of mixing is the {\bf length spectrum} of $\Gamma$, i.e.\ the collection 
% $$ \left\{ \ell(\gamma) : \gamma \in \Gamma \right\} $$
% where $\ell(\gamma)$ is as defined in \S\ref{sub:restrict_project}. 
% % :=\frac12 \log \frac{\lambda_1}{\lambda_{n+1}}(\gamma)$. 
% We remark that this collection contains the lengths of all closed Hilbert geodesics in $S\Omega/\Gamma$, which are given by the set of $\ell(\gamma)$ as $\gamma$ varies over the rank-one elements of $\Gamma$ with axes in $\Omega$. 
% The following result says that the length spectrum of a non-elementary rank-one group is {\bf non-arithmetic}.
% \begin{lem}\label{lem:nonarithm}
% Let $\Omega\subset\proj(V)$ be a properly convex open set and $S\subset\Aut(\Omega)$ a sub-semigroup that generates a non-elementary, rank-one discrete subgroup $\Gamma$. Then $\{\ell(s) : s\in S\}$ generates a dense subgroup of $\real$.
% \end{lem}
% \begin{proof}
%  By Lemma~\ref{lem:proj rk1 gp}, the length spectrum of $S$ is the length spectrum of a semigroup $S'$ that generates a strongly irreducible group (and that preserves a properly convex domain). Therefore the lemma follows immediately from \cite[Rem.\,p.\,165]{Benoist2000} and \cite[Cor.\,4.4]{Blayac_topmixing}.
% \end{proof}

% The smaller set of lengths of closed rank-one Hilbert geodesics in $S\Omega/\Gamma$ should also generate a dense subgroup of $\real$, but we have not worked out the details there.

\subsection*{Equidistribution of unstable horospheres}

We further remark that Babillot obtains equidistribution of unstable horospheres as a consequence of mixing of the geodesic flow \cite[Th.\,3]{Babillot}, 
and we can do likewise here. 

Unstable horospheres of $S\Omega$ are sets of vectors tangent to geodesics backwards-asymptotic to a common point $\xi\in\del\Omega$, and with foot-points along  horospheres centered at a common preimage of $\xi$ in $\del_{\hor}\Omega$ (as described in \S\ref{sub:horothings}); if $\xi$ is $C^1$, then these are strong unstable manifolds for the Hilbert geodesic flow. The Hopf parametrization we adopted in \S\ref{sec:hopf} is convenient for parametrizing stable horospheres, but not unstable horospheres, and so we define the reversed Hopf parametrization $\Hopf^-:=\flip\circ \Hopf\circ \flip$ (recall that $\flip$ denotes the flip involution); using the reversed Hopf parametrization, we may write unstable horospheres of $\pi_{\hor}^{-1}(\Geod^\infty\Omega)\times \real$ as sets of the form $\{\xi\} \times J \times \{t\}$ with $\xi\in\del_{\hor}\Omega$ and $t\in\real$.

\begin{thm} \label{thm:horosphere_equidist}
Let $\Omega \subset \RPn$ be a domain and $\Gamma \leq \Aut(\Omega)$ be a non-elementary rank-one
discrete group such that $S \Omega / \Gamma$ admits a finite Sullivan measure $m_\Gamma$ associated to a $\Gamma$-invariant $\delta(\Gamma)$-conformal density $(\mu_x)$.

Fix a basepoint $o \in \Omega$. Let $J \subset \del_{\hor}\Omega$ be a closed subset with $\mu_o(J) \neq 0$, 
% non-zero $\mu_o$-measure, 
and $\xi_0 \in \supp(\mu_o)$ be such that $\{\xi_0\} \times J\subset \pi_{\hor}^{-1}(\Geod^\infty \Omega)$.

Then for any bounded and uniformly continuous function $\phi: S\Omega / \Gamma \to \real$, we have
\[ \frac{1}{c_J(\xi_0)} \int_J \tilde\phi(\xi_0,\eta,t) e^{-2\delta_\Gamma \grop{o}{\xi_0}\eta} d\mu_o(\eta) 
\xrightarrow[t\to\+\infty]{}
\frac{1}{\|m_\Gamma\|} \int_{S\Omega/\Gamma} \phi \,dm_\Gamma \]
where $c_J(\xi_0) := \int_J e^{-2\delta_\Gamma \grop{o}{\xi_0}\eta} d\mu_o(\eta)$ and the function $\tilde\phi$ is the $\Gamma$-invariant lift of $\phi$ to $\pi_{\hor}^{-1}(\Geod\Omega)\times\real$.
\end{thm}
Recall that if $\Omega$ is strictly convex with $C^1$ boundary, then the (reversed) Hopf parametrization is a homeomorphism. Therefore, in this case, $\pi_{\hor}^{-1}(\Geod\Omega)\times\real$, $\del^2\Omega\times\real$ and $S\Omega$ are all identified, the support of $\mu_0$ is $\Lambda_\Gamma$, and the condition $\xi_0\times J\subset \pi_{\hor}^{-1}(\Geod^\infty\Omega)$ in the above statement can be simply reformulated as $\xi_0\not\in J$.

\begin{proof}
We may suppose that $\phi$ is non-negative. Consider a compact neighborhood $I_0 \ni \xi_0$ sufficiently small such that $I_0 \times J\subset\pi_{\hor}^{-1}(\Geod^\infty\Omega)$. Then for any $\eps > 0$, we can choose a compact neighborhood $I \subset I_0$ of $\xi_0$ and $r >0$ such that 
\begin{enumerate}[(i)]
    \item $1-\eps \leq \frac{e^{-2\delta_\Gamma \grop o{\xi_0}\eta}}{e^{-2\delta_\Gamma \grop o\xi\eta}} \leq 1+\eps$ for any $(\xi,\eta) \in I \times J$, and
    \item $\left| \tilde\phi(\xi,\eta,t+s) - \tilde\phi(\xi_0,\eta,t) \right| \leq \eps$ for any $(\xi,\eta) \in I \times J$, $s \in [0,r]$, and $t > 0$.
\end{enumerate}
The second property holds since choosing $I$ and $r$ sufficiently small ensures that, for $\xi \in I$ and $s \in [0,r]$, the vectors $\Hopf^-(\xi,\eta,s)$ and $\Hopf^-(\xi_0,\eta,0)$ are uniformly (in $\eta \in J$) close. Moreover, since they both belong to the same weak stable leaf, given by the vectors pointing towards $\pi_{\hor}(\eta)$,
flowing them by the geodesic flow does not increase their distance \cite[Prop.\,5.1]{Blayac_topmixing}.

It then follows from properties (i) and (ii) that, for any fixed $t$, the integral
\[ \frac{1}{c_J(\xi_0)} \int_J \tilde\phi(\xi_0,\eta,t) e^{-2\delta_\Gamma\grop o{\xi_0}\eta} d\mu_o (\eta)\]
differs from
\begin{align*} 
& \frac{1}{c_J(\xi_0) r \mu_o(I)} \int_{I \times J \times [0,r]} \tilde\phi(\xi,\eta,t+s) e^{-2\delta_\Gamma\grop o\xi\eta} \,d\mu_o(\xi) \,d\mu_o(\eta) \,ds \\
& \quad = \frac{1}{c_J(\xi_0) r \mu_o(I)} \int_{I \times J \times [0,r]} \tilde\phi(\xi,\eta,t+s) \,dm^{\hor}
\end{align*}
by at most $\eps\left( 1 + \frac{2}{1-\eps} |\phi|_\infty \right)$.
By the mixing property of the geodesic flow and Remark~\ref{rmk:mixing in the universal cover} (which are still true if we replace $S\Omega$ by $\pi_{\hor}^{-1}(\Geod^\infty\Omega) \times\real$ since $\Hopf^-$ is a bijection on sets of full measure), we may choose $t$ large enough so that 
\[ \left|\int_{I \times J \times [0,r]} \tilde\phi(\xi,\eta,t+s) \,dm^{\hor} - \frac{m^{\hor}(I \times J \times [0,r])}{\|m_\Gamma\|} \cdot \int_{S\Omega/\Gamma} \phi \,dm_\Gamma \right| \leq \eps r \mu_o(I) ,\]
whence the conclusion follows, since
\[ \frac{m^{\hor}(I \times J \times [0,r])}{r\mu_o(I)} = \frac{1}{\mu_o(I)} \int_{I \times J} e^{-2\delta_\Gamma \grop o\xi\eta} \,d\mu_o(\xi) \,d\mu_o(\eta) \]
is bounded from below by $(1-\eps) \, c_J(\xi_0)$ and from above by $(1+\eps) \; c_J(\xi_0)$.
\end{proof}

\section{Equidistribution of group orbits} \label{sec:orbit_equidist}

In this section, following \cite[Th.\,4.1.1]{Roblin}, we prove an orbital equidistribution result, with consequences for orbital counting functions:
\begin{thm}%[cf. {\cite[Th. 4.1.1]{Roblin}}] 
\label{thm:orbit_equidist}
Let $\Omega\subset \RPn$ be a domain, and suppose $\Gamma < \Aut(\Omega)$ is a non-elementary rank-one discrete subgroup such that $S \Omega / \Gamma$ admits a finite Sullivan measure $m_\Gamma$ 
associated to a $\Gamma$-equivariant conformal density $\mu$ 
of dimension $\delta(\Gamma)$.
Then, for all $x, y \in \Omega$,  
\[ \delta \|m_\Gamma\| e^{-\delta t} \sum_{\substack{\gamma \in \Gamma\\ d_\Omega(x, \gamma y) \leq t}} \dirac_{\gamma y} \otimes \dirac_{\gamma^{-1} x} \xrightarrow[t\to+\infty]{} \mu_x \otimes \mu_y \] 
weakly in $C(\bar\Omega \times \bar\Omega)^*$.
\end{thm}

This has as immediate corollaries, by integrating in one or both factors,
\begin{cor}
$\displaystyle \delta \|m_\Gamma\| e^{-\delta t} \!\!\! \sum_{\substack{\gamma \in \Gamma \\ d_\Omega(x,\gamma y) \leq t}}  \dirac_{\gamma y} \xrightarrow[t\to\infty]{} \|\mu_y\| \, \mu_x$ weakly in $C(\bar\Omega)^*$. 
\end{cor}
\begin{cor}
$\# \left\{ \gamma \in \Gamma \st d_\Omega(x, \gamma y) \leq t \right\} \underset{t\to\infty}{\widesim} \frac{\|\mu_x\| \|\mu_y\|}{\delta \|m_\Gamma\|} e^{\delta t}$, \ie the ratio of the two sides goes to 1 as $t \to \infty$.
\end{cor}

The second corollary is most directly a counting result; the corollary before that 
is an equidistribution result for the Patterson--Sullivan measures. The theorem includes both of these statements, and is more directly related to the mixing of the Hilbert geodesic flow (Theorem \ref{thm:mixing}), as we shall see below in the proof.

These results are very much in the spirit of results first formulated by Margulis in the setting of closed manifolds of variable negative curvature \cite{Margulis} and subsequently extended and generalized to much more general settings with some trace of negative curvature. We refer the interested reader to \cite{Babsurvey} and the beginning of \cite[Chap.\,4]{Roblin} for a more extended discussion of this history. 
% Similar results, albeit using a slightly different notion of length, are known to hold for convex cocompact strictly convex projective holonomies by the work of Sambarino in \cite{Sambarino_cvx}, as discussed in \S\ref{sec:counting}.

% \subsection{Idea of proof} 
\begin{proof}[Proof of Theorem \ref{thm:orbit_equidist}]
% The proof follows that of \cite{Roblin}, Th\'eor\`eme 4.1.1, with minor corrections as noted in \cite{Link}, \S6 and \S8. We give a brief presentation of the proof here for completeness.

Let $\nu^t_{x,y}$ denote the measure $\delta \|m_\Gamma\| e^{-\delta t} \sum_{\gamma \in \Gamma} \dirac_{\gamma y} \otimes \dirac_{\gamma^{-1} x}$. To prove the desired convergence, we need to show that 
\[ \int_{\bar\Omega \times \bar\Omega} \varphi \,d\nu^t_{x,y} \xrightarrow[t\to\infty]{} \int_{\bar\Omega \times \bar\Omega} \varphi \,d(\mu_x \otimes \mu_y) \]
for all $\varphi \in C(\bar\Omega \times \bar\Omega)$.

Let us give a overview of the proof before plunging into the details. The proof uses mixing of the geodesic flow applied to suitable geometrically-described sets: given $x \in \Omega$, $A \subset \del\Omega$, and $r>0$, define
\begin{align*}
\cone_r^+(x,A) & := \left\{ y \in \Omega \st \exists x' \in B(x,r), \xi \in A : B(y, r) \,\cap\, (x'\xi) \neq \varnothing \right\} ,\\
\cone_r^-(x,A) & := \left\{ y \in \Omega \st B(y,r) \subset \bigcap_{x' \in B(x,r)} \; \bigcup_{\xi\in A} \; (x'\xi) \right\} .
\end{align*}
% \todo[color=P.-L.color!50!white,linecolor=black,inline]{Another way to write $\cone_r^+(x,A)$ is as the union over $x'\in B_\Omega(x,r)$ of the $r$-nbhd of $\cone(x',A)$, and I believe that this is contained in the $2r$-nbhd of $\cone(x,A)$, which is easier to define. Am I right? Should we take this new bigger set instead?}
% \todo[inline]{Think that's right but might take a little bit of work (the two $<r$ distances are in different places.) Also might be work to check that the proof works as written (or to do the necessary minor tweaks) with the larger set. So ... maybe, but right now seems like it may not be worth the effort.}

These may be thought of as expanded or contracted cones from $x$ to $A$, with the parameter $r$ controlling the expansion or contraction. We can use mixing to show that the $(\mu_x \otimes \mu_y)$-measures of  sufficiently small measurable sets $\bar{\mathcal{A}} \times \bar{\mathcal{B}} \subset \bar\Omega \times \bar\Omega$ are uniformly well-approximated by $\nu_{x,y}^t$-measures of corresponding products of cones over related sets $A$ and $B$. Here ``sufficiently small'' means ``contained in one of a system of neighborhoods $\hat{V} \times \hat{W} \subset \bar\Omega \times \bar\Omega$, one for each $(\xi_0,\eta_0) \in \del\Omega \times \del\Omega$.'' 

Using this, we can approximate
% , topologically and hence in measure, 
any sufficiently small Borel subset of $\overline\Omega \times \overline\Omega$ using products of such cones. From there, using standard arguments to approximate continuous positive functions using characteristic functions, we obtain the desired convergence of integrals if we replace the domain $\overline\Omega \times \overline\Omega$ with $\hat{V} \times \hat{W}$. We are then done by taking a finite subcover of the cover of the compact $\overline\Omega \times \overline\Omega$ by these neighbourhoods $\hat{V} \times \hat{W}$ and using a partition of unity subordinate to this subcover.

The technical crux of the proof is then the following 
% {\color{P.-L.color}(now I understand better the noun ``horcrux'' in Harry Potter)}
\begin{prop} \label{prop:1ere_etage}
In the setting of Theorem~\ref{thm:orbit_equidist}, fix $\eps > 0$, $\xi_0, \eta_0 \in \del\Omega$ extremal and $\mathcal C^1$, and $x, y \in \Omega$. Then there exist $R>0$ and open neighborhoods $V$ and $W$ of $\xi_0$ and $\eta_0$ (resp.) in $\del\Omega$, such that for all Borel subsets $A \subset V, B \subset W$, we have 
\begin{align*}
\limsup_{T\to+\infty} \nu_{x,y}^T (\cone_R^-(x,A) \times \cone_R^-(y,B)) & \leq e^\eps \mu_x(A) \, \mu_y(B) ,\\
\liminf_{T\to+\infty} \nu_{x,y}^T (\cone_R^+(x,A) \times \cone_R^+(y,B)) & \geq e^{-\eps} \mu_x(A) \, \mu_y(B) .
\end{align*}
\end{prop}

\begin{proof}
We proceed by estimating the $\nu_{x,y}^T$-measure of products of cones using several other geometric objects, all naturally equivariant under the isometries of $\Omega$: 
\begin{enumerate}[1.]
\item For $z \in \Omega$ and $(\xi,\eta) \in \Geod^\infty \Omega$,
let $z_{\xi\eta}$ denote the point of $S\Omega$ parallel to $(\eta\,\xi)$ (\ie determining a geodesic with forward endpoint $\xi$) with foot-point the middle point of the segment $(\eta\,\xi)\cap \overline{B}_\Omega(z,d_\Omega(z,(\eta\,\xi))$.
Given in addition $r > 0$ and $A \subset \del\Omega$, define 
% \[ K^+(z,r,A) = \left\{ g^s z_{\xi\eta} \st -\frac{r}2 < s < \frac{r}2, (\xi,\eta) \in \del^2\Omega, \eta \in A, d_\Omega(z,(\xi\eta)) < r \right\} .\]
% {\color{NSCcolor}
% \[ K^+(z,r,A) = \left\{ g^s z_{\xi\eta} \st -\frac{r}2 < s < \frac{r}2, (\xi,\eta) \in \Geod^\infty(\Omega), \eta \in A, d_\Omega(z,(\xi\eta)) < r \right\} .\]}

\[  K^+(z,r,A) := \left\{ g^s z_{\xi\eta} \st -\frac{r}2 \leq s \leq \frac{r}2, (\xi,\eta) \in \Geod^\infty \Omega, \eta \in A, d_\Omega(z,(\xi\,\eta)) \leq r \right\} .\]
We remark that when $\Omega$ is strictly convex with $C^1$ boundary, the foot-point is also the nearest-point projection of $z$ onto $(\xi\,\eta)$.

% Roughly, $K^+(z,r,A)$ is a smoothed set of points in $S\Omega$ which determine geodesics with forward endpoint in $A$ and which have foot-point $r$-close to $z$.
Inverting the role of $\xi$ and $\eta$ in the above definition yields $\flip  K^+(z,r,A) =:  K^-(z,r,A)$. We will also write $ K(z,r)$ to denote $ K^+(z,r,\del\Omega) \cup  K^-(z,r,\del\Omega)$. We remark that $ K(z,r) \subset S\overline B_\Omega(z,3r/2)$ by construction.

\item Given $r > 0$ and $a, b \in \Omega$ with $d_\Omega(a,b) > 2r$, we will consider the following enlarged and contracted shadows:
% \begin{align*}
% \shadow_r^+(a,b) & := \left\{ \xi \in \del\Omega \st \exists a' \in B(a,r):\, ]a'\xi) \cap B(b,r) \neq \varnothing \right\} \\
% \shadow_r^-(a,b) & := \left\{ \xi \in \del\Omega \st \forall a' \in B(a,r):\, ]a'\xi) \cap B(b,r) \neq \varnothing \right\}
% \end{align*}
\begin{align*}
{\shadow}_r^+(a,b) & := \left\{ \xi \in \del\Omega \st \exists a' \in \overline{B}_\Omega(a,r):\, (a'\xi) \cap \overline B_\Omega (b,r) \neq \varnothing \right\} ,\\
{\shadow}_r^-(a,b) & := \left\{ \xi \in \del\Omega \st \forall a' \in \overline B_\Omega(a,r):\, (a'\xi) \cap \overline B_\Omega(b,r) \neq \varnothing \right\} .
\end{align*}

Observe that, when $a \to \eta \in \del \Omega$, if $\eta$ is extremal then these variant shadows have as a common limit
(in the Hausdorff topology) 
\[ \shadow_r(\eta,b) = \left\{\xi \in \del\Omega \st (\eta\,\xi) \cap \overline B(b,r) \neq \varnothing \right\} =: \shadow_r^\pm(\eta,b) .\]

Note also that 
\begin{equation}\label{eq:shad+ in shad}
 \shadow_r^+(a,b)\subset\shadow_{2r}(a,b).
\end{equation}
Indeed if $a'\in\overline B_\Omega(a,r)$, $b'\in \overline B_\Omega(b,r)$ and $\xi \in \del\Omega$ are aligned in this order, then the point $b''\in [a\,\xi)$ at distance $d_\Omega(a',b')$ from $a$ is at distance at most $2r$ from $b$ since by Lemma~\ref{lem:crampon} it is at distance at most $r$ from $b'$.
% \item For $r > 0$ and $a, b \in \Omega$ with $d_\Omega(a,b) > 2r$, we denote by $\limshade_r(a,b)$ the set of $(\xi,\eta) \in \del^2 \Omega$ {\color{NSCcolor} $\Geod^\infty(\Omega)$} such that the geodesic $(\xi\eta)$ crosses first $B(a,r)$ and then $B(b,r)$. Observe that
% \[ \limshade_r(a,b) \subset \shadow_r^+(b,a) \times \shadow_r^+(a,b) .\]

\item For $r > 0$ and $a, b \in \Omega$ with $d_\Omega(a,b) > 2r$, we denote by $\limshade_r(a,b)$ the set of $(\xi,\eta) \in \Geod^\infty\Omega$ such that the geodesic $(\xi\eta)$ crosses first $\overline B_\Omega(a,r)$ and then $\overline B_\Omega(b,r)$. Observe that
\begin{equation}\label{eq:O-<L<O+}
\shadow_r^-(b,a) \times \shadow_r^-(a,b) \subset
\limshade_r(a,b) \subset \shadow_r^+(b,a) \times \shadow_r^+(a,b).
\end{equation}
\end{enumerate}

Now write $\eps'=\eps/100$. Choose $r \in \left( 0, \min\{1,\eps'/\delta\} \right)$ with $\mu_x(\del\shadow_r(\xi_0,x)) = 0 = \mu_y(\del\shadow_r(\eta_0,y))$. Note $\del\shadow_r$ is the boundary of $\shadow_r$ as a subset of $\del\Omega$.

\subsection{Stage 1}
We first prove the result for $x,y \in \Omega$ where $x \in(\xi_0 \xi_0')$ and $y \in (\eta_0 \eta_0')$ for some $\xi_0', \eta_0' \in \Lambda_\Gamma$ (and with $R=1$). Thus we have $\xi_0' \in \shadow_r(\xi_0,x)$, and similarly $\eta_0' \in \shadow_r(\eta_0,y)$; hence
\[ M:=r^2\, \mu_x(\shadow_r(\xi_0,x)) \, \mu_y(\shadow_r(\eta_0,y)) > 0 .\]
Take two open sets $\widehat{V}_1, \widehat{W}_1$ of $\overline\Omega$, containing $\xi_0, \eta_0$ respectively, and sufficiently small so that for all $a \in \widehat{V}_1, b \in \widehat{W}_1$, we have 
\begin{align} % \label{eq:where we need extremality}
\label{eq:measure of shadows1}
e^{-\eps'} \mu_x( \shadow_r(\xi_0,x)) \leq \mu_x( \shadow_r^\pm(a,x)) & \leq e^{\eps'} \mu_x ( \shadow_r(\xi_0,x)) ,\\
\label{eq:measure of shadows2}
e^{-\eps'} \mu_y( \shadow_r(\eta_0,y)) \leq \mu_y( \shadow_r^\pm(b,y)) & \leq e^{\eps'} \mu_y ( \shadow_r(\eta_0,y)).
\end{align}
%P.-L.: extra argument needed in the NSC case
This is possible because $\xi_0$ and $\eta_0$ are extremal. Indeed one can see that for any compact subset $K\subset\shadow_r(\xi_0,x)$ and any neighbourhood $U$ of $ \shadow_r(\xi_0,x) $, there exists a neighbourhood $\widehat{U}$ of $\xi_0$ in $\overline{\Omega}$ such that for any $a\in\widehat{U}$,
\begin{equation*}
    K\subset\shadow^-_r(a,x)\subset\shadow^+_r(a,x)\subset U.
\end{equation*}
Then, using again the fact that $\xi_0$ and $\eta_0$ are extremal, we can find two open sets $\widehat{V}, \widehat{W}$ of $\overline\Omega$, containing $\xi_0, \eta_0$ respectively, and sufficiently small so that for all $a \in \widehat{V}, b \in \widehat{W}$, we have $\overline{B}_{\overline\Omega}(a,1)\subset\widehat{V}_1$ and $\overline{B}_{\overline\Omega}(b,1)\subset\widehat{W}_1$. 

Choose open neighborhoods $V, W$ of $\xi_0, \eta_0$ (resp.) in $\del\Omega$, such that $\overline{V} \subset \widehat{V} \cap \del\Omega$ and $\overline{W} \subset \widehat{W} \cap \del\Omega$. In general, these will be the open neighborhoods we desire. 

Consider Borel subsets $A \subset V$ and $B \subset W$. Let $K^+ := K^+(x,r,A)$ and $K^- := K^-(y,r,B)$, and, for $T>0$,
\begin{equation*}
S^\pm_T := \{\gamma\in\Gamma \:|\: d_\Omega(x,\gamma y)\leq T, \ \gamma y\in \cone_1^\pm(x,A), \ \gamma^{-1}x\in\cone_1^\pm(y,B)\}.
\end{equation*}

We will estimate asymptotically, in two different ways, the quantity
\[ \int_0^T e^{\delta t} \sum_{\gamma \in \Gamma} m(K^+ \cap g^{-t} \gamma K^-) \,dt .\]
On the one hand, this can be estimated using our mixing result. On the other hand, we can obtain a different estimate by examining how the elements of $\Gamma$ contribute to the various parts of the integral; we will observe that the elements which contribute are, asymptotically as $T \to \infty$, those in $S^\pm_T$.

More precisely, we establish the following estimates for $T>0$ large enough:
\begin{align}
 \label{eq:crux1}
 & e^{\delta T} M \mu_x(A) \mu_y(B) 
 \leq e^{10\eps'}\delta \|m_\Gamma\| \int_0^{T-3r} e^{\delta t} \sum_{\gamma \in \Gamma} m( K^+ \cap g^{-t} \gamma  K^- ) dt +c_1 \\
 \label{eq:crux2}
 & \leq e^{(10+5)\eps'} \delta \|m_\Gamma\|  r^2 \sum_{\mathclap{\gamma \in S^+_T}} \mu_x( \shadow_r^+(\gamma y, x)) \mu_x( \shadow_r^+(x, \gamma y)) e^{\delta\cdot d_\Omega(x, \gamma y)} + c_1 + c_2 \\ 
 \label{eq:crux3}
 & \leq e^{(10+5+6)\eps'}  e^{\delta T} M \nu_{x,y}^T(\cone^+_1(x,A) \times \cone^+_1(y,B)) + c_1 + c_2 + c_3,
\end{align}
% % P.-L.: The following in an attempt to write the estimates backward. But it was failure because the formulas are too long
% \begin{comment}
% \begin{align}
%  \label{eq:crux1}
%  e^{\delta T} M \nu_{x,y}^T(\cone^+_1(x,A) \times \cone^+_1(y,B)) 
%  & \geq e^{6\eps'} \delta \|m_\Gamma\|  r^2 \sum_{\gamma \in S^+_T} \mu_x( \shadow_r^+(\gamma y, x)) \mu_x( \shadow_r^+(x, \gamma y)) e^{\delta\cdot d_\Omega(x, \gamma y)} \\
%  \label{eq:crux2}
%  & \geq e^{(6+?)\eps'}\delta \|m_\Gamma\| \int_0^{T-3r} e^{\delta t} \sum_{\gamma \in \Gamma} m( K^+ \cap g^{-t} \gamma  K^- ) dt - c_3 \\
%  \label{eq:crux3}
%  & \geq e^{\delta T} M \mu_x(A) \mu_y(B) - c_3 - c_4, 
% \end{align}
% \end{comment}
% % End of the attempt
and
\begin{align}
\label{eq:crux1'} & e^{\delta T} M \mu_x(A) \mu_y(B)
 \geq e^{-6\eps'} \delta \|m_\Gamma\| \int_0^{T+3r} e^{\delta t} \sum_{\gamma \in \Gamma} m( K^+ \cap g^{-t} \gamma  K^- ) dt - c_4 \\
 \label{eq:crux2'}
 & \geq e^{-(6+3)\eps'} \delta \|m_\Gamma\| r^2 \sum_{\mathclap{\gamma \in S^-_T}} \mu_x( \shadow_r^-(\gamma y, x)) \mu_x( \shadow_r^-(x, \gamma y)) e^{\delta\cdot d_\Omega(x, \gamma y)} - c_4 - c_5 \\ 
 \label{eq:crux3'}
 & \geq e^{-(6+3+2)\eps'} M e^{\delta T} \nu_{x,y}^T(\cone^-_1(x,A) \times \cone^-_1(y,B)) - c_4 - c_5 - c_6,
\end{align}
where $(c_i)_{1\leq i\leq 6}$ are constants independent of $T$.

\subsubsection{\eqref{eq:crux3} and \eqref{eq:crux3'}: shadows to cones} 
\eqref{eq:crux3} and \eqref{eq:crux3'} are consequences of the definition of $\nu_{x,y}^T$, of the conformality of $(\mu_z)_z$, and of \eqref{eq:measure of shadows1} and \eqref{eq:measure of shadows2}. Indeed, consider the following slight modification of Lemma \ref{lem:roblin_12}.
\begin{lem}[{\cite[Lem.\,4.3]{Blayac_PSdensities}}]\label{lem:estimate Busemann behind shadows}
For all $\xi \in \pi_{\hor}^{-1} \left( \shadow^+_r(x,y) \right)$, we have 
$$d_\Omega(x,y) - 4r \leq \beta_\xi(x,y) \leq d_\Omega(x,y).$$ 
\end{lem}

We apply this to obtain, by conformality of $(\mu_z)_z$, that, for any $\gamma\in\Gamma$,
\begin{align*}
\mu_y( \shadow_r^\pm(\gamma^{-1}x, y)) \leq \mu_x( \shadow_r^\pm(x,\gamma y)) e^{\delta \cdot d_\Omega(x,\gamma y)} \leq e^{4\eps'} \mu_y( \shadow_r^\pm(\gamma^{-1} x, y)).
\end{align*}

Denote $S^\pm=\bigcup_{T>0}S^\pm_T$ and $S:=\{\gamma\in\Gamma \st \gamma y\in\widehat{V}_1 , \gamma^{-1}x\in\widehat{W}_1\}$. By \eqref{eq:measure of shadows1} and \eqref{eq:measure of shadows2}, and using the definition of $\nu_{x,y}^T$ we obtain, on the one hand,
\begin{align*}
\sum_{\mathclap{\gamma \in S^+_T}} & \mu_x
 ( \shadow_r^+(\gamma y, x)) \, \mu_x( \shadow_r^+(x, \gamma y)) e^{\delta \cdot d(x, \gamma y)} \\
 & \leq e^{4\eps'}\sum_{\mathclap{\gamma \in S^+_T}} \mu_x( \shadow_r^+(\gamma y, x)) \, \mu_y( \shadow_r^+(\gamma^{-1}x,y))\\
& \leq e^{4\eps'} \sum_{\mathclap{\gamma \in S^+_T\cap S}} \mu_x( \shadow_r^+(\gamma y, x)) \, \mu_y( \shadow_r^+(\gamma^{-1}x,y)) + e^{4\eps'} \Vert\mu_x\Vert \cdot \Vert\mu_y\Vert \cdot |S^+\smallsetminus S| \\
& \leq e^{6\eps'} r^{-2}M\cdot |S^+_T\cap S| + c_3' 
 \leq e^{6\eps'} r^{-2}M\cdot |S^+_T| + c_3' \\
 &\leq e^{6\eps'} r^{-2}M\cdot \frac{e^{\delta T}}{\delta\Vert m_\Gamma \Vert} \nu_{x,y}^T(\cone^+(x,A)\times \cone_1^+(y,B)) + c_3',
\end{align*}
where one can check that $c_3':=e^{4\eps'} \Vert\mu_x\Vert \cdot \Vert\mu_y\Vert \cdot |S^+\smallsetminus S|$ is finite. Indeed, since $\overline V \subset \widehat{V}$ which is open in $\overline \Omega$, there must exist some $R>0$ such that for any $x'\in \overline{B}_{\overline\Omega}(x,1)$, $\xi\in\overline V$ and $z\in [x'\xi)$, if $d_\Omega(x,z)\geq R$, then $z\in \widehat{V}$; as a consequence, if $\gamma y\in\cone_1^+(x,A)\smallsetminus\widehat{V}_1$, then $d_\Omega(x,\gamma y)\leq R+1$. Similarly one can find $R'>0$ such that $d_\Omega(x,\gamma y)\leq R'$ whenever $\gamma^{-1}x\in\cone_1^+(y,B)\smallsetminus \widehat{W}_1$.

On the other hand,
\begin{align*}
\sum_{\mathclap{\gamma \in S^-_T}} \mu_x( \shadow_r^-(\gamma y, x)) &  \mu_x( \shadow_r^-(x, \gamma y)) e^{\delta d(x, \gamma y)} 
\geq \sum_{\mathclap{\gamma \in S^-_T\cap S}} \mu_x( \shadow_r^-(\gamma y, x)) \mu_y( \shadow_r^-(\gamma^{-1}x,y)) \\
& \geq e^{-2\eps'} r^{-2}M\cdot |S^-_T\cap S|  \\
& \geq e^{-2\eps'} r^{-2}M\cdot |S^-_T| - e^{2\eps'} r^{-2}M\cdot |S^-\smallsetminus S| \\
&= e^{-2\eps'} r^{-2}M\cdot \frac{e^{\delta T}}{\delta\Vert m_\Gamma \Vert} \nu_{x,y}^T(\cone_1^-(x,A)\times \cone_1^-(y,B)) - c_6',
\end{align*}
where one can check that $c_6':=e^{-2\eps'} r^{-2}M\cdot |S^-\smallsetminus S|$ is finite.

\subsubsection{ \eqref{eq:crux2}: geodesic corridors to shadows, upper bound} 

We may verify, by recalling the definitions of $m$ and $K^\pm$, that for $\gamma \in \Gamma$ with $d_\Omega(x, \gamma y) > 2r$, we have
\begin{equation}\label{eq:m(K+nK-)}
m( K^+ \cap g^{-t} \gamma  K^-) = \int \frac{d\mu_x(\xi) \,d\mu_x(\eta)}{e^{-2 \delta \grop{x}\xi\eta}} \int_{-r/2}^{r/2} \mathbf{1}_{ K(\gamma y,r)} (g^{t+s} x_{\xi\eta} ) \,ds 
\end{equation}
where the integral is supported on $ \limshade_r(x, \gamma y) \cap (\gamma B \times A)$. 

Then \eqref{eq:crux2} is a consequence of the following facts:
\begin{enumerate}[(i)]
\item the following non-negative number is finite:
\begin{align*}
    c_2' & := \int_0^\infty e^{\delta t} \sum_{d_\Omega(x,\gamma y)\leq 2r} m( K^+ \cap g^{-t} \gamma  K^- ) dt;
\end{align*}
indeed, if $\gamma\in\Gamma$ is such that $d_\Omega(x,\gamma y)\leq 2r$, then $ K^+ \cap g^{-t} \gamma  K^-$ is empty as soon as $t> 5r$;
\item
for $(\xi,\eta) \in  \limshade_r(x,\gamma y)$, $|s| \leq \frac r2$, and $T > 0$, we see, by examining the definition of $ K(\gamma y, r)$, that
\[ \int_0^{T-3r} e^{\delta t} \mathbf{1}_{ K(\gamma y,r)} (g^{t+s} x_{\xi\eta} ) dt \leq e^{3\delta r} r e^{\delta\cdot d_\Omega(x,\gamma y)} \leq e^{3\eps'} r e^{\delta\cdot d_\Omega(x,\gamma y)} \]
and also that this integral is zero if $d_\Omega(x,\gamma y) > T$;  

\item $e^{-2 \delta \grop{x}\xi\eta} \geq e^{-2\delta r} \geq e^{-2\eps'}$ for $(\xi,\eta) \in \limshade_r(x,\gamma y)$;
\item if $ \limshade_r(x, \gamma y) \cap (\gamma B \times A) \neq \varnothing$, then $\gamma y \in \cone_1^+(x,A)$ and $\gamma^{-1} x \in \cone_1^+(y,B)$ (we have taken care to ensure $r < 1$);

\item according to \eqref{eq:O-<L<O+}, $ \limshade_r(x,\gamma y) \subset  \shadow_r^+(\gamma y, x) \times  \shadow_r^+(x,\gamma y)$. 
\end{enumerate}

\subsubsection{ \eqref{eq:crux2'}: geodesic corridors to shadows, lower bound}

\eqref{eq:crux2'} follows from \eqref{eq:m(K+nK-)} together with the following facts:
\begin{enumerate}[(i)]
\item $e^{-\grop{x}\xi\eta} \leq 1$;
\item for $(\xi, \eta) \in  \limshade_r(x, \gamma y)$, $|s| \leq \frac r2$ and $T > 0$, we have
\[ \int_0^{T+3r} e^{\delta t} \mathbf{1}_{ K(\gamma y,r)} (g^{t+s} x_{\xi\eta} ) dt \geq e^{-3\eps'} r e^{\delta\cdot d_\Omega(x,\gamma y)} \]
if $3r \leq d_\Omega(x, \gamma y) \leq T$;
\item 
if $\gamma y \in \cone_1^-(x,A)$, then $A \supset  \shadow_r^-(x,\gamma y)$. Similarly, if $\gamma^{-1} x \in \cone_1^-(y,B)$, then $B \supset  \shadow_r^-(y, \gamma^{-1} x)$, i.e. $\gamma B \supset  \shadow_r^-(\gamma y, x)$. By \eqref{eq:O-<L<O+}, if both conditions are satisfied, then
%$(\gamma y, \gamma^{-1} x) \in \cone_1^-(x,A) \times \cone_1^-(y,B)$, then
%If $(\gamma y, \gamma^{-1} x) \in \cone_1^-(x,A) \times \cone_1^-(y,B)$, then $A \supset  \shadow_r^-(x,\gamma y)$ and $B \supset  \shadow_r^-(y, \gamma^{-1} x)$ i.e. $\gamma B \supset  \shadow_r^-(\gamma y, x)$. By \eqref{eq:O-<L<O+}, this yields 
\[  \limshade_r(x, \gamma y) \cap (\gamma B\times A)  \supset  \shadow_r^-(\gamma y,x)\times  \shadow_r^-(x,\gamma y);\]

\item the following non-negative number is finite:
\begin{align*}
c_5'&:=e^{-3\eps'} r^2\sum_{d_\Omega(x,\gamma y)\leq 3r} \mu_x(\shadow_r^-(\gamma y,x)) \, \mu_x(\shadow_r^-(x,\gamma y)) e^{\delta \cdot d_\Omega(x,\gamma y)}\\
& \leq e^{-3\eps'}r^2e^{3\delta r}\Vert\mu_x\Vert^2\#\{\gamma\in\Gamma \st d_\Omega(x,\gamma y)\leq 3r\}.
\end{align*}
\end{enumerate}

\subsubsection{\eqref{eq:crux1} and \eqref{eq:crux1'}: the mixing step}

Since the geodesic flow in the quotient is strongly mixing with respect to $m_\Gamma$ (Theorem \ref{thm:mixing}; see also Remark~\ref{rmk:mixing in the universal cover}), we have some $t_0 > 0$, such that for all $t > t_0$,
\begin{equation} e^{-\eps'} m(K^+) \, m(K^-) \leq \|m_\Gamma\| \sum_{\gamma \in \Gamma} m(K^+ \cap g^{-t} \gamma K^-)  \leq e^{\eps'} m(K^+) \, m(K^-) . \label{eqn:nsc411_mixing_closed} \end{equation}
Recalling the definition of $ K^+ =  K^+(x,r,A)$, we see that
\[ m( K^+) = r \int_A d\mu_x(\xi) \int_{ \shadow_r(\xi,x)} e^{\delta \grop{x}\xi\zeta} d\mu_x(\zeta) .\]
Since $0\leq\grop{x}\xi\zeta \leq r$ (from Corollary \ref{lem:grop_bound}) and $A \subset \widehat{V}_1$, by \eqref{eq:measure of shadows1} we obtain
\[e^{-\eps'} \mu_x( \shadow_r(\xi_0,x)) \leq \int_{ \shadow_r(\xi,x)} e^{2\delta \grop{x}\xi\zeta} d\mu_x(\zeta) \leq e^{3\eps'} \mu_x( \shadow_r(\xi_0,x)) \] 
and hence
\begin{equation} e^{-\eps'} r \, \mu_x(A) \, \mu_x(\shadow_r(\xi_0,x)) \leq m(K^+) \leq e^{3\eps'} r \, \mu_x(A) \, \mu_x(\shadow_r(\xi_0,x)) .\label{eqn:nsc411_K+_closed} \end{equation}

Arguing similarly with $ K^- =  K^-(y,r,B)$, we obtain
\begin{equation} e^{-\eps-} r \, \mu_y(B) \, \mu_y( \shadow_r(\eta_0,y)) \leq m(K^-) \leq e^{3\eps'} r \, \mu_y(B) \, \mu_y( \shadow_r(\eta_0,y)) . \label{eqn:nsc411_K-_closed} \end{equation}
Hence, using (\ref{eqn:nsc411_mixing_closed}) we have 
\begin{align*}
\resizebox{0.95\width}{!}{$\displaystyle \delta \|m_\Gamma\| \int_0^{T-3r} e^{\delta t} \sum_{\gamma \in \Gamma} m( K^+ \cap g^{-t} \gamma  K^- ) dt$}
& \geq \resizebox{0.95\width}{!}{$\displaystyle \delta\|m_\Gamma\| \int_{T_0}^{T-3r} e^{\delta t} \sum_{\gamma \in \Gamma} m( K^+ \cap g^{-t} \gamma  K^- ) dt$} \\ 
& \geq e^{-\eps'} m( K^+) m( K^-) \int_{T_0}^{T-3r} \delta  e^{\delta t}
dt \end{align*}
and so, by (\ref{eqn:nsc411_K+_closed}) and (\ref{eqn:nsc411_K-_closed}),
\begin{align*}
 \delta \|m_\Gamma\| \int_0^{T-3r} & e^{\delta t} \sum_{\gamma \in \Gamma} m(K^+ \cap g^{-t} \gamma  K^-) dt \\
% & \geq e^{-\eps'} m( K^+) m( K^-) \int_{T_0}^{T-3r} \delta  e^{\delta t} dt \\
& = e^{-\eps'} m(K^+) \, m(K^-) \, (e^{\delta T} e^{-3\delta r} - e^{T_0}) \\
& \geq e^{-4\eps'} e^{\delta T} m(K^+) \, m(K^-) - c_1' \\
& \geq e^{-6\eps'} e^{\delta T} r^2 \mu_x(A) \, \mu_y(B) \, \mu_x( \shadow_r(\xi_0,x)) \, \mu_x( \shadow_r(\eta_0,y))  - c_1' \\
& \geq e^{-6\eps'} e^{\delta T} M \mu_x(A) \, \mu_y(B) - c_1',
\end{align*} 
where $c_{1}' := e^{T_0} e^{-\eps/3} m(K^+) \, m(K^-)$.
Similarly, 
\begin{align*}
\delta \|m_\Gamma\| \int_0^{T+3r} & e^{\delta t} \sum_{\gamma \in \Gamma} m( K^+ \cap g^{-t} \gamma  K^- ) dt \\
& \leq \delta\|m_\Gamma\| \int_{T_0}^{T+3r} e^{\delta t} \sum_{\gamma \in \Gamma} m( K^+ \cap g^{-t} \gamma  K^- ) dt + c_4'\\ 
& \leq e^{\eps'} m(K^+) \, m(K^-) \int_{T_0}^{T+3r} \delta  e^{\delta t} + c_4'\\
& = e^{\eps'} m(K^+) \, m(K^-) \, (e^{\delta T} e^{3\delta r} - e^{T_0}) + c_4' \\
& \leq e^{4\eps'} e^{\delta T} m(K^+) \, m(K^-) + c_4' \\
% & \leq e^{10\eps'} e^{\delta T} r^2 \mu_x(A) \mu_x(B) \mu_x( \shadow_r(\xi_0,x)) \mu_x( \shadow_r(\eta_0,y)) + c_4' \\
& \leq e^{10\eps'} e^{\delta T} M \, \mu_x(A) \, \mu_y(B) + c_4'
\end{align*}
where $c_4':= \delta \|m_\Gamma\| \int_0^{T_0} e^{\delta t} \sum_{\gamma \in \Gamma} m( K^+ \cap g^{-t} \gamma  K^- ) dt$.

\subsection{Stage 2}
For more general $x,y \in \Omega$, 
% {\bf Second stage} --- 
% Let $x, y \in \Omega$. Given any $(\xi_0, \eta_0) \in \del\Omega \times \del\Omega$, there exists $r > 0$ and open neighborhoods $V$ and $W$ of $\xi_0$ and $\eta_0$ (resp.) in $\del\Omega$, such that for all Borel subsets $A \subset V, B \subset W$, we have, as $t \to +\infty$, 
% \begin{align*}
% \limsup \nu_{x,y}^t (\cone_r^-(x,A) \times \cone_r^-(y,B)) & \leq e^\eps \mu_x(A) \mu_y(B) \\
% \liminf \nu_{x,y}^t (\cone_r^+(x,A) \times \cone_r^+(y,B)) & \geq e^{-\eps} \mu_x(A) \mu_y(B)
% \end{align*}
given $\xi_0, \eta_0 \in \del\Omega$ $\mathcal C^1$ and extremal, choose $\zeta_0 \in \Lambda_\Gamma \smallsetminus \{\xi_0,\eta_0\}$ which is strongly extremal, and $x_0\in (\xi_0\zeta_0)$ and $y_0 \in (\eta_0\zeta_0)$. From the previous step we have neighborhoods $V_0, W_0$ of $\xi_0, \eta_0$ (respectively) such that the result of the lemma holds for $x_0$ and $y_0$ in the place of $x$ and $y$ and $V_0$ and $W_0$ in the place of $V$ and $W$.
% and with $\eps/3$ in the place of $\eps$ ... ?

Let then $\widehat{V_0}$ and $\widehat{W_0}$ be two open sets of $\overline\Omega$ containing respectively $\xi_0$ and $\eta_0$ such that $\widehat{V_0} \cap \del\Omega \subset V_0$ and $\widehat{W_0} \cap \del\Omega\subset W_0$, and for all $a \in \pi_{\hor}^{-1}\widehat{V_0}$ and $b \in \pi_{\hor}^{-1}\widehat{W_0}$, 
\begin{align*}
|\beta_a(x_0,x) - \beta_{\xi_0}(x_0,x)| & < \frac{\eps}{6\delta} \\
|\beta_b(y_0, y) - \beta_{\eta_0}(y_0,y)| & < \frac{\eps}{6\delta}.
\end{align*}

Set $R := 1 + \max\{d_\Omega(x,x_0), d_\Omega(y,y_0)\}$, and take a neighbourhood $\widehat{V}_1$ (\resp $\widehat{W}_1$) of $\xi_0$ (\resp $\eta_0$) such that $\overline{B}_{\overline\Omega}(z,R)$ is contained in $\widehat{V}_0$ (\resp $\widehat{W}_0$) for any $z$ in $\widehat{V}_1$ (\resp $\widehat{W}_1$).
Take two open neighborhoods $V$ and $W$ of $\xi_0$ and $\eta_0$ (respectively) in $\del\Omega$, such that $\overline{B}_{\overline\Omega}(\overline{V},R) \cap \del\Omega \subset \widehat{V_1} \cap \del\Omega$ and $\overline{B}_{\overline\Omega}(\overline{W},R) \cap \del\Omega \subset \widehat{W_1} \cap \del\Omega$. Consider $A \subset V$ and $B \subset W$. {\bf We now relate the orbit of $y$ seen from $x$ to that of $y_0$ seen from $x_0$}, thanks to the following observations:

First, using the definition of $R$ and $\cone_R^\pm$, one can easily verify that, for any $\gamma\in\Gamma$,
    \begin{gather*}
        (\gamma y,\gamma^{-1}x)\in \cone_R^-(x,A)\times\cone_R^-(y,B) \Rightarrow (\gamma y_0,\gamma^{-1} x_0) \in \cone_1^-(x_0,A)\times\cone_1^-(y_0,B), \\
        (\gamma y_0,\gamma^{-1}x_0)\in \cone_1^+(x_0,A)\times\cone_1^+(y_0,B) \Rightarrow (\gamma y,\gamma^{-1} x) \in \cone_R^+(x,A)\times\cone_R^+(y,B).
    \end{gather*}
    
Second, if $(\gamma y, \gamma^{-1}x) \in \widehat{V_1} \times \widehat{W_1}$, then $\gamma y_0 \in\widehat{V_0}$ and $\gamma^{-1}x \in \widehat{W_0}$, whence
    \begin{align*}
        d_\Omega(x_0, \gamma y_0) & = d_\Omega(x_,\gamma y_0) + \beta_{\gamma y_0}(x_0,x)\\
        & \leq d_\Omega(x, \gamma y_0) +  \beta_{\xi_0}(x_0, x) + \frac\eps{6\delta} \\
        & = d_\Omega(y_0, \gamma^{-1}x) + \beta_{\xi_0}(x_0,x) + \frac\eps{6\delta} \\
        & \leq d_\Omega(y, \gamma^{-1} x) + t_0 + \frac\eps{3\delta},
    \end{align*}
    where $t_0:=\beta_{\eta_0}(y_0,y) + \beta_{\xi_0}(x_0,x)$. Symmetrically, if $(\gamma y_0, \gamma^{-1}x_0) \in \widehat{V_1} \times \widehat{W_1}$, then
    \begin{equation*}
        d_\Omega(x_0,\gamma y_0)\geq d_\Omega(y, \gamma^{-1} x) + t_0 - \frac\eps{3\delta}.
    \end{equation*}

Third, the sets ($\cone_R^+(x,A)\cup \cone_R^+(x_0,A)) \smallsetminus\widehat{V_1}$ and $(\cone_r^+(y,B)\cup \cone_r^+(y_0,B))\smallsetminus\widehat{W_1}$ are bounded in $\Omega$ because any accumulation point of any sequence of $\cone_r^+(x,A)$ which diverges in $\Omega$ must belong to $\overline{B}(\overline A,r)\subset \widehat{V_1}$. Therefore
    \begin{gather*}
        \nu_{x,y}^t\left(\cone_R^-(x,A) \times \cone_R^-(y,B)\smallsetminus \widehat{V_1}\times\widehat{W_1}\right)\xrightarrow[t\to\infty]{}0, \\
        \nu_{x_0,y_0}^t\left(\cone_R^+(x_0,A) \times \cone_R^+(y_0,B)\smallsetminus \widehat{V_1}\times\widehat{W_1}\right)\xrightarrow[t\to\infty]{}0.
    \end{gather*}
%    \begin{gather*}
%        \nu_{x,y}^t(\cone_R^-(x,A) \times \cone_R^-(y,B)) - \nu_{x,y}^t(\cone_R^-(x,A) \times \cone_R^-(y,B)\cap \widehat{V_1}\times\widehat{W_1})\xrightarrow[t\to\infty]{}0, \\
%        \nu_{x_0,y_0}^t(\cone_R^+(x_0,A) \times \cone_R^+(y_0,B)) - \nu_{x_0,y_0}^t(\cone_R^+(x_0,A) \times \cone_R^+(y_0,B)\cap \widehat{V_1}\times\widehat{W_1})\xrightarrow[t\to\infty]{}0.
%    \end{gather*}
%    \begin{gather*}
%        \limsup_{t\to+\infty} \nu_{x,y}^t(\cone_R^-(x,A) \times \cone_R^-(y,B)) = \limsup_{t\to+\infty} \nu_{x,y}^t(\cone_R^-(x,A) \times \cone_R^-(y,B)\cap \widehat{V_1}\times\widehat{W_1}), \\
%        \liminf_{t\to+\infty} \nu_{x_0,y_0}^t(\cone_R^+(x_0,A) \times \cone_R^+(y_0,B)) = \liminf_{t\to+\infty} \nu_{x_0,y_0}^t(\cone_R^+(x_0,A) \times \cone_R^+(y_0,B)\cap \widehat{V_1}\times\widehat{W_1}).
%    \end{gather*}
i.e. when we are taking the limits of the measures of $\cone_R^-(x,A) \times \cone_R^-(y,B)$ and $\cone_R^+(x_0,A) \times \cone_R^+(y_0,B)$, we may restrict to looking at points inside $\widehat{V_1} \times \widehat{W_1}$, where the distance estimates from the previous observation hold.

Together, these observations imply that
\begin{align*}
    \resizebox{0.96\width}{!}{$\displaystyle \limsup_{t\to+\infty} \nu_{x,y}^t(\cone_R^-(x,A) \times \cone_R^-(y,B))$}
    & \leq \resizebox{0.96\width}{!}{$\displaystyle e^{\delta \cdot (t_0 + \frac{\eps}{3\delta})} \limsup_{t\to+\infty} \nu_{x_0,y_0}^{t + t_0 + \frac\eps{3\delta}} \left( \cone_1^-(x_0,A) \times \cone_1^-(y_0,B) \right) $} \\
    & \leq e^{2\eps/3} e^{\delta t_0} \mu_{x_0}(A) \, \mu_{y_0}(B) \\
    & \leq e^{2\eps/3} (e^{\delta \cdot (\beta_{\xi_0}(x_0,x)}\mu_{x_0}(A)) \, (e^{\beta_{\eta_0}(y_0,y))}\mu_{y_0}(B)) \\
    & \leq e^\eps \mu_x(A) \, \mu_y(B),
\end{align*}
where we have used the conformality of the $(\mu_z)_z$ to say, for example, that $\frac{d\mu_{x_0}}{d\mu_x}(\xi)$ is bounded above by $e^{-\delta\beta_{\xi_0}(x_0,x)+\frac\eps6}$ for any $\xi\in A$.

Similarly one shows that 
\begin{align*}
\resizebox{0.97\width}{!}{$\displaystyle \liminf_{t\to+\infty} \nu_{x,y}^t 
    (\cone_R^+(x,A)  \times \cone_R^+(y,B))$}
    & \geq 
\resizebox{0.97\width}{!}{$\displaystyle e^{\delta \left(t_0 - \frac{\eps}{3\delta} \right)} \liminf_{t\to+\infty} \nu_{x_0,y_0}^{t + t_0 - \frac\eps{3\delta}} \left( \cone_1^+(x_0,A) \times \cone_1^+(y_0,B) \right)$} \\
    & \geq e^{-2\eps/3} e^{\delta t_0} \mu_{x_0}(A)\, \mu_{y_0}(B) \\
%    & \geq e^{-2\eps/3} (e^{\delta \cdot (\beta_{\xi_0}(x_0,x)}\mu_{x_0}(A))\cdot(e^{\beta_{\eta_0}(y_0,y))}\mu_{y_0}(B)) \\
    & \geq e^{-\eps} \mu_x(A)\,\mu_y(B).
\end{align*}

This concludes the proof of Proposition~\ref{prop:1ere_etage}.
\end{proof}

\subsection{Negligibility from lack of extremal points}

Here we pause to prove two lemmas which will together be used in the next step of the proof.
The first shows that a sum of Dirac masses at orbit points is bounded above by a Patterson--Sullivan measure, and hence has negligible mass away from extremal points. The second establishes that the differences between measurable subsets of $\overline\Omega$ and certain associated cones do not contain extremal points.

Given $\Omega\subset \RPn$ a domain and $\Gamma\leq\Aut(\Omega)$ a discrete subgroup, $x,y\in\Omega$ and $t\geq 0$, we write $\alpha_{x,y}^t$ to denote the measure appearing in the first corollary to Theorem \ref{thm:orbit_equidist}, \ie
 \begin{equation*}
  \alpha_{x,y}^t:= \sum_{\substack{\gamma\in\Gamma \\ d_\Omega(x,\gamma y)\leq t}} \dirac_{\gamma y}.
 \end{equation*}
 If $M=\Omega/\Gamma$ is non-elementary and rank-one, then the Sullivan shadow lemma (Lemma~\ref{lem:sullivan_shadow}) implies that $(e^{-\delta_\Gamma t}\alpha_{x,y}^t(\overline\Omega))_t$ is bounded, see e.\,g.\ \cite[Prop.\,4.7]{Blayac_PSdensities}.
%  (one may also find the proof hidden in that of Lemma~\ref{lem:alpha<mu} below). 
The following lemma is a refinement of this idea:
 
 \begin{lem}\label{lem:alpha<mu}
  Let $\Omega\subset \RPn$ be a domain, $\Gamma\leq\Aut(\Omega)$ a divergent non-elementary rank-one discrete subgroup, and $(\mu_x)$ a conformal density of dimension $\delta=\delta_\Gamma$ on $\Omega$.
  Then there exists $C>0$ such that $
   \alpha\leq C\mu_x$ for any accumulation point $\alpha$ of $(\alpha_{x,y}^t)_{t\to\infty}$.
   
  In particular, if $K\subset\overline{\Omega}$ is compact and does not contain any extremal point, then
  \begin{equation*}
      \alpha_{x,y}^t(K) \xrightarrow[t\to\infty]{} 0.
  \end{equation*}
 
 \begin{proof}
  Since $\del_{\sse}\Omega$ has full $\mu_x$-measure by Theorem~\ref{thm:HTSR}, and by the interior regularity of finite measures, it is enough to find $C>0$ and $R>0$ such that for any compact subset $K\subset \overline{\Omega}$, 
  \begin{equation*}
   \alpha(K)\leq C\mu_x(\overline{B}_{\overline{\Omega}}(K,R)).
  \end{equation*}
 
  By Lemma~\ref{lem:sullivan_shadow}, there exist $R>0$ and $C_1>0$ such that for any $\gamma\in\Gamma$,
  \begin{equation*}
   \mu_x(\shadow_R(x,\gamma y))\geq C_1^{-1}e^{-\delta d_\Omega(x,\gamma y)}.
  \end{equation*}
  Fix $\epsilon>0$. Let $U$ be a neighbourhood in $\overline{\Omega}$ of $\overline{B}_{\overline{\Omega}}(K,R)$ such that $\mu_x(U)\leq \mu_x(\overline{B}_{\overline{\Omega}}(K,R))+\epsilon$. Let $V$ be a neighbourhood in $\overline{\Omega}$ of $K$ such that $\shadow_R(x,z)\subset U$ for any $z\in V\cap\Omega$. Observe that for all $t\geq 0$ and $\xi\in\del\Omega$,
  \begin{equation}\label{Equation : intersection d ombres}
   \#\{\gamma \st t-1\leq d_\Omega(x,\gamma y)\leq t ,\  \xi\in\shadow_R(x,\gamma y)\}\leq \#\{g \st d_\Omega(y,g y)\leq 4R+1\}.
  \end{equation}
  Indeed, if $\gamma,g\in\Gamma$ are such that $t-1\leq d_\Omega(x,\gamma y),d_\Omega(x,\gamma g y)\leq t$ and $\xi\in\shadow_R(x,\gamma y)\cap\shadow_R(x,\gamma g y)$, then let $y_1\in [x\,\xi)\cap B_\Omega(\gamma y,R)$ and $y_2\in[x\,\xi)\cap B_\Omega(\gamma g y,R)$, so that
  \begin{align*}
   d_\Omega(y,gy) & = d_\Omega(\gamma y,\gamma g y)
   \leq R + d_\Omega(y_1,y_2) + R \\
   & = 2R + | d_\Omega(x,y_1) - d_\Omega(x,y_2) | \\
   & \leq 4R + | d_\Omega(x,\gamma y) - d_\Omega(x,\gamma g y) |
   \leq 4R+1.
  \end{align*}

Then we have
\begin{align*}
   \alpha(K) & \leq \limsup_{t\to\infty} e^{-\delta t}\alpha_{x,y}^t(V) \\
   & := \limsup_{t\to\infty} e^{-\delta t} \# \{\gamma \st 0\leq d_\Omega(x,\gamma y)\leq t, \gamma y\in V\} \\
   & \leq e^{\delta} \limsup_{t\to\infty} \! \sum_{1\leq k\leq \lfloor t+1\rfloor} e^{\delta (k-\lfloor t+1\rfloor)}e^{-\delta k} \# \{\gamma \st k-1\leq d_\Omega(x,\gamma y)\leq k, \gamma y\in V\} \\
   & \leq C_1e^{\delta} \limsup_{n\to\infty} \sum_{1\leq k\leq n} e^{\delta (k-n)} \!\!\! \sum_{\substack{k-1\leq d_\Omega(x,\gamma y)\leq k\\\gamma y\in V}} \!\!\!\!\! \mu_x(\shadow_R(x,\gamma y)).
\end{align*}
Now recall that
$\shadow_R(x,\gamma y)\subset U$ for any $\gamma y \in V$. Hence by \eqref{Equation : intersection d ombres} we have
\begin{align*}
\alpha(K) & \leq C_1e^\delta \limsup_{n\to\infty} \sum_{1\leq k\leq n} e^{\delta (k-n)} \int_U  \sum_{\substack{k-1\leq d_\Omega(x,\gamma y)\leq k\\\gamma y\in V}} \!\!\! \mathbf{1}_{\shadow_R(x,\gamma y)}(\xi) \;d\mu_x(\xi) \\
   & \leq C_1e^\delta \limsup_{n\to\infty} \sum_{1\leq k\leq n} e^{\delta (k-n)} \#\{\gamma \st d_\Omega(y,\gamma y)\leq 4R+1\} \; \mu_x(U) \\
   & \leq C\mu_x(\overline{B}_{\overline{\Omega}}(K,R))+C\epsilon,
\end{align*}
where $C:=\frac{C_1 e^\delta }{1-e^{-\delta}} \cdot \#\{\gamma \:|\: d_\Omega(y,\gamma y)\leq 4R+1\}$. The previous estimates hold for any $\epsilon>0$, and therefore we have
$\alpha(K)\leq C \mu_x\left( \overline{B}_{\overline{\Omega}}(K,R) \right)
 $ as desired. \end{proof}
\end{lem}

This lemma will be useful below in combination with the next one, which shows that certain sets we will want to have small measure do not contain extremal points: 
\begin{lem}\label{lem:no extremal point}
Let $\Omega\subset \RPn$ be a domain, $r>0$, $x\in\Omega$ and $\Acal\subset \overline{\Omega}$ be measurable. 
Then for any open neighborhood $A^+$ of $\overline{\Acal}\cap\del\Omega$ in $\del\Omega$, and any compact subset $A^- \subset \interior(\Acal)\cap\del\Omega$, $\overline{\Acal\cap\Omega\smallsetminus \cone_r^-(x,A^+)}$ and $\overline{\cone_r^+(x,A^-)\smallsetminus\Acal}$ do not contain any extremal point of $\overline{\Omega}$.
\end{lem}

\begin{proof}
Consider a sequence $(y_n)_n$ of points in ${\Acal\cap\Omega\smallsetminus \cone_r^-(x,A^+)}$ that converges to $y\in\del{\Omega}\cap\overline{\Acal}$; we show that $y$ is not extremal. By definition, for each $n$ there exist $x_n\in \overline{B}_\Omega(x,r)$ and $z_n\in \overline{B}_{\Omega}(y_n,r)\smallsetminus \cone_r(x_n,A^+)$. Up to subsequence, we can assume that $x_n\neq z_n$ for any $n$ and that $(x_n)_n$ and $(z_n)_n$ converge respectively to $x'\in\overline{B}_\Omega(x,r)$ and $z\in\overline{B}_{\overline{\Omega}}(y,r)$. 
For each pair $(a,b)\in\Omega\times\overline{\Omega}$ such that $a\neq b$, denote by $\shadow(a,b)$ the unique point $c \in\del\Omega$ such that $b\in[a\,c]$. The map $\shadow$ is continuous, so $\shadow^{-1}(\del\Omega\smallsetminus A^+)$ is closed, hence it must contain $(x',z)$, since it contains $\{(x_n,z_n)\}_n$.
Since $z\in\del\Omega$, we have that $\shadow(x',z)=z\not\in A^+\ni y$. Hence $z \in \overline{B}_{\overline{\Omega}}(y,r) \smallsetminus \{y\}$, and so $y$ is not extremal.

Consider a sequence $(y_n)_n$ of points in $\cone_r^+(x,A^-)\smallsetminus\Acal$ that converges to $y\in\del\Omega\smallsetminus\interior(\Acal)$; we show that $y$ is not extremal. By definition, for each $n$ there exist $x_n\in\overline{B}_\Omega(x,r)$ and $z_n\in \overline{B}_\Omega(y_n,r)\smallsetminus\{x_n\}$ such that $\shadow(x_n,z_n)\in A^-$. Up to subsequence, $(x_n)_n$ and $(z_n)_n$ converge respectively to $x'\in\overline{B}_\Omega(x,r)$ and $z\in \overline{B}_{\overline{\Omega}}(y,r)$. By continuity, $z=\shadow(x',z)$ is in the closed set $A^-$. Thus, $z \in \overline{B}_{\overline{\Omega}}(y,r) \smallsetminus \{y\}$, and hence $y$ is not extremal.
\end{proof}
%P.-L.: end

\subsection{End of proof} 
\begin{prop}\label{prop:3eme etape}
In the setting of Theorem~\ref{thm:orbit_equidist}, for any $x,y\in\Omega$, $\eps>0$, and $\mathcal C^1$ and extremal $\xi_0,\eta_0 \in \del\Omega$, there exist open neighborhoods $\widehat{V}$ and $\widehat{W}$ of $\xi_0$ and $\eta_0$ in $\overline{\Omega}$ such that for any non-negative function $\varphi$ supported on $\widehat{V}\times\widehat{W}$,
\[ e^{-\eps} \int \varphi \,d(\mu_x \otimes \mu_y) \leq \liminf \int \varphi \,d\nu_{x,y}^t \leq \limsup \int \varphi \,d\nu_{x,y}^t \leq e^\eps \int \varphi \,d(\mu_x \otimes \mu_y) .\]
\end{prop}

\begin{proof}
Let $\widehat{V}$ and $\widehat{W}$ be the open neighborhoods given by Proposition~\ref{prop:1ere_etage}. It is enough to prove that for all measurable subsets $\Acal$ and $\Bcal$ such that $\overline{\Acal}\subset\widehat{V}$, $\overline{\Bcal}\subset\widehat{W}$ and $\mu_x(\del\Acal)=\mu_y(\del\Bcal)=0$, we have \begin{align*}
e^{-\eps} \mu_x(\mathcal{A}) \mu_y(\mathcal{B})\leq \liminf \nu_{x,y}^t(\mathcal{A} \times \mathcal{B})\leq \limsup \nu_{x,y}^t(\mathcal{A} \times \mathcal{B}) \leq e^\eps \mu_x(\mathcal{A}) \mu_y(\mathcal{B}).
\end{align*} 
Let $\eps'>0$. Let $A^+$ (\resp $B^+$) be an open neighborhood of $\Acal\cap\del\Omega$ (\resp $\Bcal\cap\del\Omega$) in $\del\Omega$, and $A^-\subset \interior(\Acal)\cap\del\Omega$ (\resp $B^-\subset\interior(\Bcal)\cap\del\Omega$) be compact such that 
\begin{equation*}
    \mu_x(\interior(\Acal)\smallsetminus A^-)+\mu_x(A^+\smallsetminus\overline{\Acal})+\mu_y(\interior(\Bcal)\smallsetminus B^-)+\mu_y(B^+\smallsetminus\overline{\Bcal})<\eps'.
\end{equation*}

By Lemma \ref{lem:alpha<mu} and Lemma~\ref{lem:no extremal point}, 
\begin{align*}
e^{-\delta t}\alpha_{x,y}^t(\Acal\smallsetminus\cone_r^-(x,A^+)), \  & e^{-\delta t}\alpha_{x,y}^t(\cone_r^+(x,A^-)\smallsetminus\Acal),\\ e^{-\delta t}\alpha_{y,x}^t(\Acal\smallsetminus\cone_r^-(y,B^+)), & \text{ and } e^{-\delta t}\alpha_{y,x}^t(\cone_r^+(y,B^-)\smallsetminus\Bcal)
\end{align*}
all converge to zero as $t \to \infty$. Since the projections of $\nu_{x,y}^t$ to the first and second coordinates are, respectively, $\delta\Vert m_\Gamma\Vert e^{-\delta t}\alpha_{x,y}^t$ and $\delta\Vert m_\Gamma\Vert e^{-\delta t}\alpha_{y,x}^t$, we have that
 \begin{align*}
     \limsup \nu_{x,y}^t(\Acal\times \Bcal) &\leq \limsup \nu_{x,y}^t(\cone_r^-(x,A^+)\times\cone_r^-(y,B^+)) \\
     \liminf \nu_{x,y}^t(\Acal\times \Bcal) &\geq \liminf \nu_{x,y}^t(\cone_r^+(x,A^-)\times\cone_r^+(y,B^-))
 \end{align*}
and hence
\begin{align*}
    \limsup \nu_{x,y}^t(\Acal\times \Bcal) &\leq \limsup \nu_{x,y}^t(\cone_r^-(x,A^+)\times\cone_r^-(y,B^+)) \\
    & \leq e^\eps \mu_x(A^+) \, \mu_y(B^+) \\
    & \leq e^\eps\mu_x(\Acal) \, \mu_y(\Bcal) + \eps'e^\eps(\Vert\mu_x\Vert + \Vert\mu_y\Vert),
\end{align*}
and
\begin{align*}
    \liminf \nu_{x,y}^t(\Acal\times \Bcal) &\geq \liminf \nu_{x,y}^t(\cone_r^+(x,A^-)\times\cone_r^+(y,B^-)) \\
    & \geq e^{-\eps} \mu_x(A^-) \, \mu_y(B^-) \\
    & \geq e^{-\eps} \mu_x(\Acal) \, \mu_y(\Bcal) - \eps'e^\eps(\Vert\mu_x\Vert + \Vert\mu_y\Vert).
\end{align*}
This ends the proof of the proposition, since $\eps'$ can be taken arbitrarily small.
\end{proof}

We now conclude the proof of Theorem \ref{thm:orbit_equidist}.

Let $\nu$ be an accumulation point of $(\nu_{x,y}^t)_{t\to\infty}$. Let $\varphi$ be a non-negative continuous function on $\overline{\Omega}^2$; it is enough to prove that for any $\eps>0$,
\[ e^{-\eps} \int \varphi \,d(\mu_x \otimes \mu_y)-\eps \leq \int \varphi \,d\nu \leq e^\eps \int \varphi \,d(\mu_x \otimes \mu_y)+ \eps .\]
As noted above, the projection of $\nu$ to the first and second coordinates are accumulation points of, respectively, $(\delta\Vert m_\Gamma\Vert e^{-\delta t}\alpha_{x,y}^t)_t$ and $(\delta\Vert m_\Gamma\Vert e^{-\delta t}\alpha_{y,x}^t)_t$. Thus, according to Lemma~\ref{lem:alpha<mu}, $\nu$ (as well as $\mu_x\otimes\mu_x$) gives full measure to the set of pairs of $\mathcal C^1$ and extremal points of $\overline{\Omega}$. Let $K\subset\del\Omega$ be a compact set of $\mathcal C^1$ and extremal points such that 
\begin{equation*}
 \int_{\overline{\Omega}^2\smallsetminus K^2}\varphi \,d\nu + \int_{\overline{\Omega}^2\smallsetminus K^2}\varphi \,d(\mu_x\otimes\mu_y)\leq \epsilon.
\end{equation*}
By Proposition~\ref{prop:3eme etape}, $K^2$ can be covered by a finite number of open sets $(U_i)_{1\leq i\leq n}$, where $n\in\nats_{>0}$, such that for any $1\leq i\leq n$, for any non-negative function $\psi$ supported on $U_i$,
\[ e^{-\eps} \int \psi \,d(\mu_x \otimes \mu_y) \leq \int \psi \,d\nu \leq e^\eps \int \psi \,d(\mu_x \otimes \mu_y) .\]
Set $U_0:=\overline{\Omega}^2\smallsetminus K^2$. Let $(\chi_i)_{0\leq i\leq n}$ be a partition of unity associated to $(U_i)_{0\leq i\leq n}$. Then
\begin{align*}
    \int \varphi \,d\nu & = \int\chi_0\varphi \,d\nu + \sum_{i=1}^n\int\chi_i\varphi \,d\nu \\
    & \leq \eps + e^\eps \sum_{i=1}^n\int\chi_i\varphi \,d(\mu_x\otimes\mu_y) \\
    & \leq e^\eps \int\varphi \,d(\mu_x\otimes\mu_y) +\eps,
\end{align*}
and similarly
\begin{equation*}
    \int \varphi \,d\nu\geq e^{-\eps} \int\varphi \,d(\mu_x\otimes\mu_y) -\eps.\qedhere
\end{equation*}
\end{proof}

\section{Equidistribution of primitive closed geodesics} \label{sec:pcgeod_equidist}

In this section, following \cite[Th.\,5.1.1]{Roblin}, we prove an equidistribution result for primitive closed geodesics, again with counting results for such geodesics as a consequence:
\begin{thm}%[cf. {\cite[Th. 5.1.1]{Roblin}}] 
\label{thm:pcgeod_equidist}
Let $\Omega\subset \RPn$ be a domain, and suppose $\Gamma \leq \Aut(\Omega)$ is a non-elementary rank-one discrete subgroup such that $S \Omega / \Gamma$ admits a finite Sullivan measure $m_\Gamma$ associated to a $\Gamma$-equivariant conformal density $\mu$ of dimension $\delta(\Gamma)$.
Then
\[ \delta \ell e^{-\delta \ell} \sum_{g \in \mathcal{G}^{\rkone}_\Gamma\!(\ell)} \dirac_g \xrightarrow[\ell \to +\infty]{}
\frac{m_\Gamma}{\|m_\Gamma\|} \]
weakly in $C_c(S\Omega/\Gamma)^*$.
\end{thm}
Here $\mathcal{G}^{\rkone}_\Gamma\!(\ell)$ denotes the set of primitive closed rank-one geodesics of length at most $\ell$ in $\Omega/\Gamma$, for any primitive closed rank-one geodesic $g \in \mathcal{G}^{\rkone}_\Gamma = \bigcup_{\ell>0} \mathcal{G}^{\rkone}_\Gamma\!(\ell)$ we write $\dirac_{g}$ to denote the normalized Lebesgue measure supported on $g$, and $C_c(S\Omega / \Gamma)^*$ denotes the weak* dual of the space of compactly-supported continuous functions on $S\Omega / \Gamma$.

If $\Omega$ is strictly convex with $C^1$ boundary, then all primitive closed geodesics are rank-one and we may omit the rank-one hypothesis in the previous paragraph.

As a corollary, we may already obtain the following counting result for primitive closed rank-one geodesics, by integrating against the measures on both sides a function $f\in C_c(S\Omega/\Gamma)$ which is equal to $1$ on the (compact) convex core:
\begin{cor} In the setting of Theorem~\ref{thm:pcgeod_equidist}, if $\Gamma$ acts convex cocompactly on $\Omega$, then

% \[ \#\mathcal{G}^{\rkone}_\Gamma \!(\ell) \sim \frac{e^{\delta\ell}}{\delta\ell} \quad\mbox{ as }\ell\to\infty .\]
% \[ \#\mathcal{G}^{\rkone}_\Gamma \!(\ell) \underset{\ell\to\infty}\sim \frac{e^{\delta\ell}}{\delta\ell}.\]
\[ \#\mathcal{G}^{\rkone}_\Gamma \!(\ell) \underset{\ell\to\infty}{\widesim} \frac{e^{\delta\ell}}{\delta\ell}.\]
% \begin{proof}
% Take $f\in C_c(S\Omega/\Gamma)$ which is equal to $1$ on the (compact) convex core. Integrating $f$ against the measure $\delta \ell e^{-\delta \ell} \sum_{g \in \mathcal{G}^{\rkone}_\Gamma\!(\ell)} \dirac_g$ gives $\delta \ell e^{-\delta \ell} \cdot \#\mathcal{G}^{\rkone}_\Gamma\!(\ell)$. From Theorem \ref{thm:pcgeod_equidist}, this integral converges to 1 as $ \ell \to \infty$. 
% \end{proof}
\end{cor}

As above, these results extend and are inspired by results first proven in the context of closed manifolds of constant negative curvature \cite{Margulis,Bowen} 
% Margulis has results for variable negative curvature; Bowen's work is in constant negative curvature
and subsequently extended to much more general settings; we refer the interested reader to the beginning of \cite[Chap.\,5]{Roblin} 
% \todo{PL:Roblin himself says that he does not give a lot of historical details, and refers to other people. Should we also refer to these people?}
for a more extended version of this history.

The proof of Theorem~\ref{thm:pcgeod_equidist} will take place in the universal cover. Thus we need to interpret in the universal cover the Lebesgue measure on rank-one periodic orbits of $SM := S\Omega / \Gamma$ (in the setting of Theorem~\ref{thm:pcgeod_equidist}). 

\begin{defn}\label{defn:strongly primitive}
Let $M=\Omega/\Gamma$ be a non-elementary rank-one convex projective orbifold. A rank-one element $\gamma\in\Gamma$ is said to be {\bf strongly primitive} if $\ell(\gamma)\leq \ell(\gamma')$ for any rank-one element $\gamma'\in\Gamma$ with the same axis as $\gamma$. 

Note strong primitivity is conjugacy-invariant.

Let $g$ be a closed rank-one $(g_\Gamma^t)$-orbit in $SM$. The {\bf (strongly primitive) conjugacy classes associated to $g$} are the conjugacy classes of (strongly primitive) elements $\gamma\in\Gamma$ such that $\gamma\tilde v=g^{\ell(\gamma)}\tilde v$ for any $\tilde v\in S\Omega$ in a lift in $S\Omega$ of $g$.
\end{defn}
\newcommand{\LebSegUniv}{\Leb_{[\tilde v\, g^\ell \tilde v]}}
\begin{notation}\label{nota:Leb}
Let $M=\Omega/\Gamma$ be a non-elementary rank-one convex projective orbifold. For any $\tilde v\in S\Omega$ and $\ell\geq 0$, we denote by $\LebSegUniv$ the push-forward by $t\mapsto g^t\tilde v$ of the Lebesgue measure (of mass $\ell$) on $[0,\ell]$.

For any rank-one $\gamma\in\Gamma$, we denote by $\Leb_\gamma$ the push-forward of the Lebesgue measure on $\real$ by $t\mapsto g^t\tilde w$ for any $\tilde w\in S\Omega$ tangent to the axis of $\gamma$.
\end{notation}

\begin{obs}\label{obs:vers l'infini et l'au-delà!}
Let $M=\Omega/\Gamma$ be a non-elementary rank-one convex projective orbifold. Consider a closed rank-one $(g_\Gamma^t)$-orbit $g=\{g^tv\}_t\subset SM$ with period $\ell$, where $v\in SM$. Let $\tilde v\in S\Omega$ be any lift of $v$, and $A$ the union of strongly primitive conjugacy classes associated to $g$. Then, using Notations~\ref{nota:Leb}, the measure $\ell\cdot\dirac_g$ is the quotient of $\sum_{\gamma\in\Gamma} \gamma_* \LebSegUniv$ by $\Gamma$ (in the same sense than $m_\Gamma$ is the quotient of $m$ for any Sullivan measures $m$ and $m_\Gamma$), and
$$\sum_{\gamma\in \Gamma}\gamma_* \LebSegUniv = \sum_{\gamma\in A}\Leb_\gamma.$$
\end{obs}
\begin{proof}
The measure $\ell\cdot\dirac_g$ is the push-forward by $\pi_\Gamma: S\Omega\rightarrow SM$ of $\LebSegUniv$, so by definition of a quotient measure, $\ell\cdot\dirac_g$ is the quotient of $\sum_{\gamma\in\Gamma} \gamma_* \LebSegUniv$ by $\Gamma$.

Fix $\gamma_0\in\Gamma$ such that $g^{\ell} \tilde v=\gamma_0\tilde v$. Let $H:= \Stab_\Gamma\{g^t\tilde v\}_t$, let $\mathcal{R}\subset\Gamma$ be a set of representatives of $\Gamma / H$ containing the identity, and $B=\gamma_0 \Stab_\Gamma(\tilde{v})\subset H$. Then $\Gamma=\mathcal{R}\cdot H=\mathcal{R}\cdot \langle \gamma_0 \rangle \cdot \Stab_H(\tilde{v})$.
\begin{align*}
 \sum_{\gamma\in\Gamma}\gamma_*\LebSegUniv & = \!\!\!\!\! \sum_{\substack{n \in \ints \\ (r,h)\in \mathcal{R}\times \Stab_H(\tilde{v})}} \!\!\!\!\! r_*\gamma^n_*h_*\LebSegUniv %\\ & 
 = \!\!\!\!\!\sum_{\substack{n \in \ints \\ (r,h)\in \mathcal{R}\times \Stab_H(\tilde{v})}} \!\!\!\!\! r_*\gamma^n_*\LebSegUniv \\
 & = \sum_{(r,h)\in \mathcal{R}\times \Stab_H(\tilde{v})}r_*\Leb_\gamma %\\ & 
 = \sum_{(r,h)\in \mathcal{R}\times \Stab_H(\tilde{v})}r_*\Leb_{\gamma h} \\
 & = \sum_{(r,g)\in \mathcal{R}\times B}\Leb_{rgr^{-1}}.
\end{align*}
Finally observe that the map $\mathcal{R}\times B\rightarrow A$ that maps $(r,g)$ to $rgr^{-1}$ is a bijection, proving our claim.
\end{proof}

\begin{proof}[Proof of Theorem \ref{thm:pcgeod_equidist}]
Let $\Gamma^{\prkone}\subset\Gamma$ be the set of strongly primitive rank-one elements of $\Gamma$ (see Definition~\ref{defn:strongly primitive}).
Using Notations~\ref{nota:Leb}, set
$$\Epsy^L := \delta L e^{-\delta L} \sum_{\substack{\gamma\in\Gamma^{\prkone} \\ \ell(\gamma) \leq L}} \frac{1}{\ell(\gamma)} \Leb_\gamma.$$
As noted in Observation~\ref{obs:vers l'infini et l'au-delà!}, the quotient of $\Epsy^L$ under the action of $\Gamma$ is precisely 
$$\delta L e^{-\delta L} \sum_{g \in \mathcal{G}^{\rkone}_\Gamma\!(L)} \dirac_g.$$ 
Since $m_\Gamma$ is the quotient of $m$, we need to prove that $\Epsy^L \to \frac{m}{\|m_\Gamma\|}$ weakly in $C_c(S\Omega)^*$ when $L\to+\infty$.

% Abusing notation slightly, we write $\mu$\todo{PL: we could call it $m_{\real}$? as for quotient by the action of $\real$, \ie the action of the geodesic flow} to also denote the measure on $\del^2\Omega$ given by
% \[ d\mu(\xi,\eta) = e^{ \delta \cdot (\beta_\xi(x,u) + \beta_\eta(x,u))} \,d\mu_x(\xi) \,d\mu_x(\eta) = e^{2\delta\cdot \grop{x}\xi\eta} d\mu_x(\xi) d\mu_x(\eta) \]
% which is independent of the choice of $x \in \Omega$ and $u \in (\xi\eta) \subset \Omega$. We recall again that by definition, $m = \mu \otimes ds$ on $S\Omega_{\sse} = \del^2_{\sse}\Omega \times \real$.

We will first use Theorem \ref{thm:orbit_equidist} to obtain a measure $\nu_{x,1}^L$ converging weakly to $\mu$ when $L\to+\infty$, then successively modify $\nu_{x,1}^L$ to form $\nu_{x,2}^L$ and $\nu_{x,3}^L$, so that $\nu_{x,3}^L$ will be supported on pairs of fixed points of rank-one elements, and $\nu_{x,3}^L$ locally approaches $\mu$. By taking the product of $\|m_\Gamma\|^{-1} \nu_{x,3}^L$ with the Lebesgue measure on $\real$, we obtain a measure $\Em_{x,3}^L$ approaching $\|m_\Gamma\|^{-1} m$ locally (i.e. near the fibre over $x\in\Omega$ in $S\Omega$.) To finish, we relate $\Em_{x,3}^L$ 
to the measure of equidistribution $\Epsy^L$. 

To relate our various modified measures we will use the following lemma. Observe that the lemma will also imply that the theorem remains true if we replace $\mathcal{G}^{\rkone}_{\Gamma}$ by any bigger set of primitive closed \emph{straight} geodesics that are in different free homotopy classes.
\begin{lem}\label{lem:neglects non-rk1}
In the setting of Theorem~\ref{thm:pcgeod_equidist}, given $x\in\Omega$, we have
\[ e^{-\delta t}\cdot \#\{\gamma\in\Gamma \text{ not rank-one} \:|\: d_\Omega(x,\gamma x)\leq t\} \xrightarrow[t\to\infty]{} 0. \]
Moreover, for any $\epsilon>0$, we have,
\[ e^{-\delta t}\cdot \#\{\gamma\in\Gamma \text{ rank-one} \:|\: d_\Omega(x,\gamma x)\leq t \text{ and } d_{\proj(V)}(\gamma x,\gamma^+)\geq \eps\} \xrightarrow[t\to\infty]{} 0. \]
% where $\gamma^+$ denote the the attracting fixed point of any proximal element $\gamma$.
\end{lem}

%The following corollary implies, roughly speaking, that the equidistribution result in Theorem \ref{thm:orbit_equidist} can be taken to extend to all of $[\Gamma]$, once we pick a suitable analogue of $\dirac_g$ for non-rank-one elements $g$, which do not in general preserve a (straight Hilbert) geodesic in $\Omega$. We denote by $[\Gamma]^{\mathrm{sing}}$ the set of conjugacy classes of non-rank-one elements of $\Gamma$.
%\begin{cor}\label{cor:nonr1equidist}
%In the setting of Theorem~\ref{thm:pcgeod_equidist}, for each non-rank-one conjugacy class $[\gamma] \in [\Gamma]^{\mathrm{sing}}$, choose $\gamma \in [\gamma]$ and a (non-unit speed) geodesic path $c_\gamma:[0,1]\rightarrow \Omega$ such that $c_\gamma(1)=\gamma c_\gamma(0)$, and $d_\Omega(c_\gamma(0),c_\gamma(1))\leq \ell([\gamma])+1$, and let $\dirac_{[\gamma]}$ be the push-forward by $\pi_{\Gamma}\circ c_\gamma$ of the Lebesgue measure on $[0,1]$. Then
%\[\delta\ell e^{-\delta \ell} \sum_{[\gamma]\in[\Gamma]^{\mathrm{sing}}}\dirac_{[\gamma]} \xrightarrow[\ell\to\infty]{} 0\]
%weakly in $C_c(\Omega/\Gamma)^*$.\todo{proof to be written}
%\end{cor} 

To prove Lemma~\ref{lem:neglects non-rk1} we need the following fact, which can be seen as a kind of closing lemma.

\begin{lem}[{\cite[Cor.\,7.13]{Blayac_PSdensities}}]\label{lem:511NSC}
Let $\Omega\subset \RPn$ be a domain. Fix $x\in\Omega$, $(\xi_-,\xi_+)\in\del^2\Omega$ two distinct strongly extremal points and $W$ a neighbourhood of $(\xi_-,\xi_+)$ in $\RPn^2$.

Then there exists a neighbourhood $U$ of $(\xi_-,\xi_+)$ in $\overline \Omega^2$ such that any $\gamma\in\Aut(\Omega)$ with $(\gamma^{-1}x,\gamma x)\in U$ is rank-one with $(\gamma^-,\gamma^+)\in W$.
\end{lem}

\begin{proof}[Proof of Lemma~\ref{lem:neglects non-rk1}]
Fix $0<\epsilon'<\eps$. Since $\mu_x\otimes \mu_x$ gives full measure to the set $A$ of distinct pairs of strongly extremal points of the boundary $\del\Omega$, we can find a compact subset $K\subset A$ such that $(\mu_x\otimes\mu_x)(\overline\Omega^2\smallsetminus K)\leq \delta\|m_\Gamma\|\epsilon'$.

By Lemma~\ref{lem:511NSC} applied with $W$ an $\eps$-neighbourhood of $K$, we can find a neighbourhood $U$ of $K$ such that for any $\gamma\in\Gamma$, if $(\gamma^{-1}x,\gamma x)\in U$, then $\gamma$ is rank-one. If furthermore $(\gamma^{-1}x,\gamma x)\in U \cap W$, we also have $d_{\RPn}(\gamma x,\gamma^+)<\eps$. 

For $t>0$ we denote by $B_t$ the set of elements $\gamma$ such that $d_\Omega(x,\gamma x)\leq t$ and such that either $\gamma$ is not rank-one or $d_{\RPn}(\gamma x,\gamma^+)\geq \eps$; then
\begin{align*}
    \limsup_{t\to\infty}e^{-\delta t}\cdot \#B_t
    & \leq \limsup_{t\to\infty}e^{-\delta t} \left(\sum_{\substack{\gamma\in\Gamma \\ d_\Omega(x,\gamma x)\leq t}}\dirac_{\gamma^{-1}x}\otimes\dirac_{\gamma x}\right)\left(\overline \Omega^2\smallsetminus (U \cap W) \right) \\
    & \leq \delta^{-1}\|m_\Gamma\|^{-1} \cdot (\mu_x\otimes\mu_x) \left(\overline\Omega^2\smallsetminus (U \cap W) \right) \\
    & \leq \epsilon'.
\end{align*}
Since this holds for arbitrarily small $\eps'$, we obtain $e^{-\delta t} \cdot \#B_t \xrightarrow[t\to\infty]{} 0$.
\end{proof}

Fix for now $x \in \Omega$. By Theorem \ref{thm:orbit_equidist}, the measure
\[ \nu_{x,1}^L := \delta \|m_\Gamma\| e^{-\delta L} \sum_{d_\Omega(x,\gamma x) \leq L} \dirac_{\gamma^{-1}x} \otimes \dirac_{\gamma x} \]
converges weakly in $C(\bar\Omega \times \bar\Omega)^*$ to $\mu_x \otimes \mu_x$ as $L\to+\infty$. 
% Since $0 \leq \grop{x}\xi\eta \leq r$ for $(\xi,\eta) \in  \pi_{\hor}^{-1} \left( V(x,r) \right)$, we have, for all $\psi \in C_c^+(V(x,r))$,
% {\color{P.-L.color}
% \[ e^{-2\delta r} \int \psi\,dm_\real
% \leq \liminf \int \psi\,d\nu_{x,1}^L = \limsup \int \psi\,d\nu_{x,1}^L 
% \leq \int \psi\,dm_\real \]
% \[ e^{-2\delta r} \int \psi\,dm_\real
% = \lim \int \psi\,d\nu_{x,1}^L = \int \psi\,d(\mu_x \otimes \mu_x)
% \leq \int \psi\,dm_\real \]
% (or omit this equation altogether, and move the surrounding bits of the sentence to around the corresponding equation for $\mu_{x,3}^L$)}
% as $L\to+\infty$, where $dm_\real(\xi,\eta) := e^{2\delta  \grop{x}\xi\eta} \,d\mu_x(\xi) \,d\mu_x(\eta)$ (so that $m = m_\real \otimes \Leb$); recall that $\mu_x$ has no atoms and gives full measure to $\del_{\sse}\Omega$, so the Gromov product $\grop x\xi\eta$ is well defined for $(\mu_x\otimes\mu_x)$-almost any pair $(\xi,\eta)$.

% Now as we start to modify our measure. 
Write $\Gamma^{\rkone}$ to denote the set of rank-one elements of $\Gamma$, and define the modified measure
\[ \nu_{x,2}^L := \delta \|m_\Gamma\| e^{-\delta L} \sum_{\substack{\gamma\in\Gamma^{\rkone} \\ d_\Omega(x,\gamma x) \leq L}} \dirac_{\gamma^{-1}x} \otimes \dirac_{\gamma x}. \]
According to Lemma \ref{lem:neglects non-rk1}, we have $\nu_{x,1}^L - \nu_{x,2}^L \to 0$ weakly as $L\to+\infty$.

For $\gamma \in \Gamma^{\rkone}$, write $\gamma^\pm$ for its attracting and repelling fixed points, and define
\[ \nu_{x,3}^L := \delta \|m_\Gamma\| e^{-\delta L} \sum_{\substack{\gamma\in\Gamma^{\rkone} \\ d_\Omega(x,\gamma x) \leq L}} \dirac_{\gamma^-} \otimes \dirac_{\gamma^+}. \]
By Lemma \ref{lem:neglects non-rk1}, we have $\nu_{x,3}^L - \nu_{x,2}^L \to 0$ weakly as $L\to+\infty$.

Fix $r > 0$, and let $V(x,r)$ denote the open set of pairs $(a,b) \in \Geod\Omega$ such that the projective segment $[a\,b]$ intersects $B(x,r)$.
%We restrict henceforth these measures to the open set $V(x,r)$, that is to say that we will consider them as elements of $C_c(V(x,r))^*$. 
Since $0 \leq \grop{x}\xi\eta \leq r$ for $(\xi,\eta) \in  \pi_{\hor}^{-1} \!\left( V(x,r) \right)$, the preceding series of convergences 
gives us that, for all $\psi \in C_c^+(V(x,r))$ (and hence for $\psi \in C_c^+\left( (\del\Omega)^2 \cap V(x,r) \right)$, since the measures we are talking about are supported on $(\del\Omega)^2$),
% \[ e^{-2\delta r} \int \psi\,dm_\real
% \leq \liminf \int \psi\,d\nu_{x,3}^L = \limsup \int \psi\,d\nu_{x,3}^L 
% \leq \int \psi\,dm_\real \]
\[ e^{-2\delta r} \int \psi\,dm_\real
\leq \lim \int \psi\,d\nu_{x,3}^L =  \int \psi\,d(\mu_x \otimes \mu_x) 
\leq \int \psi\,dm_\real \]
as $L\to+\infty$, where $dm_\real(\xi,\eta) := e^{2\delta  \grop{x}\xi\eta} \,d\mu_x(\xi) \,d\mu_x(\eta)$ (so that $m = m_\real \otimes \Leb$); recall that $\mu_x$ has no atoms and gives full measure to $\del_{\sse}\Omega$, so the Gromov product $\grop x\xi\eta$ is well defined for $(\mu_x\otimes\mu_x)$-almost any pair $(\xi,\eta)$.

Finally, set
%For $\gamma \in \Gamma^{\rkone}$, we let $g_\gamma \subset S\Omega$ denote the oriented axis of $\gamma$, and $\Leb_\gamma$ denote the Lebesgue measure supported on $g_\gamma$, and finally\todo{Part of this discussion is now reproduced in Remark \ref{vers l'infini et l'au-delà!}}
\[ \Em_{x,3}^L = \delta e^{-\delta L} \!\!\!\!\! \sum_{\substack{\gamma\in\Gamma^{\rkone} \\ d_\Omega(x,\gamma x) \leq L}} \!\!\!\!\! \Leb_\gamma .\]
In other words, $\Em_{x,3}^L$ is the push-forward by the Hopf parametrization of the measure $\|m_\Gamma\|^{-1} \nu_{x,3}^L \otimes \Leb$. Note that for any $L$, the measure $\nu_{x,3}^L$ is concentrated on $\del^2_{\sse}\Omega$, which is identified with its preimage in $\pi_{\hor}^{-1}(\Geod^\infty\Omega)$, thus the Hopf parametrization is well defined $(\nu^L_{x,3}\otimes \Leb)$-almost surely. 

Let $\widehat{V}(x,r)\subset S\Omega$ be the set of vectors $v$ with $(v^-,v^+)\in V(x,r)$. From the preceding, we obtain, for $\psi \in C_c^+(\widehat{V}(x,r))$, 
\begin{align} 
e^{-2\delta r} \|m_\Gamma\|^{-1} \int \psi\,dm & \leq \lim \int \psi\,d\Em_{x,3}^L \notag \\ &
= \int \psi\,d(\mu_x \otimes \mu_x) \,ds \leq \|m_\Gamma\|^{-1} \int \psi\,dm \label{eqn:511_Emx3tNSC}.
\end{align}

{\bf We now relate $\Em_{x,3}^L$ to $\Epsy^L$}, via a slight modification $\Em^L$ (which is in fact independent of $x$). With $\ell(\gamma)$ the translation length of $\gamma \in \Gamma^{\rkone}$, as defined at the end of \S\ref{sub:hilbgeom}, define
\[ \Em^L = \delta e^{-\delta L} \sum_{\substack{\gamma\in\Gamma^{\rkone} \\ \ell(\gamma) \leq L}} \Leb_\gamma .\]
Making the elementary observation that
$\ell(\gamma) \leq d_\Omega(x,\gamma x) \leq \ell(\gamma) + 2 d_\Omega(x, g_\gamma)$,
we deduce that
\begin{equation} \Em_{x,3}^L \leq \Em^L \leq e^{2\delta r} \Em_{x,3}^{L+2r} \label{eqn:511_EmxtNSC} \end{equation}
when restricted to $\widehat{V}(x,r)$.

Now
%, with $\Gamma^{\prkone}$ the set of \emph{primitive} rank-one elements of $\Gamma$, 
let us check that
\begin{equation} \Em^L = \delta e^{-\delta L} \sum_{\substack{\gamma\in\Gamma^{\prkone} \\ \ell(\gamma) \leq L}} \left\lfloor \frac{L}{\ell(\gamma)} \right\rfloor \Leb_\gamma . \label{eqn:511_Emxt_ghpNSC} \end{equation}
Indeed, let $\gamma\in\Gamma$ be rank-one with $\ell(\gamma)\leq L$. Let $A$ be the set of rank-one elements $\gamma'$ with $(\gamma')^\pm = \gamma^\pm$ and with $\ell(\gamma')\leq L$ (these incidentally satisfy $\Leb_{\gamma'}=\Leb_\gamma$), let $\gamma_0\in A$ be strongly primitive, and let $H\leq\Gamma$ be the group of elements that fix every point of the axis of $\gamma$. Then $A = \{\gamma_0^kh \:|\: 1\leq k\leq \frac{L}{\ell(\gamma_0)}, \ h\in H\}$ has cardinality $\left\lfloor \frac{L}{\ell(\gamma_0)} \right\rfloor\cdot \# H$, while $A\cap\Gamma^{\prkone}=\{\gamma_0 h \:|\: h\in H\}$ has cardinality $\#H$. Thus
$\sum_{\gamma'\in A}\Leb_{\gamma'}=\sum_{\gamma'\in A\cap\Gamma^{\prkone}}\left\lfloor \frac{L}{\ell(\gamma')} \right\rfloor\Leb_{\gamma'}$
%since for any $\gamma\in\Gamma^{\rkone}$,
%\[ \left\lfloor \frac{L}{\ell(\gamma)} \right\rfloor = \#\{ n \in \ints_{>0} \st \ell(\gamma^n) = n\ell(\gamma) \leq L \} .\]
and we add these up over rank-one elements to obtain \eqref{eqn:511_Emxt_ghpNSC}.

It is then clear that
\begin{equation} 
\Em^L \leq \Epsy^L := \delta  e^{-\delta L} \sum_{\substack{\gamma\in\Gamma^{\prkone} \\ \ell(\gamma) \leq L}} \frac{L}{\ell(\gamma)} \Leb_\gamma .
\label{eqn:511_MtEtNSC} \end{equation}

{\bf To obtain a complementary inequality}, consider $\varphi \in C_c^+(\widehat{V}(x,r))$. Fix $L\geq 2e^r$. Observe that $\left\lfloor \frac{L}{\ell} \right\rfloor \geq 1 \geq \frac{e^{-r}L}{\ell}$ for any $e^{-r}L < \ell\leq L$, while $\frac{1}{\ell}\leq \lfloor\frac{e^{-r}L}{\ell}\rfloor$ for any $\ell\leq e^{-r}L$ (because $e^{-r}L\geq 2$). Therefore
\begin{align*}
 \int \varphi\,d\Em^L & \geq \delta e^{-\delta L} \!\!\!\sum_{\substack{\gamma\in\Gamma^{\prkone} \\ e^{-r}L < \ell(\gamma) \leq L}} \!\! \frac{e^{-r}L}{\ell(\gamma)} \int \varphi\;d\Leb_\gamma \notag \\
 & = e^{-r} \int \varphi\,d\Epsy^L - e^{-r} L \delta e^{-\delta L} \!\!\!\sum_{\substack{\gamma\in\Gamma^{\prkone} \\ \ell(\gamma) \leq e^{-r}L}} \!\! \frac{1}{\ell(\gamma)} \int \varphi\;d\Leb_\gamma \\
 & \geq e^{-r} \int \varphi\,d\Epsy^L - e^{-r}L e^{-\delta(1-e^{-r})L} \int \varphi\,d\mathcal{M}^{e^{-r}L}.
\end{align*}
By (\ref{eqn:511_Emx3tNSC}) and (\ref{eqn:511_EmxtNSC}), $\int \varphi\,d\Em^{e^{-r}L}$ is bounded as $L\to\infty$. Hence we have 
\[ \limsup \int \varphi\,d\Em^L \geq e^{-r} \limsup \int \varphi\,d\Epsy^L .\]
Combining the last inequality with (\ref{eqn:511_Emx3tNSC}), (\ref{eqn:511_EmxtNSC}) and (\ref{eqn:511_MtEtNSC}) above, we establish that for all $\varphi \in C_c^+(\widehat{V}(x,r))$, as $L\to+\infty$,
\begin{align*} e^{-2\delta r} \|m_\Gamma\|^{-1} \int \varphi\,dm & \leq \liminf \int \varphi\,d\Epsy^L \\
& \leq \limsup \int \varphi\,d\Epsy^L \leq e^{(2\delta+1)r} \|m_\Gamma\|^{-1} \int \varphi\,dm .
\end{align*}

We now let $x\in\Omega$ vary (but keep $r>0$ fixed). Appealing to a locally-finite partition of unity subordinate to a covering of $S\Omega$ by open sets of the form $\widehat{V}(x,r)$ with $x \in \Omega$, we extend the validity of the preceding inequalities to all functions $\varphi \in C_c^+(S\Omega)$. It remains only to take $r \to 0$ to conclude the proof. 
\end{proof}

\section{Periodic geodesics and conjugacy classes} \label{sec:perconj}

%Let $M=\Omega/\Gamma$ be a convex projective orbifold.  In Definition~\ref{biprox_orbit}, we have associated to any biproximal element $\gamma\in\Gamma$ whose axis meets $\Omega$ a periodic geodesic in $M$ (and $SM$). To any non-trivial power or conjugate of $\gamma$ is associated the same periodic geodesic. 
In this section, given a rank-one convex projective manifold $M=\Omega/\Gamma$, we use Theorem~\ref{thm:pcgeod_equidist} on equidistribution of closed rank-one geodesics o $SM$ to prove counting results for rank-one conjugacy classes in $\Gamma$. This is closely related to the discussion in \cite[\S9.5]{Blayac_PSdensities}.

If $M$ is a compact hyperbolic manifold, then there is a correspondence between the periodic $(g^t_\Gamma)$-orbits in $SM$ (as subsets) and the conjugacy classes of primitive elements of $\pi_1(M)$. (Recall that an element $\gamma\in\pi_1(M)$ is primitive if it does not belong to $\{h^k: h\in\pi_1(M), \ k\geq 2\}$.)

In this section we prove that, under an irreducibility assumption, ``most rank-one periodic geodesics'' are associated to exactly one conjugacy class, and ``most rank-one conjugacy classes'' are strongly primitive. This has the following consequence.

\begin{prop}\label{prop:countconjclassirred}
 Let $\Omega\subset\proj(V)$ be a domain, and $\Gamma < \Aut(\Omega)$ a strongly irreducible discrete subgroup with $M=\Omega/\Gamma$ rank-one. Consider a Sullivan measure $m_\Gamma$ of dimension $\delta_\Gamma$, and suppose it is finite.
 Then
 \[{\delta_\Gamma T}e^{-\delta_\Gamma T}\sum_{c\in [\Gamma]^{\rkone}_T}\dirac_c \xrightarrow[T\to\infty]{} \frac{m_\Gamma}{\Vert m_\Gamma\Vert}\]
 in $C_c^*(T^1M)$, where $[\Gamma]^{\rkone}_T$ denotes the set of conjugacy classes of rank-one elements of $\Gamma$ with translation length less than $T$. 
\end{prop}

Integrating a function which is constant on the compact core against both sides of the last statement, we obtain the following
\begin{cor}
Let $\Omega\subset\proj(V)$ be a domain, and $\Gamma\leq\Aut(\Omega)$ a strongly irreducible discrete subgroup which acts convex cocompactly on $\Omega$ with rank-one quotient. Then 
\[ \# [\Gamma]^{\rkone}_T \underset{T\to\infty}{\widesim} \frac{e^{-\delta_\Gamma T}}{\delta_\Gamma T} .\] 
\end{cor}

Such a correspondence fails for general convex projective orbifolds, for several reasons (which can combine), one of which is torsion, as we can see in the following 
\begin{eg}\label{eg:timesZ/2Z}
Let $\Gamma'$ be a cocompact torsion-free discrete subgroup of $\PSO(2,1)$, which naturally embeds in $\PSO(3,1)$, and $M'=\HH^3/\Gamma'$; let $r\in\PO(3,1)$ be the orthogonal reflection of $\HH^3$ that preserves the natural embedding of $\HH^2$ in $\HH^3$, so that $r$ commutes with $\Gamma'$; let $\Gamma$ be the group generated by $\Gamma'$ and $r$ (\ie $\Gamma\simeq \Gamma'\times \ints/2\ints$), and $M=\HH^3/\Gamma$. Note that $\Gamma$ is not irreducible. One can check that for any $\ell>0$, there are exactly as many periodic $(g^t_\Gamma)$-orbits of period less than $\ell$ in $SM$ and in $SM'$, whereas there are at least
% (actually there are even more) 
twice as many primitive conjugacy classes of translation length less than $\ell$ in $\Gamma'$ than in $\Gamma$.
\end{eg}

When the fundamental group has torsion, a possible way to obtain a well-defined correspondence is, as we did in Definition~\ref{defn:strongly primitive}, to associate to each closed rank-one orbit of the unit tangent bundle of a rank-one convex projective orbifold $M=\Omega/\Gamma$ a collection of strongly primitive conjugacy classes of the fundamental group $\Gamma$. (Note that each strongly primitive element belongs to such a collection, and two such collections are disjoint or equal.)

\begin{notation}\label{nota:nbconj}
Let $M=\Omega/\Gamma$ be a rank-one convex projective orbifold. For each closed rank-one geodesic $g\subset SM$, denote by $\nbconj_g$ the number of associated conjugacy classes of strongly primitive rank-one elements. We also write $\nbconj_c:=\nbconj_g=:\nbconj_\gamma$ if $c$ is a conjugacy class associated to $g$ and $\gamma\in c$.

More generally, given $\gamma \in \Gamma$ a (not necessarily strongly primitive) rank-one element, fix $\tilde{v} \in S\Omega$ such that $\gamma \tilde{v} = g^{\ell(\gamma)}\tilde{v}$, and denote by $\nbconj_\gamma$ the number of conjugacy classes in $\Gamma$ of elements of $\gamma\Stab_\Gamma(\tilde v)$, \ie elements with same attracting and repelling fixed points, and translation length as $\gamma$.
\end{notation}

In the setting of Notation~\ref{nota:nbconj}, we have of course 
\begin{equation}\label{eq:N<Stab}
\nbconj_\gamma \leq \#\Stab_\Gamma(\tilde v),
\end{equation}
and we may have a strict inequality. Note also that $N_\gamma$ and $N_{\gamma^2}$ may be different.

Theorem~\ref{thm:pcgeod_equidist} may then be reformulated as
\[\delta_\Gamma Te^{-\delta_\Gamma T}\sum_{c\in [\Gamma]^{\prkone}_T}\frac{\dirac_c}{\nbconj_c} \xrightarrow[T\to\infty]{} \frac{m_\Gamma}{\Vert m_\Gamma\Vert}\]
where $\mathcal D_c$ denotes the flow-invariant probability measure on the closed orbit associated to $c$, and $[\Gamma]^{\prkone}_T$ denotes the set of conjugacy classes of strongly primitive rank-one elements of $\Gamma$ with translation length less than $T$.  

In Example~\ref{eg:timesZ/2Z}, we have $N_g=2$ for every closed rank-one geodesic $g\subset SM$.
In general two different rank-one periodic geodesics may be associated to a different number of conjugacy classes.

We note that Proposition~\ref{prop:countconjclassirred} does not apply directly to Example~\ref{eg:timesZ/2Z}.
We will prove a slightly stronger result (Proposition \ref{prop:countingconj}), which implies Proposition~\ref{prop:countconjclassirred} and encompasses Example~\ref{eg:timesZ/2Z}, and partially answers the following more general

\begin{question}
 Let $M=\Omega/\Gamma$ be a non-elementary rank-one convex projective orbifold with finite Sullivan measure $m_\Gamma$ of dimension $\delta_\Gamma$, and $$\nbconj=\min\{\nbconj_g:g\subset SM \ \text{closed rank-one geodesic}\}.$$
 Does the following hold?
 \[\frac{\delta_\Gamma T}{N}e^{-\delta_\Gamma T}\sum_{c\in [\Gamma]^{\rkone}_T}\dirac_c \xrightarrow[T\to\infty]{} \frac{m_\Gamma}{\Vert m_\Gamma\Vert}.\]
\end{question}

\subsection{The core-fixing subgroup}\label{sec:r1per pr1conj}

We now define a special subgroup of $\Gamma$, which happens to be the ``smallest stabilizer,  among the stabilizers $\Stab_\Gamma(\tilde v)$ for $\tilde v\in S\Omega_{\bip}=\pi_\Gamma^{-1}SM_{\bip}$''. 

\begin{defn}\label{defn:corfix}
 Let $M=\Omega/\Gamma$ be a non-elementary rank-one convex projective orbifold. The {\bf core-fixing subgroup} of $\Gamma$ is the kernel of the restriction of $\Gamma$ to the span of $\Lambda_\Gamma$. 
\end{defn}

By the non-elementary rank-one assumption and Proposition~\ref{prop:density of att-rep pairs}, the core-fixing subgroup always contains the center of $\Gamma$. Note that when $\Gamma$ is strongly irreducible, the core-fixing subgroup is trivial.

Below, we write $S\Omega_{\bip}$ to denote the preimage of $SM_{\bip}$ under $\pi_\Gamma:S\Omega\rightarrow SM$.
\begin{lem}\label{lem:corfix}
 Let $M=\Omega/\Gamma$ be a rank-one non-elementary convex projective orbifold, and let $\corfix < \Gamma$ be the core-fixing subgroup. Then $\{ x\in\Lambda_\Gamma \:|\: \Stab_\Gamma(x)=F\}$ is dense in $\Lambda_\Gamma$, and $\{v\in S\Omega_{\bip} \:|\: \Stab_\Gamma(v)=F\}$ is open and dense in $S\Omega_{\bip}$.
\end{lem}
\begin{proof}
 Consider $\gamma\in\Gamma\smallsetminus \corfix$. The set $\{x\in\Lambda_\Gamma \:|\: \gamma x\neq x\}$ is open. Let us show that it is dense, which will imply that the set of points $$\{x\in\Lambda_\Gamma \:|\: \Stab_\Gamma(x)=\corfix\} = \bigcap_{\gamma \in \Gamma \smallsetminus F} \{x \in \Lambda_\Gamma \:|\: \gamma x \neq x\}$$ is a dense \Gdelta\ set in $\Lambda_\Gamma$. Let $\varnothing \neq U\subset \Lambda_\Gamma$ be open. Let $\Gamma_0\subset\Gamma$ be the identity component for the Zariski topology. By \cite[Prop.\,3.2.2]{blayac_these}, 
%  a strengthening of the minimality of the action of $\Gamma_0$ on $\Lambda^{\prox}$, 
 $\{g\in\Gamma_0 \:|\: gx\in U\}$ is Zariski-dense in $\Gamma_0$ for any $x\in\Lambda_\Gamma$. This implies that there exists a minimal family of smooth and strongly extremal points $x_1,\dots,x_n\in U$ that spans the same space as the whole proximal limit set. Moreover, this also implies that there is a point $x\in U$ outside of the span of $x_1,\dots,x_{i-1},x_{i+1},\dots,x_n$ for any $1\leq i\leq n$. If $\gamma$ were to fix all points $x_1,\dots,x_n,x$, then its restriction to the span of $\Lambda_\Gamma$ would be diagonal in any basis associated to $x_1,\dots,x_n$, and $\gamma x=x$ would imply that all diagonal entries are equal, and hence that $\gamma\in\corfix$. Thus there exists $y\in U$ such that $\gamma y\neq y$, \ie our set $\{x\in\Lambda_\Gamma \:|\: \gamma x\neq x\}$ meets $U$, as desired.

 The density of points $x\in \Lambda_\Gamma$ with stabilizer $F$ implies the density of vectors $v\in S\Omega_{\bip}$ such that the endpoints $v^+$ has stabilizer $F$; the stabilizer of such vectors must be contained in $F$, hence equal to $F$, since $\Stab_\Gamma(w) \supset F$ for all $w \in S\Omega_{\bip}$. Thus $\{v\in S\Omega_{\bip} \:|\: \Stab_\Gamma(v)=\corfix\}$ is dense in $S\Omega_{\bip}$. 
 Let us show that it is also open.
 This is an immediate consequence of the fact that the map $v\in S\Omega\mapsto \Stab_\Gamma(v)$ is upper semi-continuous (with global minimum $F$), in the sense that $\Stab_\Gamma(w)\subset\Stab_\Gamma(v)$ for $w$ close enough to $v$.
\end{proof}

\subsection{Counting conjugacy classes}\label{sec:nb per conj}

We now relate the number of strongly primitive rank-one conjugacy classes, the number of rank-one conjugacy classes, and the number of closed rank-one $(g^t_\Gamma)$-orbits.

\begin{rmk}\label{rmk:nbconj}
 Let $M=\Omega/\Gamma$ be a non-elementary rank-one convex projective orbifold, $\gamma\in\Gamma$ be a (not necessarily strongly primitive) rank-one element, and $\tilde v\in S\Omega$ be such that $\gamma\tilde v=g^{\ell(\gamma)}\tilde v$. 
 
 If two elements $\gamma h,\gamma h'\in \gamma\Stab_\Gamma(\tilde v)$ are in the same conjugacy class, \ie we have $g\in \Gamma$ such that $\gamma h'=g\gamma hg^{-1}$, then the conjugating element $g$ is in the stabilizer of the axis of $\gamma$, namely $\gamma^{\ints}\Stab_\Gamma(\tilde v)$.
 Thus, $\nbconj_\gamma$ is the number of conjugacy classes in $\Stab_\Gamma(\tilde v)$ if this subgroup is centralized by $\gamma$, and it is exactly $\#\Stab_\Gamma(\tilde v)$ if moreover $\Stab_\Gamma(\tilde v)$ is abelian.
\end{rmk}

\begin{lem}\label{lem:nb per conjrob}
 Let $\Omega\subset\proj(V)$ be a domain, and $\Gamma<\Aut(\Omega)$ a discrete subgroup with $M=\Omega/\Gamma$  non-elementary rank-one. Let $m_\Gamma$ be a Sullivan measure of dimension $\delta_\Gamma$ and suppose it is finite. 
 Then for any non-negative function $f\in C_c(SM)$, there exists $C>0$ such that for any $T>0$,
 \[\sum_{c\in\mathcal{G}_T^{\rkone}}\int f\diff\dirac_c \leq \sum_{c\in [\Gamma]_T^{\prkone}}\int f\diff\dirac_c\leq C\sum_{c\in\mathcal{G}_T^{\rkone}}\int f\diff\dirac_c,\]
 and if $\Gamma$ contains a torsion-free finite-index subgroup $\Gamma'$, then we can take $C=[\Gamma:\Gamma']$ (which does not depend on $f$). Moreover,
 $${Te^{-\delta_\Gamma T}}  \left(\sum_{c\in [\Gamma]_T^{\rkone}}\dirac_c-\sum_{c\in [\Gamma]_T^{\prkone}}\dirac_c\right) f \xrightarrow[T\to\infty]{} 0.$$
\end{lem}
\begin{proof}
 This is a consequence of the discussion in \S\ref{sec:r1per pr1conj}. Indeed, let $f\in C_c(SM)$ be non-negative, let $K\subset S\Omega$ be compact such that its projection in $SM$ contains the support of $f$, and let $A\subset \Gamma$ be the finite set of elements that fix a point of $K$. By \eqref{eq:N<Stab}, the number of strongly primitive rank-one conjugacy classes associated to a given rank-one periodic $(g^t_\Gamma)$-orbit intersecting the support of $f$ is less than $\#A$. This implies the first assertion with $C=\#A$.
 
 If $\Gamma'<\Gamma$ is a torsion-free finite-index subgroup, then $\Gamma'\cap\Stab_\Gamma(v)={\id}$, and hence $\#\Stab_\Gamma(v)\leq [\Gamma:\Gamma']$ for any $v\in S\Omega$; thus, by \eqref{eq:N<Stab}, the number of strongly primitive rank-one conjugacy classes associated to any rank-one periodic $(g^t_\Gamma)$-orbit is at most $[\Gamma:\Gamma']$.
 
 By \eqref{eq:N<Stab}, for any $\ell>0$, the number of conjugacy classes of length $\ell$ associated to a given rank-one periodic $(g^t)_t$-orbit intersecting the support of $f$ is less than $\#A$. Therefore, for any $T>0$,
 \[\sum_{c\in[\Gamma]_T^{\prkone}}\int f\diff\dirac_c\leq \sum_{c\in[\Gamma]_T^{\rkone}}\int f\diff\dirac_c \leq \sum_{c\in[\Gamma]_T^{\prkone}}\int f\diff\dirac_c + \#A\cdot\sum_{k\geq 2}\sum_{c\in[\Gamma]_{\frac Tk}^{\prkone}}\int f\diff\dirac_c.\]
 Let $\eps>0$ be such that $\sum_{c\in[\Gamma]_\eps^{\prkone}}\int f\diff\dirac_c=0$. By Theorem~\ref{thm:pcgeod_equidist} and the first part of this lemma, for $T$ large enough, $$\sum_{c\in[\Gamma]_T^{\prkone}}\int f\diff\dirac_c\leq \frac{2\cdot\#A}{\delta_\Gamma T}e^{\delta_\Gamma T}\int f\diff \frac{m_\Gamma}{\Vert m_\Gamma\Vert}.$$ 
 Thus,
 \begin{align*}
 \sum_{c\in[\Gamma]_T^{\rkone}}\int f\diff\dirac_c - \sum_{c\in[\Gamma]_T^{\prkone}}\int f\diff\dirac_c &\leq \#A \cdot \frac T\eps \sum_{c\in[\Gamma]_{\frac T2}^{\prkone}}\int f\diff\dirac_c \\
 & \leq \frac{4(\#A)^2}{\delta_\Gamma T}\frac T\eps e^{\delta_\Gamma T/2}\int f\diff \frac{m_\Gamma}{\Vert m_\Gamma\Vert}.
 \end{align*}
 This concludes the proof of the second point.
\end{proof}

The following implies Proposition~\ref{prop:countconjclassirred} because strong irreducibility implies that the core-fixing subgroup is trivial.

\begin{prop}\label{prop:countingconj}
 Let $\Omega\subset\proj(V)$ be a domain, and $\Gamma<\Aut(\Omega)$ a discrete subgroup with $M=\Omega/\Gamma$  non-elementary rank-one. Consider a Sullivan measure $m_\Gamma$ of dimension $\delta_\Gamma$, and suppose it is finite. Let $\corfix<\Gamma$ be the core-fixing subgroup. Let $K\subset SM_{\bip}$ be the set of vectors whose lifts $\tilde{v}\in S\Omega$ satisfy $\Stab_\Gamma(\tilde{v})\neq \corfix$, and 
 $A\subset \mathcal{G}^{\rkone}$ be the set of rank-one periodic orbits contained in $K$. Then
 \[Te^{-\delta_\Gamma T}\sum_{g\in A_T}\dirac_{g} \xrightarrow[T\to\infty]{} 0\]
 in $C_c^*(T^1M)$, where $A_T$ is the set of orbits in $A$ with period less than $T$.
 
 Suppose further that  $\corfix$ is the center of $\Gamma$. Then 
 \[\frac{\delta_\Gamma T}{\#\corfix}e^{-\delta_\Gamma T}\sum_{c\in [\Gamma]^{\prkone}_T}\dirac_c \xrightarrow[T\to\infty]{} \frac{m_\Gamma}{\Vert m_\Gamma\Vert}.\]
\end{prop}
\begin{proof}
 Let us only give a proof of the first point, since the other is an elementary consequence of it, Remark~\ref{rmk:nbconj} and Theorem~\ref{thm:pcgeod_equidist}.
 
 Consider a non-negative function $f\in  C_c(SM)$. Fix $\epsilon>0$. Observe that $m_\Gamma(K)=0$ since $K\subset SM_{\bip}$ has empty interior by Lemma~\ref{lem:corfix}, and $m_\Gamma$ is ergodic with support $SM_{\bip}$ by Theorem~\ref{thm:HTSR}. Therefore, we can find a non-negative function $\chi\in C_c(SM)$  such that $\chi\geq 1$ on $K$ and $\int \chi f\diff m_\Gamma\leq \epsilon \Vert m_\Gamma\Vert$. Then
 \[\delta_\Gamma Te^{-\delta_\Gamma T}\sum_{c\in A_T}\dirac_c(f)\leq \delta_\Gamma Te^{-\delta_\Gamma T}\sum_{c\in \mathcal{G}_T^{\rkone}}\dirac_c(\chi f)  \xrightarrow[T\to\infty]{} \frac1{\Vert m_\Gamma\Vert}\int \chi f \diff m_\Gamma\leq \epsilon.\]
 This holds for any $\epsilon >0$, so $(\delta_\Gamma Te^{-\delta_\Gamma T}\sum_{c\in A_T}\dirac_c(f))_T$ converges to zero.
\end{proof}

% \subsection{More examples}

Observe that Proposition~\ref{prop:countingconj} applies to Example~\ref{eg:timesZ/2Z}. We present a more general family of examples, to which Proposition~\ref{prop:countingconj} does not (necessarily) apply:

\begin{eg}
 Consider two natural numbers $n,k\geq 1$, a non-elementary discrete subgroup $\Gamma < \orth(n,1)$ (for instance a Schottky group), a finite subgroup $F < \orth(k)$, and a morphism $\rho:\Gamma\rightarrow \orth(k)$ whose image normalizes $F$. Consider the subgroup
 \[\Gamma':=\left\{ c(\gamma,f):=\begin{bmatrix} \gamma & 0 \\ 0 & \rho(\gamma)f \end{bmatrix} : \gamma\in\Gamma, \ f\in F \right\}< \PO(n+k,1).\]
 It preserves the projective model of the real hyperbolic space of dimension $n+k$, and has the following properties.
 \begin{itemize}
     \item An element $c(\gamma,f)$ is (strongly primitive) rank-one if and only if $\gamma$ is (strongly primitive) rank-one;
     \item the limit set of $\Gamma'$ is the image of the limit set of $\Gamma$ under the natural embedding $\proj(\real^{n+1})\hookrightarrow\proj(\real^{n+1+k})$;
     \item if $\Gamma$ is torsion-free or irreducible, then the core-fixing subgroup of $\Gamma'$ consists of the elements of the form $c(\id,f)$ for $f\in F$, and it is not the center of $\Gamma'$ if for instance $F$ is not abelian;
     \item if $\rho$ is trivial and $\Gamma$ is torsion-free, then Remark~\ref{rmk:nbconj} implies that $\nbconj_g$ is the number of conjugacy classes in $F$ for any closed rank-one geodesic $g\subset SM':=S\mathbb H^{n+k}/\Gamma'$  (see Notation~\ref{nota:nbconj});
     \item if $\rho$ is trivial, and $\Gamma$ is torsion-free or irreducible and admits a finite Sullivan measure of dimension $\delta_\Gamma$, then Remark~\ref{rmk:nbconj} and Proposition~\ref{prop:countingconj} imply that
     \[\frac{\delta_\Gamma T}{N}e^{-\delta_\Gamma T}\sum_{c\in [\Gamma]^{\prkone}_T}\dirac_c \underset{T\to\infty}{\longrightarrow} \frac{m_\Gamma}{\Vert m_\Gamma\Vert},\]
     where $N$ is the number of conjugacy classes in $F$.
 \end{itemize} 
\end{eg}

\section{Geometrically finite subgroups} \label{sec:geomfin}

\newcommand{\fundom}{F}
\newcommand{\perifundom}{D}

Among the discrete subgroups $\Gamma \leq \Aut(\Omega)$, where $\Omega$ is a properly and strictly convex domain with $C^1$ boundary, a distinguished class of subgroups with additional topological and geometric structure associated to them is given by the geometrically finite subgroups, i.e. subgroups $\Gamma$ which act geometrically finitely on $\Omega$ in the sense of \cite{CM12}, as described in \S\ref{sub:geomfin}.

A larger class of subgroups, which does not appear to have similarly strong topological and geometric characterisations, is given by the subgroups $\Gamma$ which act geometrically finitely on $\del\Omega$ in the sense of \cite{CM12}, as described in \S\ref{sub:geomfin}.

In this section we show that geometrically finite subgroups (it suffices that $\Gamma$ acts geometrically finitely on $\del\Omega$) admit finite Sullivan measures, and that a strengthened equidistribution result for primitive closed geodesics holds for these subgroups.

{\bf By a ``smooth domain'', we will mean a properly and strictly convex domain with $C^1$ boundary.}

\begin{thm} \label{thm:finite_BMmeas}
If $\Omega \subset \RPn$ is a smooth domain and $\Gamma \actson \del\Omega$ geometrically finitely, then the Sullivan measure $m_\Gamma$ on $SM := S\Omega / \Gamma$ associated to any conformal density $\mu$ of dimension $\delta_\Gamma$ is finite. In particular, $\Gamma$ is divergent, so there is a unique conformal density of dimension $\delta_\Gamma$, up to scaling.
\end{thm}

We remark that a proof of Theorem \ref{thm:finite_BMmeas}, in the case of groups $\Gamma$ acting geometrically finitely on $\Omega$, has previously appeared in Crampon's thesis \cite[Th\,4.3.1]{crampon_these}. We include a self-contained proof here, following the gist of the argument in \cite{crampon_these}, and extend it to groups $\Gamma$ acting geometrically finitely on $\del\Omega$. 
The proofs of Theorem \ref{thm:finite_BMmeas} and of auxiliary results such as Proposition~\ref{prop:noatoms} and their consequences also take inspiration from those of analogous results of Dal'bo--Otal--Peign\'e in \cite{DOP}, which characterize geometrically finite Riemannian manifolds of pinched negative curvature with finite Sullivan measure in terms of Poincar\'e series.

\begin{thm}[{cf.\ \cite[Th.\,5.2]{Roblin}}] \label{thm:pcgeod_equidist_geomfin}
% Let $\Omega\subset\proj(V)$ be a properly and strictly convex open set with $C^1$ boundary, and
%Let $\Omega \subset \RPn$ be a smooth domain, $\Gamma\leq\Aut(\Omega)$ a discrete subgroup acting geometrically finitely on $\del\Omega$, and $m_\Gamma$ the Sullivan measure associated to a $\delta_\Gamma$-conformal density. Then
% Suppose that the Sullivan measure on $S\Omega/\Gamma$ is finite. 
In the setting of Theorem~\ref{thm:finite_BMmeas}, we have
\[ \delta \ell e^{-\delta \ell} \sum_{g \in \mathcal{G}_\Gamma(\ell)} \dirac_g \xrightarrow[\ell \to +\infty]{} \frac{m_\Gamma}{\|m_\Gamma\|} \]
in $C_b(S\Omega / \Gamma)^*$, the dual to 
the space of \emph{bounded} continuous functions on $S\Omega / \Gamma$. 
\end{thm}

By integrating the constant function 1 against both sides, we obtain the following

\begin{cor}\label{cor:countgeomfin}
In the setting of Theorem~\ref{thm:pcgeod_equidist_geomfin}, $\#\mathcal{G}_\Gamma(\ell) \underset{\ell\to\infty}{\widesim} \frac{e^{\delta\ell}}{\delta\ell}$.
% \begin{proof}
% By integrating the constant function 1 against the measure $\delta \ell e^{-\delta \ell} \sum_{g \in \mathcal{G}_\Gamma(\ell)} \dirac_g$ gives $\delta \ell e^{-\delta \ell} \#\mathcal{G}_\Gamma(\ell)$. From Theorem \ref{thm:pcgeod_equidist_geomfin}, this integral converges to 1 as $ \ell \to \infty$. 
% \end{proof}
\end{cor}

\subsection{Discrete parabolic groups}\label{sec:parabgroups}
In this section we establish two useful properties of discrete parabolic groups of automorphisms of smooth domains: 
\begin{prop}\label{prop:parab diverge}
 Let $\Omega\subset \RPn$ be a smooth domain. Then any discrete parabolic subgroup of $\Aut(\Omega)$ is finitely presented and divergent.
\end{prop}

The divergent property had already been established by Crampon--Marquis in the particular case where the parabolic subgroup is conjugate into $\SO(1,\dimension)$, and they gave moreover an explicit formula for the critical exponent:
\begin{lem}[{\cite[Lem.\,9.8]{CM14}}] \label{lem:critexp_para} 
Any parabolic group $P \actson \Omega \subset \RPn$ which is conjugate into $\SO(1,\dimension)$ is divergent, and $\delta_P = \frac{r}2$ where $r$ is the rank of $P$ (\ie $P$ contains $\ints^r$ as a finite-index subgroup).
\end{lem}

The idea to prove Proposition~\ref{prop:parab diverge} more generally is to use the strong structural results given to us by the next proposition to prove that the Zariski closure of our parabolic group is divergent in some sense. 

\begin{prop}[{\cite[Prop.\,7.1 \ \& \ Lem.\,7.6]{CM12}}]\label{prop:zarclosure of parab}
 Let $\Omega\subset \RPn$ be a smooth domain, and $P\leq\Aut(\Omega)$ a discrete parabolic subgroup.
 
 Then $P$ is a cocompact lattice of its Zariski closure $\mathcal{N}$, which is nilpotent and equal to the direct product $K\times U$, where $K\leq \mathcal N$ is Zariski-closed, compact, abelian, and consists of semi-simple elements, $U\leq \mathcal{N}$ is Zariski-closed and unipotent, and the map $(k,u)\mapsto ku$ is an isomorphism from $K\times U$ to $\mathcal N$.
\end{prop}

To make sense of divergence for the Zariski closure, we extend the definitions of critical exponent and divergent subgroups to closed, not necessarily discrete subgroups of $\SL(\Rnplusone)$ which do not necessarily preserve a properly convex open set. (This definition will not be needed anywhere else than this section.) For any element $g\in\SL(\Rnplusone)$, let $\|g\|$ denote the square root of the sum of the squared entries of $g$.

\begin{defn}\label{def:generalised divergent}
 For any closed subgroup $G\leq\SL(\Rnplusone)$, we define the {\bf critical exponent} $\delta_G$ of $G$ as the (possibly infinite) supremum of the numbers $s\geq 0$ such that $\int_G(\|g\|\|g^{-1}\|)^{-\frac s2}\,d\mu_G(g)=\infty$, where $\mu_G$ is any Haar measure on $G$. We say that $G$ is {\bf divergent} if $\delta_G$ is finite and $\int_G(\|g\|\|g^{-1}\|)^{-\frac {\delta_G}2}\,d\mu_G(g)=\infty$.
\end{defn}

Note that Definition~\ref{def:generalised divergent} is compatible with the definition of the critical exponent given in \S\ref{subsec:pat_sul}, thanks to Proposition~\ref{prop:kappa and Hilbert}. We now prove
% that algebraic unipotent groups are divergent.

\begin{lem}\label{lem:unip diverge}
 Any unipotent Zariski-closed subgroup of $\SL(\Rnplusone)$ is divergent.
\end{lem}
\begin{proof}
 Let $U<\SL(\Rnplusone)$ be a Zariski-closed unipotent subgroup; denote by $\mathfrak u$ its Lie algebra. The exponential map $\exp : \mathfrak u\rightarrow U$ is a diffeomorphism such that the entries of $\exp(x)$ are polynomials in the entries of $x\in\mathfrak u$ (by the Baker--Campbell--Hausdorff formula and nilpotence; see \cite[\S 4]{BorelLinAlgGpshort}). Furthermore, the push-forward by $\exp$ of any Lebesgue measure on $\mathfrak u$ (let us fix one) is a Haar measure on $U$ (see \cite[Th.\,1.2.10]{repofnilp}).
 
 Set $P(x):=\left\|\exp\left(x\right)\right\|^2\cdot\left\|\exp\left(-x\right)\right\|^2\geq 1$ for any $x\in\mathfrak u$, and observe that $P$ is a polynomial on $\mathfrak u$, and is proper, in the sense that $P(x)\to\infty$ as $\|x\|\to\infty$. By definition, $\delta_U$ is the supremum of the $s\geq 0$ such that $\int_{\mathfrak u}P^{-s/4}$ diverges, where we integrate against the Lebesgue measure. To conclude the proof, it is enough to show that $\int_{\mathfrak u}P^{-\delta_U/4}$ diverges. This is a consequence of Lemma~\ref{lem:poly diverge} below.
\end{proof}

\begin{lem}\label{lem:poly diverge}
 Let $P$ be a proper polynomial in $\dimension \in \ints_{\geq 1}$ variables, with real coefficients, such that $P\geq 1$ on $\real^\dimension$. Let $\delta = \sup \{s\geq 0 \st \int_{\real^n}P^{-s} d\mu \text{ diverges}\}$, where $\mu$ denotes the Lebesgue measure on $\real^\dimension$. Then $\delta$ is finite and $\int_{\real^\dimension}P^{-\delta} d\mu$ diverges.
\end{lem}
\begin{proof}
 For any $s>0$, we have
 \begin{align*}
     \int_{\real^\dimension}P^{-s}(x)\,d\mu(x) & = \int_{x\in\real^\dimension}\int_{t=0}^{P^{-s}(x)} dt\,d\mu(x) \\
     & = \int_{t=0}^1\mu(P^{-s}\geq t)\,dt \\
     & = s\int_{u\geq 1}\mu(P\leq u)\;u^{-s-1} du
 \end{align*}
 By \cite[Prop.\,7.2]{BenoistOh}, there exist $a>0$, $r\in\rats_{>0}$ and $k\in\nats$ such that $\mu(P\leq u)$ is asymptotically equivalent to $au^r(\log u)^k$ as $u \to\infty$. This concludes the proof since $\int_{u\geq 1}u^{r-s-1} (\log u)^k\,du$ is finite for any $s>r$, and is equal to $\lim_{u\to\infty}\frac{\log(u)^{k+1}}{k+1}=\infty$ for $s=r$.
\end{proof}

We conclude this section with the proof of Proposition~\ref{prop:parab diverge}.

\begin{proof}[Proof of Proposition~\ref{prop:parab diverge}]
 Let $P\leq\Aut(\Omega)$ be a discrete parabolic subgroup. By Proposition~\ref{prop:zarclosure of parab}, the group $P$ is a uniform lattice of its Zariski closure $\mathcal N=K\times U$, where $K$ is compact, and $U$ is Zariski-closed and unipotent. The group $P$ is finitely presented by \cite[Cor.\,6.14]{Raghunathan}; 
 % Raghunathan: Discrete subgroups of Lie groups, Theorem 2.21
 let us prove that it is divergent. The restriction to $P$ of the projection onto $U$ has finite kernel, and its image $P'$ is a uniform lattice of $U$. By Lemma~\ref{lem:unip diverge}, the group $U$ is divergent, and hence so are $P'$ and $P$. Indeed, denote by $\mu$ a Haar measure on $U$ and fix a relatively compact measurable subset $\fundom\subset U$ such that
 $(p,g)\in\ P'\times\fundom \mapsto pg\in U$ is a bijection. Then for any $s\geq 0$:
 \[\int_U (\|g\|\cdot\|g^{-1}\|)^{-s}\,d\mu(g) = \sum_{p\in P'}\int_{\fundom}(\|pg\|\cdot\|pg^{-1}\|)^{-s}\,d\mu(g).\]
 Set $C:=\max\{\|g\|\cdot\|g^{-1}\| \:|\: g\in\fundom\}$, which is finite since $\fundom$ is relatively compact. The norm $\|\cdot\|$ we have chosen is submultiplicative, therefore
 \[C^{-s}\sum_{p\in P'}(\|p\|\cdot\|p^{-1}\|)^{-s} \leq \frac 1{\mu(\fundom)} \int_U (\|g\|\cdot\|g^{-1}\|)^{-s}\,d\mu(g) \leq C^s\sum_{p\in P'}(\|p\|\cdot\|p^{-1}\|)^{-s}.\]
These estimates conclude the proof.
\end{proof}

\subsection{Finiteness properties for boundary geometrically finite subgroups}

The proofs of Theorems \ref{thm:finite_BMmeas} and \ref{thm:pcgeod_equidist_geomfin} for the general case of subgroups $\Gamma$ acting geometrically finitely on $\del\Omega$, but not geometrically finitely on $\Omega$, will require some finiteness results for such subgroups. We will establish these here.

\begin{prop}[{\cite[Prop.\,9.10]{CM12}}]\label{prop:finitely many cusps}
Let $\Omega\subset \RPn$ be a smooth domain and $\Gamma\leq\Aut(\Omega)$ a discrete subgroup acting geometrically finitely on $\del\Omega$. Then there are finitely many $\Gamma$-orbits of parabolic points, $\Gamma$ is hyperbolic relative to its maximal parabolic subgroups, and $\Gamma$ is finitely presented.
\begin{proof}
We can assume that $\Gamma$ is non-elementary. 
% 
% Let us prove that there are finitely many $\Gamma$-orbits of parabolic points. According to Lemma~\ref{lem:disjoint horoballs}, we can choose for each parabolic point $\xi\in \Lambda_\Gamma$ a horoball $\horoball_\xi$ centered at $\xi$ such that $\horoball_{\gamma \xi}=\gamma H_\xi$ for any $\gamma\in\Gamma$, and $H_\xi$ and $H_\eta$ are at distance at least $1$ (for $d_\Omega$) for any other parabolic point $\eta$; choose also a unit tangent vector $v_\xi\in S\Omega$ such that $v^+_\xi=\xi$, $\pi v_\xi\in \partial \horoball_\xi\cap\Omega$ and $v_\xi^-\in\Lambda_\Gamma$.
% 
% Suppose by contradiction that there are infinitely many $\Gamma$-orbits of parabolic points. By Proposition~\ref{prop:noncuspidal is cpct}, we can find a sequence of pairwise distinct parabolic points $(\xi_n)_n$ such that $(\pi v_{\xi_n})_n$ remains in a compact subset of $\Omega$, hence, up to extraction, converges in $\Omega$. This contradicts the fact that $d_\Omega(\pi v_{\xi},\pi v_{\eta})\geq 1$ for all parabolic points $\xi\neq\eta$.
% 
% Let us prove that $\Gamma$ is hyperbolic relative to its maximal parabolic subgroups and finitely presented. 
% 
Since $\Gamma$ is a discrete subgroup of $\Aut(\Omega)$, it is countable. 

By Yaman's criterion \cite{Yaman}, $\Gamma$ is hyperbolic relative to its maximal parabolic subgroups. In particular, it has finitely many classes of maximal parabolic subgroups (see also \cite[Th.\,1B]{Tukia}).
% \todo{Didn't read Yaman carefully enough before ... apparently ``finitely many $\Gamma$-orbits of parabolic points'' follows once every limit point is conical or bounded parabolic, from a theorem of Tukia. (The previous results are still used elsewhere below though, so it wasn't wasted work ... )}

By Proposition \ref{prop:parab diverge}, the maximal parabolic subgroups are finitely presented.
Hence, by \cite[Cor.\,2.4]{Osin} (which states that relatively hyperbolic groups, as defined in \cite{Osin}, inherit finite properties from their peripheral subgroups) and \cite[Th.\,5.1]{Hruska} (which proves the equivalence of several characterizations of countable relatively hyperbolic groups, including the definitions used in \cite{Osin} and in \cite{CM12}), $\Gamma$ is also finitely presented.
\end{proof} \end{prop}

\begin{lem} \label{lem:disjoint horoballs}
Let $\Omega \subset\proj(\real^{\dimension +1})$ be a smooth domain and $\Gamma\leq\Aut(\Omega)$ a discrete subgroup, $\xi,\xi'$ two bounded parabolic points of the proximal limit set, and $\horoball'$ a horoball centered at $\xi'$. Then there exists a horoball $\horoball$ centered at $\xi$ such that for any $\gamma\in\Gamma$, either $\horoball'\cap\gamma \horoball=\varnothing$ or $\gamma \xi=\xi'$.
\end{lem}

\begin{proof}
Suppose by contradiction that we can find a decreasing sequence of horoballs $(\horoball_n)_n$ centered at and converging to $\xi$, and a sequence of elements $(\gamma_n)_{n\in\nats} \subset \Gamma$ such that $\horoball'\cap\gamma_n \horoball_n\neq\varnothing$ and $\gamma_n \xi \neq \xi'$. From Proposition \ref{fact:horofoliation}, the intersection $[\gamma_n\xi\ \xi']\cap \horoball \cap \gamma_n\horoball_n$ is non-empty (and compact in $\Omega$) for all $n$, and we can consider its closest point $x_n$ to $\gamma_n \xi$, which belongs to $\horosphere'$. Similarly, for any $n$, we consider $y_n\in[\xi\ \gamma_n^{-1}\xi']\cap \horosphere \cap \gamma_n^{-1}\horoball'$. 

Since $\xi$ and $\xi'$ are bounded parabolic, up to replacing $(\gamma_n)_n$ by a sequence of the form $(g_n\gamma_nh_n)_n$, where $(g_n)_n\subset\Stab_\Gamma(\xi')$ and $(h_n)_n\subset\Stab_\Gamma(\xi)$, we can assume that $(x_n)_n$ and $(y_n)_n$ stay in a compact subset of $\Omega$; moreover, up to extraction, we can assume that these sequences converge respectively to $x$ and $y$ in $\Omega$. 

Let us prove that $\Stab_\Gamma(\xi)$ is finite, which will contradict the fact that $\xi$ is bounded parabolic, and hence conclude the proof. Fix $\gamma\in\Stab_\Gamma(\xi)$. Observe that $d_\Omega(\gamma z_n,z_n) \to 0$ as $n \to \infty$, where $z_n=\gamma_n^{-1}x_n$. Indeed for each $n$ let $v_n\in S_{y_n}\Omega$ be such that $v_n^+=\xi$, and $t_n=d_\Omega(y_n,z_n)$, so that $\pi g^{t_n} v_n=z_n$. By construction $(t_n)_n$ diverges, hence by Lemma~\ref{lem:crampon}, 
\[\limsup_{n\to\infty}d_\Omega(z_n,\gamma z_n) \leq \lim_{t\to\infty}\limsup_{n\to\infty}d_\Omega(\pi g^{t}v_n,\pi g^{t}\gamma v_n) = \lim_{t\to\infty}d_\Omega(\pi g^t v,\pi g^t\gamma v)=0,\]
where $v\in S_y\Omega$ is such that $v^+=\xi$. As a consequence,
\begin{align*}
d_\Omega(\gamma_n\gamma\gamma_n^{-1}x,x) & \leq d_\Omega(\gamma_n\gamma\gamma_n^{-1}x,\gamma_n\gamma\gamma_n^{-1}x_n) + d_\Omega(\gamma_n\gamma\gamma_n^{-1}x_n,x_n) + d_\Omega(x_n,x) \\
& \leq 2d_\Omega(x_n,x) + d_\Omega(\gamma z_n,z_n) \xrightarrow[n\to\infty]{} 0
\end{align*}
and $(\gamma_n\gamma\gamma_n^{-1}x)_n$ converges to $x$, hence $\gamma_n\gamma\gamma_n^{-1}$ stabilizes $x$ for $n$ large enough by proper discontinuity. Thus $\Stab_\Gamma(\xi)$ is no larger than $\Stab_\Gamma(x)$, which is finite.
\end{proof}

\begin{prop}\label{prop:noncuspidal is cpct}
 Let $\Omega \subset \RPn$ be a smooth domain and $\Gamma\leq\Aut(\Omega)$ a discrete subgroup acting geometrically finitely on $\del\Omega$. For each parabolic point $\xi\in\Lambda_\Gamma$, fix an open horoball $\horoball_\xi$ centered at $\xi$ such that $\horoball_{\gamma\xi}=\gamma \horoball_{\xi}$ for each $\gamma\in\Gamma$. Then the set of unit tangent vectors of $SM_{\bip}$ whose foot-point does not belong to the projection of a horoball in $\{\horoball_\xi \st \xi\in\Lambda_\Gamma \text{ parabolic}\}$ is compact.
\end{prop}

\begin{proof}
Pick $o\in\Omega$ and let $D:=\{x\in\Omega \st d_\Omega(x,\gamma o)\geq d_\Omega(x,o) \ \forall\gamma\in\Gamma\}$ be the Dirichlet domain associated to $o$ and $\Gamma$. It is enough to show that the set 
\[A:=D \cap \bigcup_{\xi,\eta\in\Lambda_\Gamma} (\xi\,\eta) \smallsetminus \bigcup_{\xi \text{ parabolic}} \horoball_\xi\subset\Omega\]
is compact. 

Assume that this is not the case, so that there is a sequence $(x_n)_{n\in\nats} \subset A$ that converges to some $\xi\in\del\Omega$. Observe that $\xi\in\Lambda_\Gamma$ since $(x_n)_n$ is contained in the convex hull of the limit set. $\Gamma$ acts geometrically finitely on $\del\Omega$, so $\xi$ is either conical or bounded parabolic. 

If $\xi$ were conical, there would exists a sequence $(\gamma_k)_{k\in\nats} \subset \Gamma$ such that $(\gamma_k o)_k$ converges to $\xi$ while staying at bounded distance from $[o\,\xi)$, and
\begin{align*}
    \infty & = \lim_{k\to\infty}\beta_\xi(o,\gamma_ko) 
    \leq \lim_{k\to\infty}\lim_{n\to\infty}\beta_{x_n}(o,\gamma_ko) \\
    & = \lim_{k\to\infty}\lim_{n\to\infty}d_\Omega(x_n,o)-d_\Omega(x_n,\gamma_ko)
    \leq 0,
\end{align*}
which is absurd! Thus $\xi$ is bounded parabolic.

By the definition of $A$, we can find sequences $(\xi_n)_n$ and $(\eta_n)_n$ in $\Lambda_\Gamma$ such that $x_n\in (\xi_n \eta_n)$ for each $n$. Since $x_n\not\in \horoball_\xi$ and $\horoball_\xi$ is convex, up to exchanging $\xi_n$ and $\eta_n$ we can assume that $[x_n \eta_n)\cap \horoball_\xi$ is empty for all $n$. Up to subsequence, we can assume that $(\eta_n)_n$ converges to $\eta\in\Lambda_\Gamma$: if $\eta$ were different from $\xi$, then $(\xi\,\eta)$ would intersect $\horoball_\xi$ non-trivially (because $\Omega$ is strictly convex and $\horoball_\xi$ is $C^1$), and thus so would $[x_n\eta_n)$ for $n$ large enough; since $[x_n\eta_n) \cap \horoball_\xi = \varnothing$ for all $n$, hence $\eta=\xi$.

Since $\xi$ is bounded parabolic, we can find a diverging sequence $(\gamma_n)_{n\in\nats} \subset \Gamma$ of parabolic elements fixing $\xi$ such that, up to subsequence, $(\gamma_n\eta_n)_n$ converges to some $\eta'\neq\xi$. Up to subsequence, we can also assume that $(\gamma_n x_n)_n$ converges to some $x\in\overline\Omega$, which is different from $\xi$ since, as before, $[\gamma_nx_n\;  \gamma_n \eta_n) \cap \horoball_\xi = \varnothing$ and so taking the limit as $n\to\infty$, $[x\,\eta') \cap \horoball_\xi$ is also empty. But then
\begin{align*}
    \infty & = \lim_{n\to\infty} d_\Omega(o,\gamma_no)
    \leq 2 \lim_{n\to\infty} \grop{o}{\gamma_n x_n}{\gamma_n o}
    \leq 2 \grop ox\xi <\infty,
\end{align*}
which is a contradiction.
\end{proof}

\subsection{Finiteness of Sullivan measure}

Since the support of the Sullivan measure $m_\Gamma$ outside of the cusp neighborhoods is compact, it suffices to check that the $m_\Gamma$-measure of (the unit tangent bundle over) each cusp neighborhood is finite. 

To obtain estimates in the cusp neighborhoods, it will be useful to have the two lemmas below, the first establishing a gap between the critical exponent $\delta_\Gamma$ and the critical exponent of any parabolic subgroup, and the second showing that the Patterson--Sullivan measures have no atoms.

% \subsubsection{Parabolic gaps}

% In this section, we show there is a gap between the critical exponent of a non-elementary group and that of any parabolic subgroup.

\begin{lem} \label{lem:critgap_para}
Let $\Omega$ be a smooth domain. For any non-elementary discrete subgroup $\Gamma\leq\Aut(\Omega)$ containing a parabolic subgroup $P$, we have $\delta(\Gamma) > \delta(P)$.
\begin{proof}
% (\cite{crampon_these}, Lemma 4.3.2.) 
It follows from the definition of the critical exponent that $\delta_\Gamma(\Omega) \geq \delta_P(\Omega)$, and it suffices to show that the inequality is strict. 
% (\cite{crampon_these}, Lemma 4.3.4.) 
Since $\Gamma$ is non-elementary, we can use a ping-pong argument to find a free product subgroup $\langle h \rangle * P \leq \Gamma$ where $h \in \Gamma$ is a hyperbolic element (up to replacing $P$ by a finite-index subgroup). 
In particular, $\Gamma$ contains all the distinct elements $h p_1 \cdots h p_k$ with $k \geq 1$, $p_i \in P$. Then we have a lower bound for the Poincar\'e series
\begin{align*}
    g_\Gamma(s,x) := \sum_{\gamma\in\Gamma}e^{-s\cdot d_\Omega(x,\gamma x)} & \geq \sum_{k \geq 1} \sum_{p_1,\dots,p_k} e^{-s \cdot d_\Omega(x, hp_1 \cdots hp_k x)} 
\end{align*}
and applying the triangle inequality
\[ d_\Omega(x, hp_1 \cdots hp_k x) \leq \sum_{i=1}^k d_\Omega(x,hx) + d_\Omega (x,p_i x) \]
to the right-hand side we obtain
\begin{align*}
    g_\Gamma(s,x) & \geq \sum_{k \geq 1} \left( e^{-s \cdot d_\Omega(x, h x)} \sum_{p \in P} e^{-s \cdot d_\Omega(x,px)} \right)^k \\
    & = \sum_{k \geq 1} \left( e^{-s \cdot d_\Omega(x, h x)} g_P(s,x)\right)^k
\end{align*}

The Poincar\'e series $g_P(s,x):=\sum_{p\in P}e^{-sd_\Omega(x,p x)}$ converges for any $s > \delta_P$ and diverges at $s = \delta_P$. Hence there exists $s_0 > \delta_P$ such that $e^{-s_0 \cdot d_\Omega(x, h x)}g_P(s_0,x) > 1$, so that $g_\Gamma(s_0,x)$ diverges. Then $\delta_\Gamma(\Omega) \geq s_0 > \delta_P$.
%In particular, $\Gamma$ contains all the distinct elements $h^{m_1} p_1 \cdots h^{m_k} p_k$ with $k \geq 1$, $n_k \in \ints_{\neq 0}$, $p_i \in P \smallsetminus \{\id\}$. Then we have a lower bound for the Poincar\'e series
%\begin{align*}
%    g_\Gamma(s,x) := \sum_{\gamma\in\Gamma}e^{-s\cdot d_\Omega(x,\gamma x)} & \geq \sum_{k \geq 1} \sum_{\substack{m_1, \dots, m_k \\ p_1,\dots,p_k}} e^{-s \cdot d_\Omega(x, h^{m_1}p_1 \cdots h^{m_k}p_k x)} 
%\end{align*}
%and applying the triangle inequality
%\[ d_\Omega(x, h^{m_1}p_1 \cdots h^{m_k}p_k x) \leq \sum_{i=1}^k d_\Omega(x,h^{m_i}x) + d_\Omega (x,p_i x) \]
%to the right-hand side we obtain
%\begin{align*}
%    g_\Gamma(s,x) & \geq \sum_{k \geq 1} \left( \left( \sum_{n \in \ints \smallsetminus \{0\}} e^{-s \cdot d_\Omega(x, h^n x)} \right) \left( \sum_{p \in P \smallsetminus \{\id\}} e^{-s \cdot d_\Omega(x,px)} \right) \right)^k \\
%    & = \sum_{k \geq 1} \left( (g_{\langle h \rangle}(s,x) - 1) (g_P(s,x) - 1) \right)^k
%\end{align*}
%The Poincar\'e series $g_{\langle h \rangle}(s,x):=\sum_{n\in\ints}e^{-sd_\Omega(x,h^nx)}$ converges for any $s > 0$, and $g_P(s,x):=\sum_{p\in P}e^{-sd_\Omega(x,p x)}$ converges for any $s > \delta_P$ and diverges at $s = \delta_P$. Hence there exists $s_0 > \delta_P$ such that $(g_{\langle h \rangle}(s_0,x)-1)(g_P(s_0,x)-1) > 1$, so that $g_\Gamma(s_0,x)$ diverges. Then $\delta_\Gamma(\Omega) \geq s_0 > \delta_P$.
\end{proof} \end{lem}

\begin{prop}[{cf.\ \cite[Prop.\,4.3.5]{crampon_these}}]
\label{prop:noatoms}
Let $\Omega\subset \RPn$ be a smooth domain, and
$\Gamma\leq\Aut(\Omega)$ be a discrete subgroup acting geometrically finitely on $\del\Omega$. 

% Let $(\mu_x)_{x\in\Omega}$ be a Patterson--Sullivan density of dimension $\delta(\Gamma)$, and $x\in\Omega$. 

Then there exists a Patterson--Sullivan density $(\mu_x)_{x\in\Omega}$ of dimension $\delta(\Gamma)$
% $\mu_x(\{\xi\})=0$ if $\xi\in\Lambda_\Gamma$ is conical or bounded parabolic with $\delta(\Stab_\Gamma(\xi))<\delta(\Gamma)$.
% In particular, if $\delta(P)<\delta(\Gamma)$ for any parabolic subgroup $P\subset \Gamma$, then $\mu_x$ 
that has no atoms.

\begin{proof}
{
Fix $x\in\Omega$. Consider a function $j:\real_{\geq 0}\rightarrow\real_{>0}$ such that 
$$g_\Gamma'(s,x):=\sum_{\gamma\in\Gamma}j(d_\Omega(x,\gamma x))e^{-s\cdot d_\Omega(x,\gamma x)}$$
diverges at $s=\delta_\Gamma$, and such that for any $\eps>0$ there exists $T\geq0$ with $j(t+r)\leq e^{\eps r}j(t)$ for all $t\geq T$ and $r\geq 0$. Let $\mu_x$ be an accumulation point as $s$ tends to $\delta_\Gamma$ of $\mu_{x,s}:=\frac1{g_\Gamma'(s,x)}\sum_{\gamma}j(d_\Omega(x,\gamma x))e^{-sd_\Omega(x,\gamma x)}\dirac_{\gamma x}$.
}

We will show that $\mu_x(\{\xi\})=0$ if $\xi\in\Lambda_\Gamma$ is conical or bounded parabolic with $\delta(\Stab_\Gamma(\xi))<\delta(\Gamma)$. This will suffice, since $\delta(P)<\delta(\Gamma)$ for any parabolic subgroup $P< \Gamma$ from Proposition \ref{prop:parab diverge}.

We can use the shadow lemma (Lemma \ref{lem:sullivan_shadow}) to show that $\mu$ has no atoms on the conical limit set. Given a conical limit point $\xi$, we have a sequence of elements $(\gamma_n^{-1}) \subset \Gamma$, a point $x \in \Omega$ and $r>0$ such that $\gamma_n^{-1} x \to \xi$ and $\gamma_n^{-1} x \in B(x_n,r)$ for some $x_n \in [x\,\xi)$. Thus $\xi \in \shadow_r(x,\gamma_n^{-1} x)$ for all $n$, and so 
\begin{equation} 
\mu_x(\{\xi\}) \leq \mu_x(\shadow_r(x,\gamma_n^{-1} x)) \leq C_{x,r} e^{-\delta_\Gamma d_\Omega(x, \gamma_n^{-1} x)} . 
\label{eqn:conical_shadow_approx} \end{equation}
Since $\gamma_n^{-1} \to \infty$ as $n \to \infty$ and $\delta_\Gamma > 0$, $e^{-\delta_\Gamma d_\Omega(x, \gamma_n^{-1} x)} \to 0$ as $n \to \infty$. Hence $\xi$ cannot be an atom.

It then remains to show that $\mu_x(\{\xi_P\})=0$ for any bounded parabolic point $\xi_P$ stabilized by $P < \Gamma$ such that $\delta(P)<\delta(\Gamma)$. 

We have $\mu_x(\{\xi_P\}) \leq \mu_x(V) \leq \limsup_{s \searrow \delta_\Gamma} \mu_{x,s}(V)$ for any open set $V \subset \overline\Omega$ containing $\xi_P$; hence it suffices to find a family  $(V_n)_{n\in\nats}$ of neighborhoods of $\xi_P$ such that $\limsup_{s\searrow\delta_\Gamma} \mu_{x,s}(V_n) \to 0$ as $n \to \infty$.

Take $\xi_0\in\Lambda_\Gamma\smallsetminus\{\xi_P\}$. Without loss of generality, we can assume that $x\in[\xi_0\xi_P]$. By Lemma~\ref{lem:disjoint horoballs}, we can find an open horoball $\horoball$ centered at $\xi_P$ and which contains no point of the orbit $\Gamma\cdot x$. Since $\xi_P$ is bounded parabolic, we can find a compact subset $K$ of $\Lambda_\Gamma \smallsetminus \{\xi_P\}$ such that $P\cdot K=\Lambda_\Gamma\smallsetminus \{\xi_P\}$. Consider the compact set $K':=\{y\in\overline\Omega \st [y\,\xi]\cap H=\varnothing\ \forall \xi\in K\}$, which does not contain $\xi_P$, and observe that $\Gamma\cdot x\subset P\cdot K'$.

Enumerate $P=\{p_1,p_2,\dots\}$. The set $V_n:=\overline\Omega\smallsetminus (p_1 K'\cup\dots\cup p_n K')$ is a neighborhood of $\xi_P$ in $\overline\Omega$ for each $n\geq 1$; moreover $V_n\cap\Gamma\cdot x=\{p_k\gamma x \st k>n, \gamma\in\Gamma'\}$, where $\Gamma':=\{\gamma\in\Gamma \st \gamma x\in K'\}$. Thus
\[ \mu_{x,s}(V_n) \leq 
\frac{1}{g'_\Gamma(s,x)}
\sum_{k>n} \sum_{\gamma \in \Gamma'} j(d_\Omega(x,p_k\gamma x)) e^{-s\cdot d_\Omega(x, p_k \gamma x)} \]
for each $n\geq 1$ and $s>\delta_\Gamma$, where $g'_\Gamma(s,x) := \sum_{\gamma\in\Gamma} j(d_\Omega(\gamma \cdot x, x)) e^{-s \cdot d_\Omega(\gamma\cdot x, x)}$ is the modified Poincar\'e series defined in \S\ref{subsec:pat_sul} (with $o=x$). 

Let us estimate $d_\Omega(x,p_k\gamma x)=d_\Omega(p_k^{-1}x,\gamma x)$ for $\gamma\in\Gamma'$ and for $k$ large (independent of $\gamma$). Take a compact neighborhood $K''$ of $\xi_P$ in $\overline\Omega$ which is disjoint from $K'$. By strict convexity of $\Omega$, we can find $R>0$ such that $(\xi\,\eta) \cap B_\Omega(x,R)$ is non-empty for every $\xi\in K''$ and $\eta\in K'$; in particular, $\grop{x}{\xi}{\eta}\leq R$. Since $(p_nx)_n$ converges to $\xi_P$, there exists $N$ such that $p_nx\in K''$ for every $n\geq N$. As a consequence, for all $k\geq N$ and $\gamma\in\Gamma'$,
\begin{align*} 
d_\Omega(p_k^{-1}x,\gamma x) & = d_\Omega(x,p_kx)+d_\Omega(x,\gamma x) - 2\grop{x}{p_k^{-1}x}{\gamma x} \\
& \geq d_\Omega(x,p_kx)+d_\Omega(x,\gamma x) -2R.
\end{align*}
Therefore we obtain, because $j$ is an increasing function,
\[ \mu_{x,s}(V_n) \leq \frac{e^{2sR'}}{g'_\Gamma(s,x)} \sum_{k> n} e^{-s d_\Omega(x, p_k x)} \sum_{\gamma\in\Gamma'} j\left(d_\Omega(x,p_k x) + d_\Omega(x,\gamma x)\right) e^{-s\cdot d_\Omega(x,\gamma x)}\]
for all $s>\delta_\Gamma$ and $n\geq N$. By assumption, $\eps:=\frac12(\delta(\Gamma)-\delta(P))>0$; by the definition of $j$, there is $C>0$ such that $j(r+t)\leq Ce^{\eps r} j(t)$ for all $r\geq 0$ and all sufficiently large $t > 0$. Hence
\[ \mu_{x,s}(V_n) \leq Ce^{2sR'} \sum_{k> n} e^{-(s-\eps) d_\Omega(x, p_k x)}\]
for any $s>\delta_\Gamma$. Thus
\[\mu_x(\{\xi_P\})\leq \liminf_{n\to\infty} \, \limsup_{s\searrow\delta_\Gamma} \mu_{x,s}(V_n)\leq \liminf_{n\to\infty}Ce^{2\delta_\Gamma R'}\sum_{k>n}e^{-(\delta_\Gamma-\epsilon)d_\Omega(x,p_kx)} = 0.\qedhere\]
\end{proof}
\end{prop}

% We pause to record a corollary of the proof which will be useful further ahead:
% \begin{corn} \label{cor:conical_fullmeas}
% For $\Gamma \actson \del\Omega$ geometrically finitely, the conical limit set has full $\mu_x$-measure for any Patterson--Sullivan measure $\mu_x$.
% \end{corn}
% In particular, by Theorem \ref{thm:HTSR}, $\Gamma$ is of divergence type.

\begin{proof}[Proof of Theorem~\ref{thm:finite_BMmeas}]
By Theorem~\ref{thm:HTSR}, we can assume that $m_\Gamma$ is associated to the specific conformal density we have constructed in Proposition~\ref{prop:noatoms}.

Let us explain why we may assume without loss of generality that $\Gamma$ is torsion-free, and hence that the action of $\Gamma$ on $\Omega$ is free (since it is properly discontinuous). We know from Proposition~\ref{prop:finitely many cusps} that $\Gamma$ is finitely generated. By Selberg's Lemma \cite{Selberg} (see also \cite{Nica_notes}),
% \todo{\cite{Selberg} is the original reference, though it was a minor lemma in that paper which went on to prove rigidity results; other newer accounts may be more enlightening: see e.g. \href{https://arxiv.org/pdf/1306.2385.pdf}{arXiv:1306.2385}, and the bibliographical comments on p. 1 there}, 
we can find a torsion-free, finite-index, normal subgroup $\Gamma'\leq\Gamma$. Let $(\mu_x)_{x\in\Omega}$ be a $\Gamma$-equivariant conformal density of dimension $\delta_\Gamma=\delta_{\Gamma'}$, with associated Sullivan measure $m_\Gamma$ (\resp $m_{\Gamma'}$) on the quotient $S\Omega/\Gamma$ (\resp $S\Omega/\Gamma'$). Let $\pi_\Gamma^{\Gamma'}:S\Omega/\Gamma'\rightarrow S\Omega/\Gamma$ be the natural projection. Observe that 
\[m_\Gamma = \frac1{[\Gamma:\Gamma']}(\pi_\Gamma^{\Gamma'})_*m_{\Gamma'}.\]
Therefore, $m_{\Gamma}$ is finite if and only if $m_{\Gamma'}$ is finite. We assume for the rest of the proof that $\Gamma$ is torsion-free.

Fix a $\Gamma$-invariant family of disjoint horoballs centered at the parabolic points of $\Lambda_\Gamma$. By Proposition~\ref{prop:noncuspidal is cpct}, we have a decomposition of $SM_{\bip}$ into a compact core and a finite number of ``cusp'' neighborhoods, which are the quotients of our fixed horoballs centered at the parabolic points. To prove the theorem, it suffices to show that the associated measure of each cusp neighborhood is finite.

Let $P\subset\Gamma$ be a maximal parabolic subgroup that fixes some $\xi_P \in \Lambda_\Gamma$. Let $\horoball = \horoball_P$ be a horoball fixed by $P$ (\ie centered at $\xi_P$), and $\Ccal = \Ccal_P$ be a 
% locally finite 
\emph{strict} fundamental domain for the action of $P$ on $\horoball$, in the sense that for any $x\in\horoball$, there exists a unique element $p\in P$ such that $px\in\Ccal$.

Since the action of $\Gamma$ on $\partial\Omega$ is geometrically finite, we can choose a relatively compact (measurable) strict fundamental domain $F\subset\Lambda_\Gamma \setminus \{\xi_P\}$ for the action of $P$. Fix $x\in\Omega$. Since $\mu_x$ has no atoms (Proposition \ref{prop:noatoms}), we have 
\[ m_\Gamma(\pi_\Gamma S\horoball) = \sum_{p,q \in P}\, \int_{p F \times qF} e^{2\delta_\Gamma\grop x{\xi^-}{\xi^+}} \mu_x^2(d\xi^- d\xi^+) \ \int_{(\xi^-\xi^+) \cap \mathcal{C}} dt .\]
By using the $\Gamma$-invariance of $\mu$ and the definition of $\mathcal{C}$, we have
\begin{align*}
m_\Gamma(\pi_\Gamma S\horoball) & = \sum_{p,q \in P}\, \int_{F \times p^{-1} q F} e^{2\delta_\Gamma\grop x{\xi^-}{\xi^+}} \mu_x^2(d\eta^- d\eta^+) \ \int_{(\eta^-\eta^+) \cap p^{-1}\mathcal{C}} dt \\
 & = \sum_{p\in P} \int_{F \times p F}\, e^{2\delta_\Gamma\grop x{\xi^-}{\xi^+}} \mu_x^2(d\eta^- d\eta^+) \ \int_{(\eta^-\eta^+) \cap \horoball} dt .
\end{align*} 

From a geometric point of view, any geodesic $(\eta^-\eta^+)$ intersecting $\horoball$ projects to a geodesic on $\Omega/\Gamma$ which makes an incursion into the cusp neighborhood which $\mathcal{C}$ projects to, and the term $\int_{(\eta^-\eta^+) \cap \horoball} dt$ corresponds to the length of this incursion. 

We will now bound the lengths of these incursions using a geometric argument.

Let $U\subset\overline\Omega$ be an open neighborhood of $\xi_P$ such that $[y\,\eta) \cap \horoball$ is non-empty for all $\eta\in F$ and $y\in U$, and set $R:=d_\Omega(x, \horosphere \smallsetminus U)<\infty$. For all $p\in P$ and $(\eta^-,\eta^+)\in \fundom\times p\fundom$, if $(\eta^-\eta^+)\cap \horoball\neq\varnothing$, then there exists $y,z\in \horosphere$ such that $\eta^-,y,z,\eta^+$ are aligned in this order, and by the definition of $U$ we observe that $d_\Omega(x,y) \leq R$ and $d_\Omega(px,z) \leq R$; thus
\[\int_{(\eta^-\eta^+)\cap \horoball}dt = d_\Omega(y,z) \leq  d_\Omega(x,px) + 2R,\text{ and } 0\leq \grop x{\eta^-}{\eta^+}\leq R.\]
As a consequence,
\begin{align*}
 m_\Gamma(\pi_\Gamma S\horoball) \leq e^{2\delta_\Gamma R}\mu_x(F)\ \sum_{p\in P} (d_\Omega(x,px)+2R) \mu_x(pF).
\end{align*}

Since $\overline F$ and $P\cdot x\cup\{\xi_P\}$ are compact and disjoint, and $\Omega$ is strictly convex, we can find $R'>0$ such that $[y\,z]\cap B_\Omega(x,R') \neq \varnothing$ for all $(y,z)\in \overline F\times \left( P\cdot x\cup\{\xi_P\} \right)$; this immediately implies that $F\subset \shadow_{R'}(px,x)$ for any $p\in P$. Therefore, by Lemma~\ref{lem:estimate Busemann behind shadows}, 
\begin{align*}
    \mu_x(pF) = \mu_{p^{-1}x}(F) =\int_{\xi\in F}e^{-\delta_\Gamma\beta_{\xi}(x,p^{-1}x)}d\mu_x(\xi) \leq e^{4\delta_\Gamma R'}e^{-\delta_\Gamma d_\Omega(x,px)}\mu_x(F).
\end{align*}
We assemble these pieces to obtain 
\begin{align*}
 m_\Gamma(\Ccal) = m_\Gamma(\pi_\Gamma S\horoball) \leq e^{2\delta_\Gamma (R+2R')}\mu_x(F)^2 \ \sum_{p\in P} (d_\Omega(x,px)+2R) e^{-\delta_\Gamma d_\Omega(x,px)}.
\end{align*}
The right-hand side is finite since $\delta_\Gamma>\delta_P$ by Lemma~\ref{lem:critgap_para}, so this concludes the proof of Theorem \ref{thm:finite_BMmeas}. \end{proof}

\subsection{Equidistribution of primitive closed geodesics: proof of Theorem \ref{thm:pcgeod_equidist_geomfin}} \label{sub:geomfin_pcgeod_equidist}
% \begin{thm}[cf. \cite{Roblin}, Th\'eor\`eme 5.2] \label{thm:pcgeod_equidist_geomfin}
% Suppose $\Gamma < \Aut(\Omega)$ is a geometrically finite subgroup.

% As $\ell \to +\infty$,
% \[ \delta \ell e^{-\delta \ell} \sum_{g \in \mathcal{G}_\Gamma(\ell)} \dirac_g \to \frac{m_\Gamma}{\|m_\Gamma\|} \]
% weakly in $C_b(S\Omega / \Gamma)^*$.
% \end{thm}

%\begin{proof}[Proof of Theorem \ref{thm:pcgeod_equidist_geomfin}]
Let $\Epsy_\Gamma^L$ denote the measure $\delta L e^{\delta L} \sum_{g \in \mathcal{G}_\Gamma(L)} \dirac_g$ on $S\Omega/\Gamma$ for $L\geq 0$. By Theorem \ref{thm:pcgeod_equidist}, we already know that $\Epsy_\Gamma^L \to \frac{m_\Gamma}{\|m_\Gamma\|}$ weakly in $C_c(S\Omega/\Gamma)^*$ when $L\to+\infty$. We start by replacing $\Epsy_\Gamma^L$ by a nearby measure which will be better adapted to the argument to come, namely
\[ \Em_\Gamma^L := \delta e^{-\delta L} \sum_{g\in\mathcal{G}_\Gamma(L)} \ell(g) \dirac_g .\]
% (we remark that $\ell(g)\dirac_g$ is simply the non-normalized Lebesgue measure along $g$.) 

We may verify, by arguing as in the proof of \cite[Th.\,5.2]{Roblin}, that we still have $\Em_\Gamma^L \to \frac{m_\Gamma}{\|m_\Gamma\|}$ weakly in $C_c(S\Omega/\Gamma)^*$ when 
$L\to+\infty$, and that it suffices to show that 
$\Em_\Gamma^L$ converges weakly to $\frac{m_\Gamma}{\|m_\Gamma\|}$ in $C_b(S\Omega/\Gamma)^*$ as $L\to+\infty$, to obtain the same (desired) conclusion for $\Epsy_\Gamma^L$.

{\bf The rest of the proof consists in demonstrating that $\Em_\Gamma^L$ converges weakly to $\frac{m_\Gamma}{\|m_\Gamma\|}$ in $C_b(S\Omega/\Gamma)^*$ as $L\to+\infty$.} We present this step in more detail since it more intimately involves the Hilbert geometry of the cusps.

Let us fix a $\Gamma$-invariant family of disjoint horoballs centered at the parabolic points of $\Lambda_\Gamma$. By Proposition~\ref{prop:noncuspidal is cpct}, we have a decomposition of $SM_{\bip}$ into a compact core and a finite number of ``cusp'' neighborhoods, which are the quotients of our fixed horoballs centered at the parabolic points. By Theorem~\ref{thm:pcgeod_equidist}, it suffices to show that $\int f d\Em_\Gamma^L \xrightarrow[L\to\infty]{} \int f d \frac{m_\Gamma}{\|m_\Gamma\|}$ for each bounded continuous function $f$ which is supported on a cusp neighborhood.

Fix a parabolic point $\xi\in\Lambda_\Gamma$, $P:=\Stab_\Gamma(\xi)$ and a open horoball $\horoball$ centered at $\xi$ such that $\gamma \horoball \cap \horoball$ is empty for any $\gamma\in\Gamma\smallsetminus P$. For each $r>0$, denote by $\horoball_r\subset \horoball$ the open horoball centered at $\xi$ whose boundary is at distance $r$ from that of $\horoball$. To prove the theorem, it is enough to prove that $$\limsup_{L\to\infty}\Em_\Gamma^L(\pi_\Gamma S\horoball_r) \xrightarrow[r\to\infty]{} 0.$$
% goes to zero as $r$ goes to infinity.

The rest of the argument will resemble a more refined version of the argument in the proof of Theorem \ref{thm:finite_BMmeas}: whereas there we had a finite bound for the measure of the cusps, here we want a bound that goes to zero as $L\to\infty$. 

Let $K\subset\Lambda_\Gamma\smallsetminus\{\xi\}$ be a compact subset such that $P\cdot K=\Lambda_\Gamma\smallsetminus\{\xi\}$. By the definition of $\Em_\Gamma^L$, we have
\begin{align}
\Em^L_\Gamma(\pi_\Gamma S\horoball_r) 
& \leq \delta e^{-\delta L} \sum_{\substack{\gamma\in\Gamma^{\mathrm{ph}}\\\ell(\gamma)\leq L,\, \gamma^-\in K}} \Leb_\gamma(S\horoball_r) \nonumber\\
& \leq \delta e^{-\delta L} \sum_{p \in P} \sum_{\gamma \in \Gamma(L,p)} \Leb_\gamma(S\horoball_r),
\label{eqn:52_decomp_horos} \end{align}
where $\Gamma^{\mathrm{ph}}\subset\Gamma$ consists of the strongly primitive hyperbolic elements (\ie $\Gamma^{\mathrm{ph}}=\Gamma^{\mathrm{pr1}}$), and, for $p\in P$, the subset $\Gamma(L, p)\subset\Gamma^{\mathrm{ph}}$ consists of the strongly primitive hyperbolic elements $\gamma$ such that $\ell(\gamma) \leq L$ and $(x_\gamma^-,x_\gamma^+)\in K\times pK$.
%the subset $\Gamma(L, p)\subset\Gamma_{hp}$ consists of the primitive hyperbolic elements $\gamma$ such that $\ell(\gamma) \leq L$ and $(\gamma^-,\gamma^+)\in K\times pK$.

{\bf We now fix $r>0$ and $p\in P$, and bound from above $\sum_{\gamma\in\Gamma(L,p)} \Leb_\gamma(S\horoball_r)$.} In particular, we will bound from above the cardinality of the set $\Gamma(L,p,r)$ of $\gamma \in \Gamma(L,p)$ such that $\Leb_\gamma(S\horoball_r) > 0$, i.e. such that the axis $(\gamma^-\gamma^+)$ intersects $\horoball_r$.

Fix $\gamma \in \Gamma(L,p,r)$. Let $a,d\in \horosphere$ and $b,c\in\horosphere_r$ be such that $\gamma^-,a,b,c,d,\gamma^+$ are aligned along $(\gamma^- \gamma^+)$ in this order. By definition, $a$ belongs to the closed subset $A\subset\overline\Omega$ of points $y$ for which there exists $\eta\in K$ with $(\eta\,y]\cap \horoball=\varnothing$.  The set $A\cap \horosphere \subset\Omega$ is compact, hence $d_\Omega(x,a)\leq R_1:=\max\{d_\Omega(x,y) \st y\in\horosphere\cap A\} < \infty$. As a first consequence, $d_\Omega(x,\gamma x)\leq L+2R_1\leq N:=\lceil L+2R_1\rceil$. Moreover, $p^{-1}d\in A$ and $d_\Omega(px,d)\leq R_1$, therefore
\begin{equation}\label{eq:cuspexcursion} \Leb_\gamma(S\horoball_r) = d_\Omega(b,c) = d_\Omega(a,d) - d_\Omega(a,b) - d_\Omega(c,d) \leq d_\Omega(x,p x) + 2R_1 - 2r. \end{equation}
Note, in particular, that $d_\Omega(x,p x) \geq 2r - 2R_1$.

According to the shadow lemma (Lemma~\ref{lem:sullivan_shadow}), we can find $R_2>0$ such that for any $R \geq R_2$, there exists $C_R > 0$ so that for any $g \in \Gamma$ we have
\[C_R^{-1}e^{-\delta\cdot d_\Omega(x,gx)}\leq  \mu_x(\shadow_R(x,gx))\leq \mu_x(\shadow_R^+(x,gx))\leq C_Re^{-\delta\cdot d_\Omega(x,gx)}.\]

Since $\gamma \horoball \cap \horoball = \varnothing$ by our definition of $\horoball$, we have $\gamma a\in[d\,\gamma^+]$, and hence by \eqref{eq:shad+ in shad}
\begin{align*}
\shadow_{R_2}(x,\gamma x) & \subset \shadow_{R_2+2R_1}^+(a,\gamma a)\subset \shadow_{2R_2+4R_1}(a,\gamma a)\subset \shadow_{2R_2+4R_1}(a,d) \\
& \subset \shadow_{R_3}^+(x,px),
\end{align*}
where $R_3:=2R_2+4R_1$. We combine all of these observations to produce:
\begin{align*}
    \#\Gamma(L,p,r) & = \sum_{0\leq n\leq N}  \#\{\gamma\in\Gamma(L,p,r) : n-1<d_\Omega(x,\gamma x)\leq n\} \\
    & \leq \sum_{0\leq n\leq N} \sum_{\substack{\gamma\in\Gamma(L,p,r) :\\ n-1<d_\Omega(x,\gamma x)\leq n}} \!\!\!\!\! C_{R_2}e^{\delta n} \mu_x(\shadow_{R_2}(x,\gamma x)) \\
    & \leq C_{R_2} \sum_{0\leq n\leq N}e^{\delta n}\int_{\xi\in\shadow^+_{R_3}(x,px)} \!\!\! \sum_{\substack{\gamma\in\Gamma(L,p,r) :\\ n-1<d_\Omega(x,\gamma x)\leq n}} \!\!\!\!\!\!\! \mathbf{1}_{\shadow_{R_2}(x,\gamma x)}(\xi) d \mu_x(\xi) \\
    \mbox{(using \eqref{Equation : intersection d ombres})}
    & \leq C_{R_2} \cdot \#\{g\in \Gamma: d_\Omega(x,gx)\leq 4R_2+1\} \sum_{0\leq n\leq N}e^{\delta n}\mu_x(\shadow_{R_3}^+(x,px))
    \\
    & \leq C_{R_2}\cdot \#\{g\in \Gamma: d_\Omega(x,gx)\leq 4R_2+1\} \frac{e^{\delta(N+1)}}{e^\delta-1} C_{R_3}e^{-\delta\cdot d_\Omega(x,px)} \\
    & \leq C e^{\delta(L-d_\Omega(x,px))},
\end{align*}
where $C:=C_{R_2}C_{R_3}\frac{e^{\delta(2R_1+1)}}{e^\delta - 1}\cdot \#\{g\in \Gamma: d_\Omega(x,gx)\leq 4R_2+1\}$.

Combining this with \eqref{eqn:52_decomp_horos} and \eqref{eq:cuspexcursion} yields 
\[ \Em^L_\Gamma(\pi_\Gamma S\horoball_r) \leq \delta {C} \!\!\! \sum_{\substack{p\in P \\ d_\Omega(x,p x) > 2r-2R_1}} (d_\Omega(x,p x) - 2r + 2R_1) e^{-\delta\cdot d_\Omega(x,p x)}. \]
By Lemma~\ref{lem:critgap_para}, $\sum_{p\in P} d_\Omega(x,p x) e^{-\delta\cdot d_\Omega(x,p x)}$ converges. Therefore,
\begin{align*}
\limsup_{L\to+\infty} \Em^L_\Gamma(\pi_\Gamma S\horoball_r) & \xrightarrow[r\to\infty]{} 0.
\qed
\end{align*}
%\end{proof}

\subsection{The number of conjugacy classes in the geometrically finite case}\label{sec:nb per conj geomfin}

As in \S\ref{sec:perconj}, we can relate the number of strongly primitive rank-one conjugacy classes, the number of rank-one conjugacy classes, and the number of rank-one periodic $(g^t_\Gamma)$-orbits; using Theorem~\ref{thm:pcgeod_equidist_geomfin} instead of Theorem~\ref{thm:pcgeod_equidist}, we can extend these results to the geometrically finite case.

We recall that for any given conjugacy class $c$, $\dirac_c$ denotes the flow-invariant probability measure on the closed orbit associated to $c$.

\begin{lem}\label{lem:nb per conjgeomfin}
 Let $\Omega\subset\proj(V)$ be a smooth domain, and $\Gamma\leq\Aut(\Omega)$ a discrete subgroup which acts geometrically finitely on $\partial\Omega$. Let $\Gamma'\leq\Gamma$ be a torsion-free finite-index subgroup. Then 
 \[\sum_{c\in\mathcal{G}_T^{\rkone}}\int f\diff\dirac_c \leq \sum_{c\in[\Gamma]_T^{\prkone}}\int f\diff \dirac_c\leq [\Gamma':\Gamma]\sum_{c\in\mathcal{G}_T^{\rkone}}\int f\diff\dirac_c\]
 for any $T>0$ and any non-negative function $f \in C_b(T^1M)$, and
 $${Te^{-\delta_\Gamma T}} \left(\#[\Gamma]_T^{\rkone}-\#[\Gamma]_T^{\prkone}\right) \xrightarrow[T\to\infty]{} 0.$$
\end{lem}
\begin{proof}
 The proof is the same as that of Lemma~\ref{lem:nb per conjrob}, except that we use Theorem~\ref{thm:pcgeod_equidist_geomfin} instead of Theorem~\ref{thm:pcgeod_equidist}.
%  As in the proof of Observation~\ref{obs:nb per conjgeomfin}, one can apply Observation~\ref{obs:nbconj} to obtain that the number of strongly primitive rank-one conjugacy classes associated to a rank-one periodic $(g^t)_t$-orbit is less than $[\Gamma':\Gamma]$. This implies the first assertion.
%
%  We also saw that, for any $\ell>0$, the number of conjugacy classes of length $\ell$ associated to a rank-one periodic $(\phi_t)_t$-orbit is less than $[\Gamma':\Gamma]$. Therefore, for any $T>0$,
%  \[\#[\Gamma]_T^{\prkone}\leq \#[\Gamma]_T^{\rkone} \leq \#[\Gamma]_T^{\prkone} + [\Gamma':\Gamma]\sum_{k\geq 2}\#[\Gamma]_{\frac Tk}^{\prkone}.\]
%  Let $\eps>0$ be such that $[\Gamma]_\eps^{\rkone}$ is empty. Then for $T$ large enough, $\#[\Gamma]_T^{\prkone}\leq \frac{2[\Gamma':\Gamma]}{\delta_\Gamma T}e^{\delta_\Gamma T}$, and
%  \begin{align*}
%  \#[\Gamma]_T^{\rkone} - \#[\Gamma]_T^{\prkone} &\leq [\Gamma':\Gamma] \frac T\eps \#[\Gamma]_{\frac T2}^{\prkone} \\
%  & \leq \frac{4[\Gamma':\Gamma]^2}{\delta_\Gamma T}\frac T\eps e^{\delta_\Gamma T/2}
%  \end{align*}
\end{proof}

\begin{prop}
 Let $\Omega\subset\proj(V)$ be a smooth domain, and $\Gamma\leq\Aut(\Omega)$ a discrete subgroup which acts geometrically finitely on $\partial\Omega$.  Let $\corfix\leq\Gamma$ be the core-fixing subgroup (see Definition~\ref{defn:corfix}). Let $K\subset T^1M_{\bip}$ be the set vectors of whose lifts $v\in T^1\Omega$ satisfy $\Stab_\Gamma(v)\neq \corfix$. 
 Let $A\subset \mathcal{G}^{\rkone}$ be the set of rank-one periodic orbits contained in $K$. Then
 \[Te^{-\delta_\Gamma T}\cdot \#A_T \xrightarrow[T\to\infty]{} 0,\]
 where $A_T$ is the set of conjugacy classes in $A$ with translation length less than $T$.
% Suppose further that  $\corfix$ is the center of $\Gamma$. Then for any $f \in C_b(T^1 M)$,
% \[\frac{\delta_\Gamma T}{\#\corfix}e^{-\delta_\Gamma T}\sum_{c\in [\Gamma]^{\prkone}_T}\int_{T^1M} f\diff\dirac_c \xrightarrow[T\to\infty]{} \int_{T^1M} f\diff\frac{m_\Gamma}{\Vert m_\Gamma\Vert}.\]
 If $\Gamma$ is strongly irreducible, then $\corfix$ is trivial and 
 \[{\delta_\Gamma T}e^{-\delta_\Gamma T}\sum_{c\in [\Gamma]^{\prkone}_T}\int_{T^1M} f\diff\dirac_c \xrightarrow[T\to\infty]{} \int_{T^1M} f\diff\frac{m_\Gamma}{\Vert m_\Gamma\Vert}.\]
\end{prop}
\begin{proof}
 The proof is the same as that of Proposition~\ref{prop:countingconj}, except that we use Theorem~\ref{thm:pcgeod_equidist_geomfin} instead of Theorem~\ref{thm:pcgeod_equidist}.\qedhere
%
%  Let us only give a proof of the first point, since the other are elementary consequences of it and of the discussion in Section~\ref{sec:r1per pr1conj}.
 
%  Consider a non-negative bounded function $f\in C(T^1M)$. Fix $\epsilon>0$. Since $m_\Gamma(K)=0$ (recall that $K\subset T^1M_{\bip}$ has empty interior by Observation~\ref{obs:corfix}, and $m_\Gamma$ is ergodic with support $T^1M_{\bip}$ by Theorems~\ref{thm:HTSR} and Proposition~\ref{prop:cvxcocpct imp dvgt}), we can find a non-negative bounded function $\chi\in\Ccal(T^1M)$  such that $\chi\geq 1$ on $\supp(f)\cap K$ and $\int \chi f\diff m_\Gamma\leq \epsilon \Vert m_\Gamma\Vert$. According to Theorem~\ref{THMPROOFOFTHM},
%  \[\delta_\Gamma Te^{-\delta_\Gamma T}\sum_{c\in A_T}\int f\diff\mathcal{D}_{c}\leq \delta_\Gamma Te^{-\delta_\Gamma T}\sum_{c\in \mathcal{G}_T^{\rkone}}\int \chi f\diff\mathcal{D}_{c}  \underset{T\to\infty}{\longrightarrow} \frac1{\Vert m_\Gamma\Vert}\int \chi f \diff m_\Gamma\leq \epsilon.\]
%  This holds for any $\epsilon >0$, so $(\delta_\Gamma Te^{-\delta_\Gamma T}\sum_{c\in A_T}\int f\diff\mathcal{D}_{c})_T$ converges to zero.
\end{proof}

% \[a_{a_{a_{a_{a_{a_{a_{a_{a_{a_{a_{a_{a_{a_{a_{a_{a_{a_{a_{a_{a_{a_{a_{a_{a_{a_{a_{a_{a_{a_{a_{a_{a_{a_{a_{a_{a_{a_{a_{a_{a_{a_{a_{a_{a_{a_{a_{a_{a_{a_{a_{a_{a_{a_{a_{a_{a_{a_{a_{a_{a_{a_{a_{a_{a_{a_{a_{a_{a_{a_{a_{a_{a_{a_{a_{a_{a_{a_{a_{a_{a_{a_{a_{a_{a_{a_{a_{a_{a_{a_{a_{a_{a_{a_{a_{a_{a_{a_{a_{a_{a_{a_{a_{a_{a_{a_{a_{a_{a_{a_{a_{a_{a_{a_{a_{a_{a_{a_{a_{a_{a_{a_{a_{}a_{}}}}}}}}}}}}}}}}}}}}}}}}}}}}}}}}}}}}}}}}}}}}}}}}}}}}}}}}}}}}}}}}}}}}}}}}}}}}}}}}}}}}}}}}}}}}}}}}}}}}}}}}}}}}}}}}}}}}}}}}}}}\]
Integrating a constant function against both sides of the last statement, we obtain
\begin{cor}
Let $\Omega\subset\proj(V)$ be a smooth domain, and $\Gamma\leq\Aut(\Omega)$ an irreducible discrete subgroup which acts geometrically finitely on $\partial\Omega$. Then 
\[ \# [\Gamma]^{\prkone}_T \underset{T\to\infty}{\widesim} \frac{e^{-\delta_\Gamma T}}{\delta_\Gamma T} .\] 
\end{cor}

\printbibliography

\end{document}